\documentclass{amsart}

\usepackage{amsmath,amssymb}
\usepackage{enumerate}
\usepackage{xcolor}
\usepackage{graphicx}
\usepackage{hyperref}
\usepackage{soul}

\newtheorem{theorem}{Theorem}[section]
\newtheorem{corollary}[theorem]{Corollary}
\newtheorem{lemma}{Lemma}[section]
\newtheorem{proposition}{Proposition}[section]
\theoremstyle{remark}
\newtheorem{remark}{Remark}[section]

\newtheorem{question}{Question}[section]
\theoremstyle{definition}
\newtheorem{definition}{Definition}[section]

\def\N{{\mathbb{N}}}
\def\Z{{\mathbb{Z}}}

\def\MN{{\mathbb{N}}}

\def\MA{{\mathbb{A}}}
\def\MB{{\mathbb{B}}}
\def\MC{{\mathbb{C}}}

\newcommand{\nsN}{\widetilde{\mathbb N}}
\newcommand{\nsZ}{\widetilde{\mathbb Z}}

\newcommand{\nsB}{\widetilde{\mathbb B}}

\newcommand{\nstdN}{\widetilde{\mathbb N}}
\newcommand{\nstdM}{\widetilde{\mathbb N}}
\newcommand{\ndots}{\ldots}
\newcommand{\List}{\mathrm{List}}
\newcommand{\Irr}{{\mathrm{Irr}}}
\newcommand{\Inv}{{\mathrm{Inv}}}
\newcommand{\Th}{{\mathop{\mathrm{Th}}}}
\newcommand{\nonstandardmodel}[2]{#1\circledast #2}

\def\N{{\mathbb{N}}}
\def\Z{{\mathbb{Z}}}

\def\MA{{\mathbb{A}}}
\def\MB{{\mathbb{B}}}
\def\MC{{\mathbb{C}}}

\begin{document}

\title[Nonstandard polynomials]{Nonstandard polynomials: algebraic properties and elementary equivalence}

\author{Alexei Myasnikov}
\email{amiasnik@stevens.edu}
\author{Andrey Nikolaev}
\email{anikolae@stevens.edu}


\begin{abstract}
We solve the first-order classification problem for rings $R$ of polynomials $F[x_1, \ldots,x_n]$ and Laurent polynomials $F[x_1,x_1^{-1}, \ldots,x_n,x_n^{-1}]$  with coefficients in an infinite field $F$ or the ring of integers $\mathbb Z$, that is, we describe the algebraic structure of all rings $S$ that are first-order equivalent to $R$. Our approach is based on a new and very powerful method of regular bi-interpretations, or more precisely, regular invertible interpretations. Namely, we prove that $F[x_1, \ldots,x_n]$ and $F[x_1,x_1^{-1}, \ldots,x_n,x_n^{-1}]$ are regularly bi-interpretable with the list superstructure $\mathbb S(F,\mathbb N)$ of $F$, which is equivalent to regular bi-interpretation with the superstructure $HF(F)$ of hereditary finite sets over $F$. The expressive power of $\mathbb S(F,\mathbb N)$ is the same as that of the weak second-order logic over $F$. Hence, the first-order logic in $R = F[x_1, \ldots,x_n]$ or $R = F[x_1,x_1^{-1}, \ldots,x_n,x_n^{-1}]$ is equivalent to the weak second-order logic in $F$ (following the terminology of~\cite{KharlampovichMyasnikovSohrabi:2021}, such structures are necessarily \emph{rich}), which allows one to describe the algebraic structure of all rings $S$ with $S\equiv R$. In fact, these rings $S$ are precisely the ``non-standard'' models of $R$, like in non-standard arithmetic or non-standard analysis. This is particularly straightforward when $F$ is regularly bi-interpretable with $\mathbb N$, in this case the ring $R$ is also bi-interpretable with $\mathbb N$. Using our approach, we describe various, sometimes rather surprising, algebraic and model-theoretic properties of the non-standard models of $R$.
\end{abstract}


\maketitle
\tableofcontents
\section{Introduction}

In this paper we solve the first-order classification problem for rings $R$ of polynomials $F[x_1, \ldots,x_n]$ (Laurent polynomials $F[x_1,x_1^{-1}, \ldots,x_n,x_n^{-1}]$ ) with coefficients in a field $F$ or the ring of integers $\mathbb Z$, that is, we describe the algebraic structure of all rings $S$ that are first-order equivalent to $R$.
Our approach is based on a new and very powerful method of regular bi-interpretations (or more generally, regular invertible interpretations).
Namely, we prove that $F[x_1, \ldots,x_n]$ and $F[x_1,x_1^{-1}, \ldots,x_n,x_n^{-1}]$ are regularly bi-interpretable with the list superstructure $\mathbb S(F,\mathbb N)$ of $F$, which is equivalent to regular bi-interpretation with the superstructure $HF(F)$ of hereditary finite sets over $F$.
The expressive power of $\mathbb S(F,\mathbb N)$, as well as $HF(F)$, is very strong, it is the same as the weak second-order logic over $F$. Hence, the first-order logic in $R = F[x_1, \ldots,x_n]$ or $R = F[x_1,x_1^{-1}, \ldots,x_n,x_n^{-1}]$ is equivalent to the weak second-order logic in $F$.
Following~\cite{KharlampovichMyasnikovSohrabi:2021} such structures are termed \emph{rich}.
This allows one to describe the algebraic structure of all rings $S$ with $S\equiv R$.
In fact, these rings $S$ are precisely the ``non-standard'' models of $R$, like in non-standard arithmetic or non-standard analysis.
This is particularly straightforward when $F$ is regularly bi-interpretable with $\mathbb N$, in this case the ring $R$ is also bi-interpretable with $\mathbb N$.
If $R = \Gamma(\mathbb N)$ is an interpretation of $R$ in $\N$ then any ring $\widetilde{R} \equiv R$ is of the form $\Gamma(\nstdN)$ for some $\nstdN \equiv \N$ and the interpretation $\Gamma$ gives the structure of $\widetilde{R}$ over $\nstdN$.
The ring $R \simeq \Gamma(\N)$ is an elementary submodel of $\widetilde{R} = \Gamma(\nstdN)$, in fact, it is the set of standard elements of $\widetilde{R}$.
Following non-standard arithmetic one can develop the theory of non-standard polynomials, which is interesting in its own right.
Using this approach, we describe various, sometimes rather surprising, algebraic and model-theoretic properties of non-standard polynomial rings $\widetilde{R}$.

In fact, interpretation and bi-interpretation offer a general approach to studying structures elementarily equivalent to a given one. Namely, suppose that a structure $\mathbb A$ is regularly (in particular, absolutely) interpretable in a structure $\mathbb B$ as $\mathbb A\cong\Gamma(\mathbb B)$.
In this case, if a structure $\mathbb B'$ is elementarily equivalent to $\mathbb B$, then $\mathbb A'\cong \Gamma(\mathbb B')$ is elementarily equivalent to $\mathbb A$. 
Moreover, if a regular interpetation is \emph{regularly invertible} (which is a certain weak form of regular bi-interpretability, see Definition~\ref{de:reg_invert}), the converse to the latter statement is also true.
That is, a structure $\mathbb A'$ is elementarily equivalent to $\mathbb A$ if and only if $\mathbb A'\cong\Gamma(\mathbb B')$ for some $\mathbb B'$ elementarily equivalent to $\mathbb B$ (Theorem~\ref{th:complete-scheme}).
We refer to Section~\ref{se:prelims} for definitions, and to papers ~\cite{KharlampovichMyasnikovSohrabi:2021,Daniyarova-Myasnikov:I} for a more detailed exposition.

Note that the decidability of the elementary theory $Th(R)$ of rings of polynomials $R = F[x_1, \ldots,x_n]$ has been well studied.
In \cite{Robinson:1951} Robinson showed how to interpret $\N$ in $R$, thus proving that $Th(R)$ is undecidable.
In \cite{Denef:1978} Denef proved that the Diophantine problem (the generalized 10th Hilbert problem) in $R$ is undecidable for fields $F$ of characteristic zero.
Later Bauval showed in \cite{Bauval:1985} that $HF(F)$ is bi-interpretable with $R$.
In \cite{Kharlampovich_Myasnikov:1998(1)} Kharlampovich and Myasnikov studied model theory of free associative algebras over fields, they proved that the rings $R = F[X]$, and also the free associative algebras $A_F(X)$ over an infinite field $F$ and finite basis $X$, are regulalrly bi-interpretable with the structure $HF(F)$, as well as $\mathbb S(F,\mathbb N)$.
In particular, this implies that for fields $F$ and $L$ and finite sets $X$ and $Y$ one has $F[X] \equiv L[Y]$ if and only if $|X| = |Y|$ and $HF(F) \equiv HF(L)$.
The weak second-order logic is quite powerful, so for many fields $HF(F) \equiv HF(L)$ implies $F \simeq L$.
This is true for ${\mathbb Q}$, finitely generated algebraic extensions of ${\mathbb Q}$, algebraically closed fields of finite transcendence degree over their prime subfields, pure transcendental finite extensions of a prime field, the field of algebraic real numbers, etc.
The proofs are based on the Gandy theorem on fixed points of $\Sigma$-definable operators (see, for example,~\cite{Ershov:1996} and~\cite{Barwise:1992}).
The list superstructures as a convenient alternative of HF-superstructures were introduced by S.~Goncharov and D.~Sviridenko~\cite{Goncharov-Sviridenko:1985,Goncharov-Sviridenko:1989} for the purposes of $\Sigma$-programming (see also \cite{Ashaev-Belyaev-Myasnikov:1993}).
However, they became useful tools in studying the weak second order logic in algebra.
For better understanding of non-standard models of the rings of polynomials $R$ one needs to understand the non-standard models of the list superstructures $\mathbb S(F,\mathbb N)$, that is the structures $\mathcal M$ such that $\mathcal M \equiv \mathbb S(F,\mathbb N)$.
Equivalently, it would suffice to understand structures $\mathcal M$ such that $\mathcal M \equiv HF(F)$.
Notice that in both multi-sorted structures $\mathbb S(F,\mathbb N)$ and $HF(F)$ one sort is the field $F$, so the same sort in $\mathcal M$ defines a field $F_\mathcal M$ such that $F \equiv F_\mathcal M$.
However, not every field $F'$ such that $F' \equiv F$ can be realized as $F_\mathcal M$ (see Remark~\ref{re:any_field}).
It seems that equivalence  $\mathcal M \equiv \mathbb S(F,\mathbb N)$ (or $\mathcal M \equiv HF(F)$) gives an equivalence of fields $F$ and $F_\mathcal M$ in some stronger logic than the first-order logic of fields. It is an interesting question what this logic is.

\subsection{Organization of this paper} In Section~\ref{se:prelims}, we recall the necessary model-theoretic notions and introduce key ideas of our study. In particular, we review the notion of bi-interpretation and explicate its role in investigating questions related to first-order equivalence (Section~\ref{se:inv_int_and_bi-int}).
Further, we recall the organization of nonstandard models of arithmetic (Section~\ref{se:nonstandard_arithmetic}), and the notion of a list superstructure (Section~\ref{se:superstructure}).
This naturally leads us to introduce nonstandard list superstructure and its elements, called nonstandard tuples, on which we expound in Section~\ref{se:tuples}.

In Sections~\ref{se:polynomials} and~\ref{se:laurent}, we exploit the introduced techniques to study nonstandard models of the ring of polynomials and Laurent polynomials (respectively) over a field.
For each of the above rings, the exposition follows the same general layout.
We establish a bi-interpretation of the ring in question with the list superstructure over the respective field $\mathbb S(F,\mathbb N)$ (for example, Theorem~\ref{th:Fx_SFN_biint}), which leads to a Tarski-type result describing elementarily equivalent rings (for example, Corollary~\ref{co:Fx_elem_equiv_list}).
Under the additional assumption that the field is bi-interpretable with $\mathbb N$, we additionally produce a bi-interpretation of the ring in question with $\mathbb N$ and the respective Tarski-type result (for example, Corollaries~\ref{co:bi-int_Fx_N},~\ref{co:Fx_elem_equiv}).
Then we establish an algebraic description of the elementarily equivalent rings (for example, Section~\ref{se:single_var}, in particular, formulas~\eqref{eq:nstd_poly_desc}, \eqref{eq:nstd_poly_sum_product}); finally, we list notable algebraic properties (for example, Lemmas and Propositions of Section~\ref{se:single_var}).

In Section~\ref{se:over_Z} we study the case of polynomials and Laurent polynomials over integers.

\section{Interpretability}\label{se:prelims}
In this section, we discuss various types of interpretability of structures; some of them are well-known, others are new. Each of them serves different purposes. Our focus is mostly on the regular interpretability, which suits well to the classical first-order classification problem. 

Let $L$ be a language (signature) with a set of functional (operational) symbols $\{f, \ldots\}$ together with their arities $n_f \in \N$, a set of constant symbols $\{c, \ldots\}$ and a set of relation (or predicate) symbols $\{R, \ldots \}$ coming together with their arities $n_R \in \N$. We write $f(x_1, \ldots,x_n)$ or $R(x_1,\ldots,x_n)$ to show that $n_f = n$ or $n_R=n$. The standard language of groups $\{\cdot\,,\,^{-1},e\}$ includes the symbol $\cdot$ for binary operation of multiplication, the symbol $^{-1}$ for unary operation of inversion, and the symbol $e$ for the group unit; the standard language of unitary rings is $\{+,\,\cdot\,,0,1\}$. An interpretation of a constant symbol $c$ in a set $A$ is an element $c^A\in A$. For a functional symbol $f$ an interpretation in $A$ is a function $f^A\colon A^{n_f}\to A$, and for a predicate $R$ it is a set $R^A\subseteq A^{n_R}$.

An algebraic structure in the language $L$ (an $L$-structure) with the base set $A$ is denoted by $\MA = \langle A; L\rangle$, or simply by $\MA = \langle A; f, \ldots,R,\ldots,c, \ldots \rangle$, or by $\MA = \langle A; f^A, \ldots,R^A,\ldots,c^A, \ldots \rangle$. For a given structure $\MA$ by $L(\MA)$ we denote the language of $\MA$. 

We usually denote variables by small letters $x,y,z, a,b, u,v, \ldots$, while the same symbols with bars $\bar x, \bar y, \ldots$ denote tuples of the corresponding variables, say $\bar x = (x_1, \ldots,x_n)$, and furthermore, $\bar{\bar x}$ is a tuple of tuples $\bar{\bar x}=(\bar x_1,\ldots,\bar x_m)$, where $|\bar x_i|=n$.

\subsection{Definable sets and interpretability}\label{se:interpretation}

Let $\MB = \langle B; L(\MB)\rangle$ be an algebraic structure. A subset $X \subseteq B^n$ is called {\em $0$-definable} (or {\em absolutely definable}, or {\em definable without parameters}) in $\MB$ if there is a first-order formula $\phi(x_1,\ldots,x_n)$ in the language $L(\MB)$ such that
\[
X = \{(b_1,\ldots,b_n) \in B^n \mid \MB \models \phi(b_1, \ldots,b_n)\}.
\] 
The set $X$ is denoted by $\phi(\MB)$. Fix a set of elements $P\subseteq B$. If $\psi(x_1,\ldots,x_n,y_1,\ldots,y_k)$ is a formula in $L(\MB)$ and $\bar{p}=(p_1,\ldots,p_k)$ is a tuple of elements from $P$, then the set 
\[
\{(b_1,\ldots,b_n) \in B^n \mid \MB \models \psi(b_1,\ldots,b_n, p_1,\ldots,p_k)\}
\] 
is called \emph{definable} in $\MB$ (\emph{with parameters} $\bar{p}$). It is denoted by $\psi(\MB,\bar p)$ or $\psi(B^n,\bar p)$.

Equivalently, $X\subseteq B^n$ is definable in $\MB$ if and only if there exists a finite set of parameters $P\subseteq B$ and a formula $\psi(x_1,\ldots,x_n)$ in the language $L(\MB)\cup P$ such that $X=\psi(\MB_P)$, here the algebraic structure $\MB_P=\langle B;L(\MB)\cup P\rangle$ is obtained from $\MB$ by adding new constant symbols from $P$ into the language.

Let $X\subseteq B^n$, $Y\subseteq B^m$ be sets and $\sim_X$, $\sim_Y$ be equivalence relations on $X$, $Y$, respectively.
A map $g\colon X/{\sim_X}\to Y/{\sim_Y}$ is called {\em definable} in $\MB$, if the following set (termed the {\em preimage in $\mathbb B$ of the graph of $g$})
\[
\begin{split}
\{(b_1,\ldots,b_n,c_1,\ldots,c_m) \in B^{n+m}\mid\ &(b_1,\ldots,b_n)\in X, (c_1,\ldots,c_m)\in Y, \\ 
&g((b_1,\ldots,b_n)/{\sim_X})=(c_1,\ldots,c_m)/{
\sim_Y}
\}
\end{split}
\]
is definable in $\MB$. It is obvious that the preimage in $\mathbb B$ of the graph $g$ coincides with preimage in $\mathbb B$ of the graph of $\tilde g\colon X\to Y/{\sim_Y}$, where $\tilde g(\bar x)= g(\bar x/{\sim_X})$, $\bar x\in X$. Note that here $g$ may also be a function of arity $0$ (i.e., a constant or a point $\bar y/{\sim_Y}\in Y/{\sim_Y}$).
Similarly, a relation $Q$ of arity $m$ on the quotient-set set $X/{\sim_X}$ is {\em definable} in $\mathbb B$ if the {\em preimage in $\mathbb B$ of its graph}, i.e., the full preimage of the set $Q$ in $B^{nm}$, is definable in $\mathbb B$. Correspondingly, $g$ and $Q$ are $0$-{\em definable} in $\mathbb B$ if preimages of their graphs in $\MB$ are $0$-definable.

By default, we consider the sets $\mathbb N$, $\mathbb Z$, $\mathbb Q$, $\mathbb R$ and $\mathbb C$ in the language $\{+,\cdot,0,1\}$.
Note that definable sets in arithmetic $\mathbb N$, the so-called arithmetic sets, are well-studied.
In particular, it is known that every computably enumerable set in $\mathbb N$ is definable. The same holds for the ring of integers $\mathbb Z$. Furthermore, every element of $\mathbb N$ (or $\mathbb Z$, or $\mathbb Q$) is absolutely definable in 
$\mathbb N$ ($\mathbb Z$, $\mathbb Q$), so every definable set in $\mathbb N$ ($\mathbb Z$, $\mathbb Q$) is absolutely definable.

\begin{definition} \label{de:interpretation}\label{de:interpretable} 
An algebraic structure $\MA = \langle A;f,\ldots,R,\ldots,c,\ldots\rangle$ is $0$-{\em interpretable} (or {\em absolutely interpretable}, or {\em interpretable without parameters}) in an algebraic structure $\MB=\langle B;L(\MB)\rangle$ if the following conditions hold:
\begin{enumerate}[1)]
\item there is a subset $A^\ast \subseteq B^n$ $0$-definable in $\MB$, 
\item there is an equivalence relation $\sim$ on $A^\ast$ $0$-definable in $\MB$, 
\item there are interpretations $f^A, \ldots$, $R^A,\ldots$, $c^A, \ldots$ of the symbols $f, \ldots,$ $R,\ldots,$ $c, \ldots$ on the quotient set $A^\ast /{\sim}$, all $0$-definable in $\MB$, 
\item the structure $\MA^\ast=\langle A^\ast /{\sim}; f^A, \ldots,R^A,\ldots, c^A, \ldots \rangle$ is $L(\MA)$-isomorphic to~$\MA$.
\end{enumerate}
\end{definition}

A more general case is the interpretation with parameters.

\begin{definition}\label{de:interpretation_P}
An algebraic structure $\MA=\langle A;L(\MA)\rangle$ is {\em interpretable} (with parameters) in an algebraic structure $\MB=\langle B;L(\MB)\rangle$ if there is a finite subset of parameters $P\subseteq B$, such that $\MA$ is absolutely interpretable in $\MB_P$. It this case, we write $\MA\rightsquigarrow\MB$.
\end{definition}

In the case when the equivalence relation $\sim$ from Definition~\ref{de:interpretation} is the identity relation, it is usually said that $\mathbb A$ is {\em definable} in $\mathbb B$ (with or without parameters correspondingly).

The structure $\mathbb A^\ast$ from Definitions~\ref{de:interpretation}, \ref{de:interpretation_P} is completely described by the tuple of parameters $\bar p$ (all the parameters used in the formulas from 1), 2), and 3) above) and the following set of formulas in the language $L(\mathbb B)$
\begin{equation} \label{eq:code}
\Gamma = \{U_\Gamma(\bar x,\bar y), E_\Gamma(\bar x, \bar x^\prime,\bar y), Q_\Gamma(\bar x_1, \ldots,\bar x_{t_Q},\bar y) \mid Q \in L(\mathbb A)\},
\end{equation}
where $\bar x$, $\bar x^\prime$ and $\bar x_i$ are $n$-tuples of variables and $\bar y = (y_1, \ldots,y_k)$ is a tuple of extra variables for parameters $\bar p$.
Namely, $U_\Gamma$ defines in $\mathbb B$ a set $A_\Gamma = U_\Gamma(B^n,\bar p) \subseteq B^n$ (the set $A^\ast$ in Definition~\ref{de:interpretation}), $E_\Gamma$ defines an equivalence relation $\sim_\Gamma$ on $A_\Gamma$ (the equivalence relation $\sim$ in Definition~\ref{de:interpretation}), and the formulas $Q_\Gamma$ define the preimages of the graphs for constants, functions, and predicates $Q\in L(\mathbb A)$ on the quotient set $A_\Gamma/{\sim_\Gamma}$ in such a way that the structure $\Gamma(\mathbb B,\bar p) = \langle A_\Gamma/{\sim_\Gamma}; L(\mathbb A) \rangle $ is isomorphic to $\mathbb A$.
Here $t_c=1$ for a constant $c\in L(\mathbb A)$, $t_f=n_f+1$ for a function $f\in L(\mathbb A)$ and $t_R=n_R$ for a predicate $R\in L(\mathbb A)$.
Note that we interpret a constant $c \in L(\mathbb A)$ in the structure $\Gamma(\mathbb B,\bar p)$ by the $\sim_\Gamma$-equivalence class of some tuple $\bar b_c \in A_\Gamma$ defined in $\mathbb B$ by the formula $c_\Gamma(\bar x,\bar p)$.
The number $n$ is called the {\em dimension} of $\Gamma$, denoted $n = \dim \Gamma$.
We refer to $k$ as the {\em parameter dimension} of $\Gamma$ and denote it by $\dim_{par}\Gamma$. 

We refer to $\Gamma$ as the {\em interpretation code} or just the {\em code} of the interpretation $\mathbb A\rightsquigarrow\mathbb B$.
Sometimes we identify the interpretation $\mathbb A\rightsquigarrow\mathbb B$ with its code $\Gamma$ or $(\Gamma,\bar p)$ (if the interpretation has parameters $\bar p$).
We may also write $\mathbb A\stackrel{\Gamma}{\rightsquigarrow}\mathbb B$, $\mathbb A\stackrel{\Gamma, \bar p}{\rightsquigarrow}\mathbb B$, $\mathbb A\cong\Gamma(\mathbb B)$ or $\mathbb A \cong \Gamma(\mathbb B,\bar p)$ to say that $\mathbb A$ is interpretable in $\mathbb B$ by means of a pair $(\Gamma,\bar p)$.
If the interpretation is absolute, we write $\mathbb A\cong\Gamma(\mathbb B)$ or $\mathbb A\cong\Gamma(\MB,\emptyset)$.

By $\mu_\Gamma$ we denote a surjective map $A_\Gamma \to A$ that gives rise to an isomorphism $\bar \mu_\Gamma\colon \Gamma(\MB,\bar p) \to \MA$. We refer to $\mu_\Gamma$ as the {\em coordinate map} of the interpretation $(\Gamma,\bar p)$. Note that such $\mu_\Gamma$ may not be unique because $\bar\mu_\Gamma$ is defined up to an automorphism of $\MA$. For this reason, the coordinate map $\mu_\Gamma$ is sometimes added to the interpretation notation: $(\Gamma,\bar p, \mu_\Gamma)$. A coordinate map $\mu_\Gamma\colon A_\Gamma \to A$ gives rise to a map $\mu_\Gamma^m\colon A^m_\Gamma\to A^m$ on the corresponding Cartesian powers, which we often denote by $\mu_\Gamma$.

When the formula $E_\Gamma$ defines the identity relation $(x_1=x^\prime_1)\wedge\ldots\wedge(x_n=x^\prime_n)$, the surjection $\mu_\Gamma$ is injective. In this case, $(\Gamma,\bar p, \mu_\Gamma)$ is called an {\em injective interpretation}.

\begin{definition}\label{def:regular_int}
We say that $\mathbb A$ is {\em regularly interpretable} in $\mathbb B$ if there exists a code $\Gamma$~\eqref{eq:code} and an $L(\mathbb B)$-formula $\phi(y_1,\ldots,y_k)$, such that $\phi(\mathbb B)\ne\emptyset$ and for each $\bar p=(p_1,\ldots,p_k)\in \phi(\mathbb B)$ the structure $\Gamma(\mathbb B,\bar p)$ provides an interpretation $\mathbb A\cong\Gamma(\mathbb B,\bar p)$ of~$\mathbb A$ in~$\mathbb B$.
\end{definition}

We denote by $(\Gamma,\phi)$ the regular interpretation with the code $\Gamma$ and the formula $\phi$ for the parameters.
If for all $\bar p\in \phi(\mathbb B)$ the interpretation $(\Gamma,\bar p)$ is injective, then the regular interpretation $(\Gamma,\phi)$ is called {\em injective}.
We write $\mathbb A\cong \Gamma(\mathbb B,\phi)$ or $\mathbb A\stackrel{\Gamma, \phi}{\rightsquigarrow}\mathbb B$ if $\mathbb A$ is regularly interpreted in $\mathbb B$ by a code $\Gamma$ and a formula $\phi$.

Note that every absolute interpretation $\Gamma$ is a particular case of a regular interpretation.
Indeed, one can add some fictitious variable $y$ in all the formulas of the code $\Gamma$ and put $\phi(y)=(y=y)$.

Let $\Gamma$ be the code described in~\eqref{eq:code}.
The \emph{admissibility conditions $\mathcal{AC}_\Gamma(\bar y)$} (see~\cite{Daniyarova-Myasnikov:I}) for the code $\Gamma$ is a set of formulas in the language $L(\mathbb B)$ with free variables $\bar y=(y_1, \ldots,y_k)$, where $k=\dim_{par}\Gamma$, such that for an arbitrary $L(\mathbb B)$-structure $\widetilde {\mathbb B}$ and an arbitrary tuple $\bar q$ in $\widetilde {\mathbb B}$, the construction of $\Gamma(\widetilde {\mathbb B},\bar q)$ described in Definitions~\ref{de:interpretation}, \ref{de:interpretation_P} indeed gives an $L(\mathbb A)$-structure if and only if $\widetilde {\mathbb B}\models \mathcal{AC}_\Gamma(\bar q)$.
Observe that the admissibility conditions $\mathcal{AC}_\Gamma(\bar q)$ state that the $L(\MA)$-structure $\Gamma(\nsB,\bar q)$ is well-defined, but they do not claim that $\MA\cong\Gamma(\nsB,\bar q)$.

Suppose that $\phi(\bar y)$ is an $L(\mathbb B)$-formula and $\widetilde {\mathbb B}$ is an $L(\MB)$-structure as before.
We say that an algebraic structure $\Gamma(\nsB,\phi)$ is {\em well-defined} if $\phi(\widetilde {\mathbb B})\ne \emptyset$ and for every $\bar q\in \phi(\widetilde {\mathbb B})$, $\Gamma(\widetilde {\mathbb B},\bar q)$ is an $L(\mathbb A)$-structure and all these structures $\Gamma(\widetilde {\mathbb B},\bar q)$, $\bar q\in \phi(\widetilde {\mathbb B})$, are isomorphic to each other (again, they need not be isomorphic to $\mathbb A$).

\subsection{The composition of interpretations}
It is known that the relation $\mathbb A\rightsquigarrow\mathbb B$ is transitive on algebraic structures (see, for example,~\cite{Hodges,Daniyarova-Myasnikov:I}).
The proof of this fact is based on the notion of $\Gamma$-translation and composition of codes, which we present now.

Let 
\[
\Gamma = \{U_\Gamma(\bar x,\bar y), E_\Gamma(\bar x, \bar x^\prime,\bar y), Q_\Gamma(\bar x_1, \ldots,\bar x_{t_Q},\bar y) \mid Q \in L(\MA)\}
\]
be a code as above, consisting of $L(\MB)$-formulas.
Then for any formula $\varphi(x_1,\ldots,x_m)$ in the language $L(\mathbb A)$ there is a formula $\varphi_\Gamma(\bar x_1,\ldots,\bar x_m, \bar y)$ in the language $L(\mathbb B)$ such that if $\mathbb A\cong\Gamma(\mathbb B,\bar p)$, then for any coordinate map $\mu_\Gamma\colon A_\Gamma\to A$ one has
\[
\mathbb A\models \varphi(a_1,\ldots, a_m) \iff \mathbb B\models\varphi_\Gamma(\mu_\Gamma^{-1}(a_1),\ldots,\mu^{-1}_\Gamma(a_m),\bar p)
\]
for any elements $a_i\in A$ (see~\cite{Daniyarova-Myasnikov:I}).
Here $\mu_\Gamma^{-1}(a_i)$ means an arbitrary preimage of $a_i$ under $\mu_\Gamma$. Furthermore, for any elements $\bar b_i\in B^n$ if $\mathbb B\models\varphi_\Gamma(\bar b_1,\ldots,\bar b_m,\bar p)$ then $\bar b_i\in \mu_\Gamma^{-1}(a_i)$ for some $a_i\in A$ with $\mathbb A\models \varphi(a_1,\ldots, a_m)$.

Note that if the language $L(\mathbb A)$ is computable then there is an algorithm that given $\varphi$ computes $\varphi_\Gamma$.

\begin{definition} \label{def:code-composition}
Let $L(\mathbb A)$, $L(\mathbb B)$, $L(\mathbb C)$ be languages. Consider codes 
\[
\Gamma = \{U_\Gamma(\bar x,\bar y), E_\Gamma(\bar x, \bar x^\prime,\bar y), Q_\Gamma(\bar x_1, \ldots,\bar x_{t_Q},\bar y) \mid Q \in L(\MA)\}
\]
as above and 
\[
\Delta=\{U_\Delta(\bar u,\bar z), E_\Delta(\bar u, \bar u^\prime,\bar z), Q_\Delta(\bar u_1, \ldots,\bar u_{t_Q},\bar z) \mid Q \in L(\MB)\},
\]
which consists of $L(\MC)$-formulas, where $|\bar u|=|\bar u^\prime|=|\bar u_i|=\dim \Delta$, $|\bar z|=\dim_{par}\Delta$. Then the {\em composition} of the codes $\Gamma$ and $\Delta$ is the code
\[
\Gamma \circ \Delta = \{U_{\Gamma\circ\Delta}, E_{\Gamma\circ\Delta}, Q_{\Gamma\circ\Delta} \mid Q \in L(\MA)\}=\{(U_\Gamma)_\Delta, (E_\Gamma)_\Delta, (Q_\Gamma)_\Delta \mid Q \in L(\MA)\},
\]
where $\dim \Gamma\circ \Delta =\dim \Gamma\cdot\dim \Delta$ and $\dim_{par}\Gamma\circ\Delta = \dim_{par}\Gamma\cdot \dim\Delta+\dim_{par}\Delta$.
\end{definition}

The following is an important technical result on the transitivity of interpretations.

\begin{lemma}[\cite{Daniyarova-Myasnikov:I}]\label{le:int-transitivity}
Let $\mathbb A=\langle A;L(\mathbb A)\rangle, \mathbb B=\langle B;L(\mathbb B)\rangle$ and $\mathbb C=\langle C;L(\mathbb C)\rangle$ be algebraic structures and $\Gamma,\Delta$ be codes as above. If $\mathbb A\stackrel{\Gamma}{\rightsquigarrow}\mathbb B$ and 
$\mathbb B\stackrel{\Delta}{\rightsquigarrow}\mathbb C$ then $\mathbb A\stackrel{\Gamma\circ\Delta}{\rightsquigarrow}\mathbb C$.

Furthermore, the following conditions hold:
\begin{enumerate}[1)]
    \item If $\bar p,\bar q$ are parameters and $\mu_\Gamma, \mu_\Delta$ are coordinate maps of interpretations $\Gamma,\Delta$ then $(\bar{\bar p},\bar q)$, where $\bar{\bar p} \in \mu_\Delta^{-1}(\bar p)$, are parameters for $\Gamma\circ\Delta$; 
    \item\label{le:int2} For any coordinate map $\mu_\Delta$ of the interpretation $\mathbb B\cong\Delta(\mathbb C,\bar q)$ and any tuple $\bar{\bar p} \in \mu_\Delta^{-1}(\bar p)$ the $L(\mathbb A)$-structure $\Gamma\circ\Delta(\mathbb C,(\bar{\bar p},\bar q))$ is well-defined and isomorphic to~$\mathbb A$;
    \item The $L(\mathbb A)$-structure $\Gamma\circ\Delta (\mathbb C, (\bar{\bar p},\bar q))$ does not depend on the choice of $\bar{\bar p} \in \mu_\Delta^{-1}(\bar p)$ when $\mu_\Delta$ is fixed; 
    \item $\mu_\Gamma\circ\mu_\Delta=\mu_{\Gamma}\circ\mu^n_{\Delta}\big|_{U_{\Gamma\circ\Delta}(\mathbb C, (\bar{\bar p},\bar q))}$ is a coordinate map of the interpretation $\mathbb A\cong \Gamma\circ\Delta(\mathbb C,(\bar{\bar p},\bar q))$ and any coordinate map $\mu_{\Gamma\circ\Delta}\colon U_{\Gamma\circ\Delta}(\mathbb C, (\bar{\bar p},\bar q))\to A$ has a form $\mu_{\Gamma1}\circ\mu_\Delta$ for a suitable coordinate map $\mu_{\Gamma1}$ of the interpretation $\mathbb A\cong\Gamma(\mathbb B,\bar p)$, provided $\mu_\Delta$ is fixed;
    \item If $\Gamma,\Delta$ are absolute, then $\Gamma\circ\Delta$ is absolute; 
    \item If $\Gamma,\Delta$ are injective, then $\Gamma\circ\Delta$ is injective; 
    \item\label{le:int7} If $\Gamma,\Delta$ are regular with the corresponding formulas $\varphi,\psi$, then $\Gamma\circ\Delta$ is regular with the formula $\varphi_\Delta\wedge\psi$. 
\end{enumerate}
\end{lemma}

\begin{remark}
 Note that by construction the structure $\Gamma\circ\Delta (\mathbb C, (\bar{\bar p},\bar q))$ depends on the choice of $\mu_\Delta$, i.\,e., the set $U_{\Gamma\circ\Delta}(\mathbb C,(\bar{\bar p},\bar q))\subseteq C^{\dim \Gamma\circ\Delta}$, the relation $\sim_{\Gamma\circ\Delta}$, the constants $c_{\Gamma\circ\Delta}$, and the functions $f_{\Gamma\circ\Delta}$ in general depend on $\mu_\Delta$. On the other hand, the structure $\Gamma\circ\Delta (\mathbb C, (\bar{\bar p},\bar q))$ does not depend on the choice of $\mu_\Gamma$.
\end{remark}

\begin{remark}\label{remark2}
Let $(\Gamma,\varphi), (\Delta,\psi)$ be regular interpretations as in Item~\ref{le:int7} of Lemma~\ref{le:int-transitivity}.
Then for any tuples of parameters $\bar p\in \varphi(\mathbb B)$ and $\bar q\in \psi(\mathbb C)$, any coordinate map $\mu_\Delta$ of the interpretation $\mathbb B\cong\Delta(\mathbb C,\bar q)$, and any preimages $\bar{\bar p} \in \mu_\Delta^{-1}(\bar p)$, we have that the tuple $(\bar {\bar p},\bar q)$ is in $\varphi_\Delta\wedge\psi(\mathbb C)$.
Also, the $L(\mathbb A)$-structure $\Gamma\circ\Delta(\mathbb C,(\bar {\bar p},\bar q))$ is well-defined and isomorphic to $\mathbb A$, by item~\ref{le:int2}.
Conversely, every tuple $\bar r\in \varphi_\Delta\wedge\psi(\mathbb C)$ has a form $\bar r=(\bar{\bar p},\bar q)$, where $\bar q\in \psi(\mathbb C)$, $\bar{\bar p}\in \varphi_\Delta(\mathbb C,\bar q)$.
For any coordinate map $\mu_\Delta$ of the interpretation $\mathbb B\cong\Delta(\mathbb C,\bar q)$, the tuple $\bar p=\mu_\Delta(\bar{\bar p})$ is in $\varphi(\mathbb B)$.
So, the $L(\mathbb A)$-structure $\Gamma\circ\Delta(\mathbb C,\bar r)=\Gamma\circ\Delta(\mathbb C,(\bar {\bar p},\bar q))$ is well-defined and isomorphic to $\mathbb A$, again due to Item~\ref{le:int2} of Lemma~\ref{le:int-transitivity}. 
\end{remark}

\subsection{Bi-interpretations and elementary equivalence}
\label{se:bi-int}

In this section we discuss a very strong version of mutual interpretability of two structures, so-called {\em bi-interpretability}. The following definition uses notation from Lemma~\ref{le:int-transitivity}.

\begin{definition}\label{def:bi}
Algebraic structures $\mathbb A$ and $\mathbb B$ are called {\em strongly bi-interpretable} (with parameters) in each other if 
there exists an interpretation $(\Gamma,\bar p,\mu_\Gamma)$ of $\mathbb A$ in $\mathbb B$ and an interpretation $(\Delta,\bar q,\mu_\Delta)$ of $\mathbb B$ in $\mathbb A$ (so the algebraic structures $\Gamma\circ\Delta(\mathbb A,(\bar{\bar p},\bar q))$ and $\Delta\circ\Gamma(\mathbb B,(\bar{\bar q},\bar p))$ are well-defined and $\Gamma\circ\Delta(\mathbb A,(\bar{\bar p},\bar q))$ is isomorphic to $\mathbb A$, while $\Delta\circ\Gamma(\mathbb B,(\bar{\bar q},\bar p))$ is isomorphic to $\mathbb B$), such that the coordinate maps $\mu_\Gamma\circ\mu_\Delta\colon A_{\Gamma \circ \Delta} \to A$ and $\mu_\Delta\circ\mu_\Gamma\colon B_{\Delta \circ \Gamma} \to B$ are definable in $\mathbb A$ and $\mathbb B$, respectively.
\end{definition}

Note that there is another slightly different notion of {\em bi-interpretation}, which we sometimes call a {\em weak bi-interpretation} for contrast, where in the above definition the condition of definability of maps $\mu_\Gamma\circ\mu_\Delta$ and $\mu_\Delta\circ\mu_\Gamma$ is replaced by a weaker one that requires definability of some coordinate maps $A_{\Gamma \circ \Delta} \to A$ and $B_{\Delta \circ \Gamma} \to B$.

In a number of papers, authors do not even mention the difference between these two notions of bi-interpretation, implicitly assuming one or the other. To keep the exposition precise, throughout this papers we are explicit in our use of this terminology. Observe that the bi-interpretation defined in the books~\cite{Hodges} and~\cite{KharlampovichMyasnikovSohrabi:2021} is weak, but in the paper~\cite{Aschenbrenner-etal:2020} it is strong. There are many interesting applications of strong bi-interpretations which we cannot derive from the weak ones. However, it is important to mention that right now we do not have examples of weak interpretations of algebraic structures that are not strong.

\begin{definition}\label{def:reg}
Algebraic structures $\mathbb A$ and $\mathbb B$ are called {\em regularly bi-interpretable} if 
\begin{enumerate}[1)]
\item there exist a regular interpretation $(\Gamma,\varphi)$ of $\mathbb A$ in $\mathbb B$ and a regular interpretation $(\Delta,\psi)$ of $\mathbb B$ in $\mathbb A$;
\item there exists formula $\theta_\mathbb A(\bar u, x, \bar r)$ in $L(\mathbb A)$, where $|\bar u|={\dim\Gamma\cdot\dim\Delta}$, $|\bar r|={\dim_{par}\Gamma\circ\Delta}$, such that for any tuple $\bar r_0\in \varphi_\Delta\wedge \psi(\mathbb A)$, the formula $\theta_\mathbb A(\bar u, x, \bar r_0)$ defines some coordinate map $U_{\Gamma\circ\Delta}(\mathbb A,\bar r_0)\to A$;
\item there exists formula $\theta_\mathbb B(\bar u, x, \bar t)$ in $L(\mathbb B)$, where $|\bar u|={\dim\Gamma\cdot\dim\Delta}$, $|\bar t|={\dim_{par}\Delta\circ\Gamma}$, such that for any tuple $\bar t_0\in \psi_\Gamma\wedge \varphi(\mathbb B)$, the formula $\theta_\mathbb B(\bar u, x, \bar t_0)$ defines some coordinate map $U_{\Delta\circ\Gamma}(\mathbb B,\bar t_0)\to B$. 
\end{enumerate}
\end{definition}

Regular bi-interpretability allows us to describe non-standard models of one algebraic structure by means of another one (see Theorem~\ref{th:equiv1} below).

\begin{definition}\label{def:st_reg}
We say that $\mathbb A$ and $\mathbb B$ are {\em strongly regularly bi-interpretable}, if they are regularly bi-interpretable, i.\,e., 1)--3) hold, and additionally 
\begin{enumerate}
\item[4)] for any 
pair of parameters $(\bar p, \bar q)$, $\bar p\in \varphi(\mathbb B)$, $\bar q\in \psi(\mathbb A)$, there exists a pair of coordinate maps $(\mu_\Gamma,\mu_\Delta)$ for interpretations $(\Gamma,\bar p)$ and $(\Delta,\bar q)$ such that for any $\bar r_0=(\bar{\bar p},\bar q)$, $\bar{\bar p}\in \mu^{-1}_\Delta(\bar p)$, and $\bar t_0=(\bar{\bar q},\bar p)$, $\bar{\bar q}\in \mu^{-1}_\Gamma(\bar q)$, the coordinate maps ${\mu_\Gamma\circ\mu_\Delta\colon} U_{\Gamma\circ\Delta}(\mathbb A,\bar r_0)\to A$ and ${\mu_\Delta\circ\mu_\Gamma\colon} U_{\Delta\circ\Gamma}(\mathbb B,\bar t_0)\to B$ are defined in $\mathbb A$ and $\mathbb B$, correspondingly, by the formulas $\theta_\mathbb A(\bar u, x, \bar r_0)$ and $\theta_\mathbb B(\bar u, x, \bar t_0)$.
\end{enumerate}
\end{definition}

\begin{remark}
Note that Item~4) in Definition~\ref{def:st_reg} does not imply items~2), 3) in Definition~\ref{def:reg}.
The statement in item~2) must be true for all tuples of parameters from the set $X=\varphi_\Delta\wedge\psi(\mathbb A)$. From Remark~\ref{remark2} we know that 
\[
X=\{(\bar {\bar p},\bar q)\:\mid\: \bar q\in\psi(\mathbb A), \:\bar p \in \varphi(\mathbb B), \: \bar{\bar p}\in \mu_\Delta^{-1}(\bar p)\},
\]
where $\mu_\Delta$ runs through all coordinate maps of the interpretation $\mathbb B\cong\Delta(\mathbb A,\bar q)$.
Denote by $\mu_\Delta^{\bar p,\bar q}$ the coordinate map of the interpretation $\mathbb B\cong\Delta(\mathbb A,\bar q)$ from item~4) that corresponds to the pair of parameters $(\bar p,\bar q)$. Thus the statement in item~4) must be true for all tuples of parameters from the set 
\[
Y=\{(\bar {\bar p},\bar q)\:\mid\: \bar q\in\psi(\mathbb A), \:\bar p \in \varphi(\mathbb B),\: \bar{\bar p}\in (\mu_\Delta^{\bar p,\bar q})^{-1}(\bar p)\}.
\]
So, $Y\subseteq X$, but it may be $Y\ne X$. 
\end{remark}

If algebraic structures $\mathbb A$ and $\mathbb B$ are strongly regularly bi-interpretable then we may use all conclusions from both regular bi-interpretability and strong bi-interpretation with parameters.

Recall that the \emph{first-order classification problem} for a structure $\mathbb A$ asks one to describe ``algebraically'' all structures $\widetilde{\mathbb A}$ such that $\mathbb A \equiv \widetilde{\mathbb A}$. 
In other words, the first-order classification problem for $\mathbb A$ requires algebraically to describe all models of the complete first-order theory $\Th(\mathbb A)$ of $\mathbb A$. Here, by ``algebraically'' we mean that the description has to reveal the algebraic structure of every model of $\Th(\mathbb A)$.

The following result, in the case of a regular bi-interpretation $\mathbb A \cong \Gamma(\mathbb B,\varphi)$, describes the algebraic structure of the models of $\Th(\mathbb A)$ via the interpretation $\Gamma$ and the models of $\Th(\mathbb B)$.
The efficacy of this description depends on the understanding of the models $\widetilde{\mathbb B}$ of $\Th(\mathbb B)$.
If $\Gamma(\widetilde{\mathbb B},\varphi)$ is well-defined, that is, the structures $\Gamma(\widetilde{\mathbb B},\bar p)$, $\bar p \in \varphi(\widetilde{\mathbb B})$, are well-defined and pairwise isomorphic to each other, then we can choose an arbitrary representative $\bar p_0 \in \varphi(\widetilde{\mathbb B})$ and view $\Gamma(\widetilde{\mathbb B},\varphi)$ as $\Gamma(\widetilde{\mathbb B},\bar p_0)$.

\begin{theorem}[\cite{Daniyarova-Myasnikov:I}]\label{th:equiv1}
Let $\mathbb A$ and $\mathbb B$ be regularly bi-interpretable in each other, so $\mathbb A \cong \Gamma(\mathbb B,\varphi)$ and $\mathbb B\cong\Delta(\mathbb A,\psi)$. Then
\begin{enumerate}[1)]
\item For any $\widetilde{\mathbb B}\equiv\mathbb B$ the algebraic structure $\Gamma(\widetilde{\mathbb B},\varphi)$ is well-defined and $\mathbb A\equiv \Gamma(\widetilde{\mathbb B},\varphi)$;
\item Every $L(\mathbb A)$-structure $\widetilde{\mathbb A}$ elementarily equivalent to $\mathbb A$ is isomorphic to $\Gamma(\widetilde{\mathbb B},\varphi)$ for some $\widetilde{\mathbb B}\equiv\mathbb B$;
\item For any $\mathbb B_1\equiv\mathbb B\equiv\mathbb B_2$ one has 
\[
\Gamma(\mathbb B_1,\varphi) \cong \Gamma(\mathbb B_2,\varphi) \iff \mathbb B_1 \cong \mathbb B_2.
\]
\end{enumerate}
\end{theorem}

Following the practice in non-standard arithmetic and non-standard analysis, we term the structures of the form $\Gamma(\widetilde{\mathbb B},\bar p)$ for some $\widetilde{\mathbb B}\equiv\mathbb B$ and $\bar p \in \varphi(\widetilde{\mathbb B})$ above as \emph{non-standard models} of $\mathbb A$ with respect to interpretation $\Gamma$. In fact, the result below shows that quite often the non-standard models in $\widetilde{\mathbb B}$ do not depend on the choice of interpretation $\Gamma$, i.\,e., for every $\widetilde{\mathbb B}\equiv\mathbb B$, there is only one non-standard model of $\mathbb A$, up to isomorphism.

\begin{theorem} [Uniqueness of non-standard models, \cite{Daniyarova-Myasnikov:I}]\label{th:equiv2}
Let a finitely generated structure $\mathbb A$ in a finite signature be regularly interpretable in $\mathbb Z$ in two ways, as $\Gamma_1(\mathbb Z,\varphi_1)$ and $\Gamma_2(\mathbb Z,\varphi_2)$. Then there exists a formula $\theta(\bar x_1,\bar x_2, \bar y_1,\bar y_2)$, $|\bar x_i|=\dim\Gamma_i$, $|\bar y_i|=\dim_{par}\Gamma_i$, such that for any $\bar p_i\in\phi_i(\mathbb Z)$ the formula $\theta(\bar x_1,\bar x_2, \bar p_1,\bar p_2)$ defines an isomorphism $\Gamma_1(\mathbb Z,\bar p_1)\to\Gamma_2(\mathbb Z,\bar p_2)$. Moreover, if $\widetilde{\mathbb Z} \equiv \mathbb Z$, then algebraic structures $\Gamma_1(\widetilde{\mathbb Z},\varphi_1), \Gamma_2(\widetilde{\mathbb Z},\varphi_2)$ are well-defined and $\theta(\bar x_1,\bar x_2, \bar p_1,\bar p_2)$ defines an isomorphism $\Gamma_1(\widetilde{\mathbb Z},\bar p_1)\to\Gamma_2(\widetilde{\mathbb Z},\bar p_2)$ for any $\bar p_1\in\varphi_1(\widetilde{\mathbb Z})$, $\bar p_2\in\varphi_2(\widetilde{\mathbb Z})$.
\end{theorem}

\begin{corollary} \label{co:unique-non-stand}
 Let $\mathbb A$ be a finitely generated structure in a finite signature regularly bi-interpretable with $\mathbb Z$. Then for every $\widetilde{\mathbb Z} \equiv \mathbb Z$ there is a unique up to isomorphism non-standard model $\mathbb A(\widetilde{\mathbb Z})$ of $\mathbb A$ and for any structure $\nsB$ one has $\nsB \equiv \mathbb A$ if and only if $\nsB \cong \mathbb A(\widetilde{\mathbb Z})$ for a suitable $\widetilde{\mathbb Z} \equiv \mathbb Z$.
\end{corollary}

If $\MA$ is as in Corollary~\ref{co:unique-non-stand} then the structure $\MA(\widetilde{\mathbb Z})$ is called the {\em non-standard model of $\MA$ with respect to $\widetilde{\mathbb Z}$}. If $\MA \cong \Gamma(\Z)$ is an interpretation of $\MA$ in $\Z$ then $\MA(\widetilde{\mathbb Z}) \cong \Gamma(\widetilde{\mathbb Z})$ and the algebraic structure of $\MA(\widetilde{\mathbb Z})$ is revealed via the interpretation~$\Gamma$.

\subsection{Invertible interpretations and elementary equivalence} \label{se:inv_int_and_bi-int}

In this section we discuss \emph{invertible interpretations} of $\mathbb A$ in $\mathbb B$, which can be also viewed as \emph{half-bi-interpretations} of $\mathbb A$ in $\mathbb B$.
The point is that in various applications, like the first-order classification, one does not need the full strength of bi-interpretability of $\mathbb A$ in $\mathbb B$ to characterize algebraically all structures $\widetilde{\mathbb A}$ with $\widetilde{\mathbb A} \equiv \mathbb A$.
Rather, it suffices to have only ``half of bi-interpretability'' of $\mathbb A$ in $\mathbb B$ to describe all models of the theory $\Th(\mathbb A)$ in terms of the models of the theory $\Th(\mathbb B)$.

\begin{definition} \label{de:invert}
An interpretation $(\Gamma,\bar p)$ of $\mathbb A$ in $\mathbb B$ is called {\em invertible} if there exists an interpretation $(\Delta,\bar q)$ of $\mathbb B$ in $\mathbb A$ and $L(\mathbb A)$-isomorphism $\mu_\mathbb A\colon \Gamma \circ \Delta(\mathbb A) \to \mathbb A$ definable in $\mathbb A$. In this case we refer to $(\Delta,\bar q)$ as a {\em right inverse} of $(\Gamma,\bar p)$.
\end{definition}

Observe that the right inverse of $(\Gamma,\bar p)$ is not unique.
Also note that for any coordinate map $\mu_\Delta\colon B_\Delta\to B$ there exists a coordinate map $\mu_\Gamma\colon A_\Gamma\to A$ such that the composition $\mu_\Gamma\circ\mu_\Delta$ is definable in $\mathbb A$.
Indeed, if $\mu_\Gamma\colon A_\Gamma\to A$ is some kind of coordinate map, then the composition $\mu_{\Gamma\circ\Delta}=\mu_\Gamma\circ\mu_\Delta$ gives rise to an $L(\mathbb A)$-isomorphism $\bar\mu_{\Gamma\circ\Delta}$.
Therefore, $\alpha=\mu_{\mathbb A}\circ \bar \mu^{-1}_{\Gamma\circ \Delta}$, where $\mu_{\mathbb A}$ is as in Definition~\ref{de:invert}, is an automorphism of $\mathbb A$.
So, $\mu_{\Gamma 0}=\alpha\circ \mu_\Gamma\colon A_\Gamma\to A$ is also a coordinate map.
For the coordinate map $\mu_{\Gamma 0\circ\Delta}=\mu_{\Gamma 0}\circ\mu_\Delta\colon A_{\Gamma\circ\Delta}\to A$ one has $\bar\mu_{\Gamma 0\circ\Delta}=\mu_{\mathbb A}$. The latter is definable in $\mathbb A$.

\begin{definition}\label{de:reg_invert}
A regular interpretation $(\Gamma,\varphi)$ of $\mathbb A$ in $\mathbb B$ is called {\em regularly invertible} if there exist a regular interpretation $(\Delta,\psi)$ of $\mathbb B$ in $\mathbb A$ (a {\em right inverse for $(\Gamma,\varphi)$}) and a formula $\theta(\bar u,x,\bar z)$ in $L(\mathbb A)$, $|\bar u|=\dim\Gamma\cdot\dim\Delta$, such that for every tuple of parameters $\bar r\in \varphi_\Delta\wedge\psi(\mathbb A)$, the formula $\theta(\bar u, x, \bar r)$ defines an isomorphism $\Gamma \circ \Delta(\mathbb A,\bar r) \cong \mathbb A$.
Moreover, $(\Gamma,\varphi)$ is called {\em regularly injectively invertible} if $(\Delta,\psi)$ is injective. 
\end{definition}

As before, in this case we may always assume that the composition $\mu_\Gamma\circ\mu_\Delta$ of coordinate maps $\mu_\Gamma\colon U_\Gamma(\mathbb B,\bar p)\to A$ and $\mu_\Delta\colon U_\Delta(\mathbb A,\bar q)\to B$ is definable for every $\bar p\in \varphi(\mathbb B)$ and $\bar q\in\psi(\mathbb A)$.

The following result gives a first step in the general approach to the first-order classification problem.

\begin{proposition}[{\cite[Lemma 4.4(2)]{KharlampovichMyasnikovSohrabi:2021}}]\label{pr:elem_equiv} \label{pr:reg-int-elem-equiv}
Let $\mathbb A$ be regularly interpretable in $\mathbb B$, so $\mathbb A \cong \Gamma(\mathbb B,\phi)$ for some interpretation $\Gamma$ and a formula $\phi$. Then for every structure $\widetilde{\mathbb B}$ with $\widetilde{\mathbb B} \equiv \mathbb B$ and for every tuple $\bar u \in \phi(\widetilde{\mathbb B})$ the algebraic structure $\Gamma(\widetilde{\mathbb B},\bar u)$ is well-defined and one has $\Gamma(\widetilde{\mathbb B},\bar u) \equiv \mathbb A$.
\end{proposition}

\begin{corollary}
If $\mathbb A$ is absolutely interpretable in $\mathbb B$, so $\mathbb A\cong\Gamma(\mathbb B,\emptyset)$, and $\widetilde{\mathbb B}\equiv\mathbb B$, then the algebraic structure $
\Gamma(\widetilde{\mathbb B},\emptyset)$ is well-defined and one has $\Gamma(\widetilde{\mathbb B},\emptyset)\equiv\mathbb A$.
\end{corollary}

Note that in Proposition~\ref{pr:reg-int-elem-equiv} the structures $\Gamma(\widetilde{\mathbb B},\bar u)$ are first-order equivalent for different 
$\bar u \in \phi(\widetilde{\mathbb B})$, but not necessary isomorphic to each other.

\begin{proposition}\label{pr:iso1}
Let $\mathbb A \cong \Gamma(\mathbb B,\phi)$ be a regularly invertible interpretation of $\mathbb A$ in $\mathbb B$ with right inverse $\mathbb B\cong\Delta(\mathbb A,\psi)$. Then for every $\widetilde{\mathbb A} \equiv \mathbb A$, an algebraic structure $\Gamma\circ\Delta(\widetilde{\mathbb A},\phi_\Delta\wedge\psi)$ is well-defined and one has
\[
\Gamma\circ\Delta(\widetilde{\mathbb A},\phi_\Delta\wedge\psi) \cong \widetilde{\mathbb A}.
\]
\end{proposition}

Similar to the case of  regular bi-interpretations the following result describes the algebraic structure of models of $\Th(\mathbb A)$. 

\begin{theorem}[\cite{Daniyarova-Myasnikov:I}]\label{th:complete-scheme}
Let $\mathbb A \cong \Gamma(\mathbb B,\phi)$ be a regularly invertible interpretation of $\mathbb A$ in $\mathbb B$ and $\mathbb B\cong\Delta(\mathbb A,\psi)$ be its right inverse regular interpretation. Then for any algebraic structure $\widetilde{\mathbb A}$ of language $L(\mathbb A)$ the following conditions are equivalent:
\begin{itemize}
\item[(1)] $\widetilde{\mathbb A}\equiv\mathbb A$;
\item[(2)] there exist an algebraic structure $\widetilde{\mathbb B}\equiv \mathbb B$ and a tuple $\bar b\in \phi (\widetilde{\mathbb B})$, such that
\[
\widetilde{\mathbb A}\cong \Gamma(\widetilde{\mathbb B}, \bar b);
\]
\item[(3)] for every tuple $\bar a \in \psi(\widetilde{\mathbb A})$ an algebraic structure $\widetilde{\mathbb B}= \Delta(\widetilde{\mathbb A},\bar a)$ is well-defined and $\widetilde{\mathbb B}\equiv \mathbb B$, furthermore, $\Gamma(\widetilde{\mathbb B},\phi)$ is well-defined and
\[
\widetilde{\mathbb A}\cong \Gamma(\widetilde{\mathbb B}, \phi).
\]
\end{itemize}
\end{theorem}

\begin{corollary}[\cite{Daniyarova-Myasnikov:I}]
Let $\mathbb A \cong \Gamma(\mathbb B,\phi)$ be a regularly invertible interpretation of $\mathbb A$ in $\mathbb B$. Then for any structure $\widetilde{\mathbb A}$ of language $L(\mathbb A)$ one has 
\[
\widetilde{\mathbb A} \equiv \mathbb A \iff \widetilde{\mathbb A} \cong \Gamma(\widetilde{\mathbb B},\phi) \ \mbox{for some} \ \widetilde{\mathbb B} \equiv\mathbb B.
\]
\end{corollary}

\begin{proposition}[\cite{Daniyarova-Myasnikov:I}]\label{pr:iso2}
Let $\mathbb A \cong \Gamma(\mathbb B,\phi)$ be a regularly invertible interpretation of $\mathbb A$ in $\mathbb B$ with right inverse $\mathbb B\cong\Delta(\mathbb A,\psi)$. Then for any $\mathbb A_1 \equiv \mathbb A_2 \equiv \mathbb A$ and $\bar a_1\in \psi(\mathbb A_1)$, $\bar a_2\in \psi(\mathbb A_2)$ algebraic structures $\Delta(\mathbb A_1,\bar a_1)$, $\Delta(\mathbb A_2,\bar a_2)$ are well-defined and one has
\[
\Delta(\mathbb A_1,\bar a_1) \cong \Delta(\mathbb A_2,\bar a_2) \Longrightarrow \mathbb A_1 \cong \mathbb A_2.
\]
\end{proposition}

\begin{proposition}[\cite{Daniyarova-Myasnikov:I}]
If $\mathbb A \cong \Gamma(\mathbb B,\phi)$ is a regular bi-interpretation then for any $\widetilde{\mathbb B}\equiv \mathbb B$ the algebraic structure $\Gamma(\widetilde{\mathbb B},\phi)$ is well-defined.
\end{proposition}

\begin{corollary}[\cite{Daniyarova-Myasnikov:I}]
If $\mathbb A \cong \Gamma(\mathbb B,\phi)$ is a regular bi-interpretation then for any $\mathbb B_1 \equiv \mathbb B_2 \equiv \mathbb B$ algebraic structures $\Gamma(\mathbb B_1,\phi), \Gamma(\mathbb B_2,\phi)$ are well-defined and one has 
\[
\Gamma(\mathbb B_1,\phi) \cong \Gamma(\mathbb B_2,\phi) \iff \mathbb B_1 \cong \mathbb B_2.
\]
\end{corollary}

The result above gives a general approach to the first-order classification problem for $\mathbb A$. This description of models $\Th(\mathbb A)$ in terms of the models of $\Th(\mathbb B)$ makes sense if we know algebraic structure of models of $\Th(\mathbb B)$.

The following corollary is useful when studying first-order classification problem. 

Another notable property here is that for regular interpretations of a finitely generated structure $\mathbb A$ in $\mathbb{N}$, the choice of formulas $\Gamma$ that give such an interpretation $\mathbb A\cong \Gamma(\mathbb N)$ does not matter.

\begin{theorem}[Uniqueness of non-standard models, \cite{Daniyarova-Myasnikov:I}]\label{th:well-def-inN} 
Let a finitely generated structure $\mathbb A$ in a finite signature be regularly interpretable in $\N$ in two ways, as $\Gamma_1(\N,\phi_1)$ and $\Gamma_2(\N,\phi_2)$. If for a structure $\nsN$ one has $\nsN \equiv \N$, then algebraic structures $\Gamma_1(\nsN,\phi_1), \Gamma_2(\nsN,\phi_2)$ are well-defined and $\Gamma_1(\nsN,\phi_1)\cong\Gamma_2(\nsN,\phi_2)$.
\end{theorem}

\begin{theorem}[Uniqueness of non-standard models, \cite{Daniyarova-Myasnikov:I}]
Let a finitely generated structure $\mathbb A$ in a finite signature be regularly interpretable in $\Z$ in two ways, as $\Gamma_1(\Z,\phi_1)$ and $\Gamma_2(\Z,\phi_2)$. If for a structure $\nsZ$ one has $\nsZ \equiv \Z$, then algebraic structures $\Gamma_1(\nsZ,\phi_1), \Gamma_2(\nsZ,\phi_2)$ are well-defined and $\Gamma_1(\nsZ,\phi_1)\cong\Gamma_2(\nsZ,\phi_2)$.
\end{theorem}

Concerning interpretation in $\mathbb Z$, we further note the following.

\begin{proposition}[\cite{Daniyarova-Myasnikov:I}]
Let $G$ be a countable group with an arithmetic multiplication table. Then $G$ is absolutely interpretable in $\Z$.
\end{proposition}

\begin{proposition}[\cite{Daniyarova-Myasnikov:I}]
Let $\mathbb A$ be a finitely generated algebraic structure in a finite language and $(\Gamma,\phi)$ be a regular interpretation of $\mathbb A$ in $\Z$. Then for any $\nsZ\equiv\Z$ algebraic structure $\Gamma(\nsZ,\phi)$ is well-defined.
\end{proposition}

\begin{proposition}[\cite{Daniyarova-Myasnikov:I}]
Let $\mathbb A$ be a finitely generated algebraic structure in a finite language and $(\Gamma,\phi)$ be a regular interpretation of $\mathbb A$ in $\N$. Then for any $\nsN\equiv\N$ algebraic structure $\Gamma(\nsN,\phi)$ is well-defined.
\end{proposition}

Now we extend Theorem~\ref{th:well-def-inN} to the case of interpretation in the list superstructure.
Suppose $\mathbb A=\langle A, L\rangle$ be a structure in a finite signature finitely generated over its definable substructure $\mathbb B=\langle B, L\rangle$.
By this we mean that there are $x_1,\ldots,x_k\in\mathbb A$ s.t. $\mathbb A$ is generated by $x_1, \ldots, x_n$ and $B$, that is, for every $x\in A$ there is a term $\tau(\bar x, \bar y)$ in $L$ with constants $x_1,\ldots,x_k$ and variables $y_1,\ldots,y_m$ such that $x = \tau(\bar x, \bar b)$ for some values $y_i = b_i \in B$.
Note that there is an effective enumeration $\tau_i(x_1,\ldots,x_k, \bar y)$ of all such terms, treated as terms of $\mathbb S(\mathbb A,\mathbb N)$.

\begin{theorem}[Uniqueness of non-standard models]\label{th:well-def}
Let $\mathbb A$ be a structure in a finite signature, and let $\mathbb B$ be a definable substructure of $\mathbb A$. Let $\mathbb A$ be finitely generated over $\mathbb B$. Suppose $\mathbb A$ is regularly interpreted in $\mathbb S(\mathbb B,\mathbb N)$ in two ways, as $\Gamma_1(\mathbb S(\mathbb B,\mathbb N),\phi_1)$ and $\Gamma_2(\mathbb S(\mathbb B,\mathbb N),\phi_2)$. Suppose there is an isomorphism between $\Gamma_1(\mathbb S(\mathbb B,\mathbb N),\phi_1)$ and $\Gamma_2(\mathbb S(\mathbb B,\mathbb N),\phi_2)$ whose restriction to the interpretation of $\mathbb B$ is definable in $\mathbb S(\mathbb B,\mathbb N)$. Then for every structure $\mathcal M\equiv \mathbb S(\mathbb B,\mathbb N)$, we have $\Gamma_1(\mathcal M,\phi_1)\cong \Gamma_2(\mathcal M,\phi_2)$. 
\end{theorem}

\begin{proof}

For book-keeping purposes, we introduce two copies of the structure $\mathbb A$, $\mathbb A_1\cong \mathbb A_2\cong \mathbb A$, with $\mathbb A_i=\langle A_i\mid L\rangle$ and we set $\Gamma_i$ to interpret $\mathbb A_i$. The corresponding copies of the substructure $\mathbb B$ we denote by $\mathbb B_1,\mathbb B_2$, with underlying sets $B_1, B_2$, respectively.
Let $\bar p_i$ be the parameter tuple of the interpretation $\Gamma_i$. Let $x=(x_1,\ldots,x_k)$ be a fixed generating tuple of $\mathbb A$ over $\mathbb B$.

Let $Q^*_i=A^*_i/{\sim_i}$ be the underlying set of $\Gamma_i(\mathbb S(\mathbb B,\mathbb N),\bar p_i)$, where $A^*_i$ and $\sim_i$ are as in items 1) and 2) of Definition~\ref{de:interpretation}, respectively.
For an element $a\in A_i$, denote by $a^*_i\subseteq A^*_i$ the equivalence class corresponding to $a$ under $\Gamma_i$. Let $R_i^*\subseteq Q^*_i$ be the underlying set of the interpretation of $\mathbb B_i$ via $\Gamma_i$, and $B_i^*\subseteq A_i^*$ the union of the corresponding equivalence classes, $R_i^*=B_i^*/{\sim_i}$.

Let $\psi:Q^*_1\to Q^*_2$ be an isomorphism s.t. the restriction of $\psi$ to $R_1^*$, that is, the set
\[
\bigcup_{b\in B_1} b_1^*\times \psi(b_1^*)
\]
is definable in $\mathbb S(\mathbb B,\mathbb N)$. Let $\sigma_B$ be a first order formula of $\mathbb S(\mathbb B,\mathbb N)$ that delivers this definability. That is, $\mathbb S(\mathbb B,\mathbb N)\models \sigma_B(y_1, y_2, \bar p_1, \bar p_2)$ if and only if $y_i\in B_i^*$ and $\psi([y_1^*]_{\sim_1})=[y_2^*]_{\sim_2}$, where $[y_i^*]_{\sim_i}$ denote the corresponding equivalence classes.

Let $\sigma_1(z,j,\bar y,\bar x^*_1,\bar p_1)$ be a first order formula of $\mathbb S(\mathbb B,\mathbb N)$ that records that $z\in A^*_1$, $j\in\mathbb N$, $\bar y$ is a tuple of elements of $B_1^*$ with $[\bar y]_{\sim_1}=\bar b^*$, and that $[z]_{\sim_1}=a^*_1\subseteq A_1^*$ interprets $\tau_j(x_1,\ldots,x_k,\bar b)$ under $\Gamma_1$; similarly for $\sigma_2$ and $\Gamma_2$.
Then the formula $\sigma(z_1,z_2,\bar x_1^*,\bar x_2^*,\bar p_1,\bar p_2)$ given by
\[
\exists j,\bar y_1,\bar y_2\ \sigma_B(y_1,y_2,\bar p_1,\bar p_2)\wedge \sigma_1(z_1,j,\bar y_1,\bar x_1^*,\bar p_1)\wedge \sigma_2(z_2,j,\bar y_2,\bar x_2^*,\bar p_2)
\]
defines an isomorphism $\varphi:[z_1]_{\sim_1}\mapsto [z_2]_{\sim_2}$ of $\Gamma_1(\mathbb S(\mathbb B,\mathbb N),\phi_1)$ and $\Gamma_2(\mathbb S(\mathbb B,\mathbb N),\phi_2)$. It is straightforward to record the fact that $\varphi$ is an isomorphism by a first order formula of $\mathbb S(\mathbb B,\mathbb N)$. Since $\mathcal M\equiv \mathbb S(\mathbb B,\mathbb N)$, then the same formula exhibits an isomorphism of $\Gamma_1(\mathcal M,\phi_1)\cong \Gamma_2(\mathcal M,\phi_2)$.
\end{proof}

\section{Nonstandard tuples}\label{se:tuples}

\subsection{Nonstandard models of arithmetic}\label{se:nonstandard_arithmetic}

Here we introduce some well-known definitions and facts on non-standard arithmetic. As a general reference we follow the book \cite{Kaye:1991}.

Let $\mathcal{L}_r = \{+,\cdot, 0,1\}$ be the language (signature) of rings with unity 1. By $\N = \langle N; +,\cdot,0,1\rangle$ we denote the \emph{standard arithmetic}, i.e., the set on non-negative integers $N$ with the standard addition $+$, standard multiplication $\cdot$, and constants $0, 1$. Sometimes, following common practice, we abuse the notation and denote the set $N$ as $\N$. Observe that $\langle N; +,\cdot,0,1\rangle$ is a semiring and $\langle N; +, 0\rangle$ is a commutative semigroup.

Usually by an arithmetic one understands any structure $\mathcal M$ in the signature $\mathcal{L}_r$ that satisfies Peano axioms.
However, for the purposes of the present paper, by a \emph{ model of arithmetic} we mean a structure $\nstdN$ in the signature $\mathcal{L}_r$ that is elementarily equivalent to $\mathbb N$, $\nstdN \equiv \N$. A model $\nstdM$ of arithmetic is called \emph{nonstandard} if it is not isomorphic to $\mathbb N$.
Further, notice that $\mathbb N$ is the only model of arithmetic finitely generated as an additive semigroup or as a semiring. In an equivalent approach one may use the ring of integers $\mathbb Z$ as the arithmetic, in this case non-standard models of $\mathbb{Z}$ are exactly the rings $\widetilde{\mathbb{Z}}$ which are elementarily equivalent to $\mathbb{Z}$, but not isomorphic to $\Z$. Again, $\Z$ is the only (up to isomorphism) finitely generated ring which is elementarily equivalent to $\Z$.
These both approaches are indeed equivalent from the model theory viewpoint since $\N$ and $\Z$ are absolutely bi-interpretable in each other (in fact, $\mathbb{N}$ is definable in $\mathbb{Z}$ by Lagrange's four-square theorem). 

Every nonstandard model $\nstdM$ of $\mathbb N$, in terms of order, can be described as follows. Recall that for ordered disjoint sets $A$ and $B$ by $A + B$ we denote the set $A \cup B$ with the given orders on $A$ and $B$ and $a < b$ for every $a \in A$ and $b \in B$. Also by $A\cdot B$ we denote the Cartesian product $A\times B = \{(a,b) \mid a \in A, b \in B\}$ with the left lexicographical order, i.e., $(a_1,b_1) < (a_2,b_2)$ if and only if either $a_1 < a_2$ or $a_1 = a_2$ and $b_1 < b_2$. Then $\nstdM$, as a linear order, has the form $\mathbb N+Q_\lambda \mathbb Z$, for some dense linear order $Q_\lambda$ without endpoints of cardinality $\lambda = |\nstdN|$, that is $\mathbb N$ followed by $Q_\lambda$ copies of $\mathbb Z$, as shown in Fig.~\ref{fig:hairbrush}. Note that the ordered set $Q_\lambda$ is uniquely determined by $\nstdN$.
\begin{figure}[thb]
 \centering
 \includegraphics[width=0.5\linewidth]{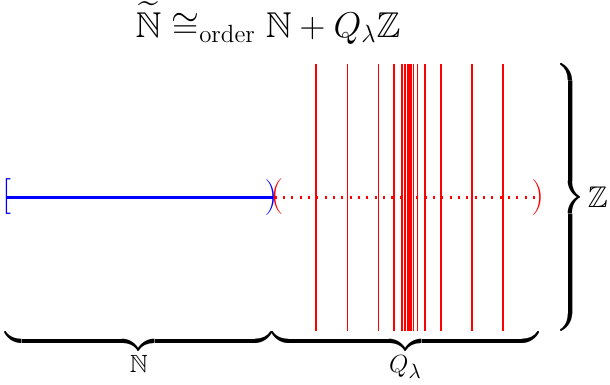}
 \caption{Nonstandard model of arithmetic.}
 \label{fig:hairbrush}
\end{figure}
In particular, any nonstandard countable model of arithmetic has the form $\mathbb N+\mathbb Q\mathbb Z$, since all countable dense linear ordered sets without endpoints are isomorphic.

In the countable case, while the order structure of $\nstdN$ is straightforward, each of the operations $+$ and $\cdot$ on $\nstdN$ is necessarily non-recursive~\cite{Tennenbaum}, so countable non-standard models of arithmetic $\nstdN$ are non-constructible or non-recursive in the sense of Malcev~\cite{Maltsev:61} or Rabin~\cite{Rabin:60}, respectively.

One can identify a natural number $n \in \N$ with a non-standard number $\tilde n = \tilde 1 + \ldots + \tilde 1$, which is the sum of $n$ non-standard units $\tilde 1$ in $\nstdN$. The map $n \to \tilde n$ gives an elementary embedding $\N \to \nstdN$, i.e., with respect to this identification $\N$ is an elementary substructure of $\nstdN$. In other words, $\mathbb N$ is a prime model of the the theory $\Th(\N)$. In the sequel we always assume that $\mathbb N\subseteq \nstdM$ via this embedding. Elements of $\N$ in $\nstdN$ are called the \emph{standard natural numbers} in $\nstdN$. They form an initial segment of $\nstdN$ with respect to the order $<$, i.e., if $x\in\nstdM$ and $x<y$ for some $y\in \mathbb N$ then $x\in\mathbb N$. It follows that in every nonstandard model $\nstdM$, there is an element $x\in\nstdM$ s.t. $x>\mathbb N$ (the latter is a shorthand for $\forall y\in\mathbb N\ x>y$).

It is a crucial fact that $\mathbb N$ is not definable in $\nstdN$ (even by a formula with parameters from $\nstdN$). It follows from the Peano induction axiom (in fact, a scheme of axioms).

\subsubsection{Nonstandard induction and overspill}
Induction in $\nstdN$ is powered by the elementary equivalence $\nstdN\equiv\mathbb N$ and works as follows: let $A\subseteq\nstdM$ be a definable with parameters subset in $\nstdM$. If $0\in A$, and $n\in A\implies n+1\in A$, then $A=\nstdM$. Equivalently, every definable subset of $\nstdM$ has a least element.

An alternative technical approach to induction is Overspill Lemma, originally due to Robinson. We do not explicitly use it in the present work, but we provide it here for context. The definition below and the following lemma are presented here after~\cite{Kaye:1991}.

\begin{definition}
 For a model $\nstdM$ of arithmetic, we say that a non-empty $I\subseteq \nstdM$ is a \emph{cut} of $\nstdM$ if and only if $x<y\in I\implies x\in I$ and $I$ is closed under successor, that is, $x\in I\implies x+1\in I$.
\end{definition}
Notice that a proper cut cannot be a definable subset. Indeed, if $I$ is definable, non-empty, and closed under successor, by induction it follows that $I=\nstdM$. This is convenient to record as the so-called Overspill Lemma.
\begin{lemma}[Overspill~\cite{Kaye:1991}]\label{le:overspill} Let $\nstdM$ be a model of arithmetic, and let $I$ be a proper cut of $\nstdM$. Suppose $\bar a\in \nstdM$ and $\varphi(x,\bar a)$ is a first-order formula with parameters $\bar a$ from $\nstdM$, s.t. $\nstdM \models \varphi(b,\bar a)$ for all $b\in I$. Then there is $c>I$ in $\nstdM$ s.t. $\nstdM\models \forall x\le c\ \varphi(x,\bar a)$.
\end{lemma}
\begin{proof}
If not, $I$ is defined by $\forall y\le x\ \varphi(x,\bar a)$ in $\nstdM$.
\end{proof}

\subsection{List superstructure}\label{se:superstructure} Recall the definition of a list superstructure of a structure $\mathbb A = \langle A;L \rangle$ (see \cite{Ashaev-Belyaev-Myasnikov:1993},\cite{Bauval:1985} \cite{KharlampovichMyasnikov:2018a}) with some small, mostly notational, changes. Let $O$ be a linear ordering with a fixed element $1 \in O$. For any $k \in O$ with $k \geq 1$ put $[1,k] = \{x \in O \mid 1\leq x \leq k\}$. Let $\List(A,O)$ be the set of all functions $s:[1,k]\to A$, where $k \in O, k \geq 1$. We call such functions $O$-\emph{tuples}, or $O$-\emph{lists}, of elements of $\mathbb A$, and write sometimes $s=(s_1,\ldots,s_k)$. If $O = \N$ then $\List(A,\N)$ is precisely the set of all finite tuples of elements of $A$. In this case we sometimes omit $\mathbb N$ from the notation and write $\List(A)$.

Now, the list superstructure $\mathbb S(\MA,\MN)$ is defined as the three-sorted structure 
\[
\mathbb S(\mathbb A,\mathbb N) = \langle \mathbb A, \List(A),\mathbb N; t(s,a,i), l(s)\rangle,
\]
where $\mathbb N$
is the standard arithmetic, $\List(A)$ is as above, $l\colon \List(A) \to \mathbb N$ is the length function, i.e., $l(s)$ is the length $k$ of a tuple $s=(s_1, \ldots,s_{k})\in \List(A)$, and $t(x,y,z)$ is a predicate on $\List(A) \times A\times \mathbb N$ such that $t(s,a,i)$ holds in $\mathbb S(\mathbb A,\mathbb N)$ if and only if $s = (s_1, \ldots,s_{k})\in \List(A), i \in \mathbb N, 1\leq i \leq k$, and $a = s_i \in A$.

Note that usually the list superstructure $\mathbb S(\mathbb A,\mathbb N)$ comes equipped (see the papers \cite{Ashaev-Belyaev-Myasnikov:1993},\cite{Bauval:1985} \cite{KharlampovichMyasnikov:2018a} mentioned above) with one more binary operation $\frown$ in the language, which is interpreted in $\mathbb S(\mathbb A,\mathbb N)$ as the concatenation of tuples, and an extra binary predicate $\in$ on $A\times \List(A)$ such that for $a \in A$ and $s \in \List(A)$, $a \in s$ holds in $\mathbb S(\mathbb A,\mathbb N)$ if and only if $a$ is the component $s_i$ of $s$ for some $i$. However, the predicate $\in$ and the concatenation $\frown$ are $0$-definable in $\mathbb S(\mathbb A,\mathbb N)$ (with the use of $t(s,i,a)$ and $\ell(s)$), so we omit them from the language, but sometimes use them in formulas, as convenient.

\begin{remark}\label{re:two_sorts}
In the case when $\mathbb A=\mathbb N$ one can simplify things by identifying $\mathbb A$ with $\mathbb N$ in $\mathbb S(\mathbb A,\mathbb N) = \langle \mathbb A, \List(A),\mathbb N; t(s,a,i), l(s)\rangle$ and use the two sorted structure $\mathbb S(\mathbb N,\mathbb N) = \langle \mathbb N, \List(\mathbb N); t(s,a,i), l(s)\rangle$ instead of the usual three-sorted one. Note that these structures are absolutely bi-interpretable in each other, so we can interchangeably use both structures. 
\end{remark}

\subsection{Interpretation of the list superstructure in arithmetic}\label{se:list_interpretation}

To define the nonstandard list superstructure in Section~\ref{se:nonstd_list} we need to describe an interpretation of~$\mathbb S(\mathbb N,\mathbb N)$ in~$\mathbb N$.

To that end, we fix a computable enumeration of all finite tuples $\tau: \mathbb N\to \cup_{k=1}^\infty \mathbb N^k$. Denote the $i$-th component of $\tau(n)$ by $\tau(n)_i$:
\[
\tau(n)_i=a\iff \mathbb S(\mathbb N,\mathbb N)\models t(\tau(n),a,i).
\]
Then the partial function $f(n,i)=\tau(n)_i$ is computable, and therefore the set $\{(n,a,i)\mid a= f(n,i)\}\subseteq \mathbb N^3$ is recursively enumerable. Hence, there is a first order formula (in fact, an $\exists$-formula) $T_\tau(n,a,i)$ in $\mathbb N$ s.t.
\[
\tau(n)_i=a\iff \mathbb N\models T_\tau(n,a,i).
\]
Further, the length $k=\ell(\tau(n))$ of the tuple $\tau(n)$ is also computable, and therefore is given by a formula $L_\tau(n,k)$:
\[
\ell(\tau(n))=k\iff \mathbb N\models L_\tau(n,k).
\]
Intuitively, the formulas $T_\tau$ and $L_\tau$ allow one to think of each $n\in\mathbb N$ as corresponding to a tuple. In particular, $T_\tau$, $L_\tau$ have the property that in $\mathbb N$ the following take place:
\begin{align}
\tag{S1} &\label{eq:tuple1}\forall n\ \exists k\ L_\tau(n,k), \\
\tag{S2} &\label{eq:tuple2}\forall n\ \forall k\ \forall i \ [(L_\tau(n,k)\wedge i>k)\to \forall a\ \neg T_\tau(n,a,i)],\\
\tag{S3} &\label{eq:tuple3}\forall n\ \forall k\ \forall i\ [ (L_\tau(n,k)\wedge i\le k)\to \exists! a\ T_\tau(n,a,i)],\\
\tag{S4} &\label{eq:tuple4}\forall n,m\ [\forall i\ \forall a\ T_\tau(n,a,i)\leftrightarrow T_\tau(m,a,i)]\to m=n
\end{align}
With these formulas in mind, in the sequel we use expressions with $\ell$ and $t$ in first-order formulas with understanding that these expressions are to be replaced with formulas involving $L_\tau$ and $T_\tau$, respectively. For example, in place of~\eqref{eq:tuple3} we may write $\forall n\ \forall i\le \ell(\tau(n))\ \exists! a\ t(\tau(n),a,i)$.
As we mentioned above, concatenation $^\frown$ can be defined through $t$ and $\ell$, and therefore through $T$ and $L$. Namely, the following first-order formula states that for every $x,y\in\mathbb N$ there exists $z\in\mathbb N$ s.t. $\tau(z)=\tau(x)^\frown\tau(y)$:
\begin{equation}\label{eq:concat}
\begin{split}
 &\forall x,y\ \exists z\\
 &\ell(\tau(z))=\ell(\tau(x))+\ell(\tau(y))\\
 &\wedge [\forall a\ \forall i \le \ell(\tau(x))\ T_\tau(z,a,i)\leftrightarrow T_\tau(x,a,i) ]\\
 &\wedge [\forall a\ \forall \ell(\tau(x))< i \le \ell(\tau(x))+\ell(\tau(y))\\
 &T_\tau(z,a,i)\leftrightarrow T_\tau(y,a,i-\ell(\tau(x))) ].
\end{split}
\end{equation}

The next formula holds when and only when $a\in \tau(x)$:
\begin{equation}\label{eq:membership}
 \exists i\ 1\le i\le \ell(x) \wedge T(x,a,i).
\end{equation}

Notice that for our purposes it suffices to consider arithmetic enumerations (not necessarily computable ones), that is, the ones given by formulas $T$ and $L$ that satisfy~\eqref{eq:tuple1}--\eqref{eq:tuple4}. In that event, for every $n\in\mathbb N$, the truth set of $T(n,a,i)$ defines a tuple corresponding to $n$. To make sure that every tuple in $\List(N,\mathbb N)$ corresponds to some tuple defined by $T(n,a,i)$ we add an infinite system of axioms. Namely, for each $k\in\mathbb N,$ we write
\begin{equation}\label{eq:tuple7}\tag{S5}
\forall a_1,\ldots,a_k\ \exists n\ L(n,k) \wedge \bigwedge_{i=1}^k T(n,a_i,i).
\end{equation}

Note that these formulas are satisfied whenever $L=L_\tau$ and $T=T_\tau$ for some computable enumeration $\tau$ as above.

\begin{definition}\label{de:arithmetic_enumeration}
The pair of formulas $\mathcal E=(T,L)$ that satisfy~\eqref{eq:tuple1}--\eqref{eq:tuple7} are called an \emph{arithmetical enumeration}
of tuples in $\mathbb N$. For an arithmetical enumeration~$\mathcal E$, by~$\mathcal E(n)$ we denote the tuple corresponding to $n\in\mathbb N$ under~$\mathcal E$.
\end{definition}

\begin{lemma}\label{le:arithmetical_enum_welldef}
Let $\mathcal E$ and $\mathcal D$ be two arithmetical enumerations of tuples in $\mathbb N$.
There exists a definable bijection $f:\mathbb N\to\mathbb N$ s.t. for all $n_1,n_2\in\mathbb N$, $\mathcal E(n_1)=\mathcal D(n_2)$ if and only if $f(n_1)=n_2$.
\end{lemma}
\begin{proof} To prove the lemma it suffices to note that the set $\{(n_1,n_2)\in\mathbb N^2\mid \mathcal E(n_1)=\mathcal D(n_2)\}$ is definable in $\mathbb{N}$. 
\end{proof}
Now we explain how every arithmetical enumeration $\mathcal E$ defines an absolute interpretation of $\mathbb S(\mathbb N,\mathbb N)$ in $\mathbb N$,
\[
\mathbb S(\mathbb N,\mathbb N)\cong \Gamma_\mathcal E(\mathbb N).
\]

For simplicity we can represent the two-sorted structure $\mathbb S(\mathbb N,\mathbb N)$ (see Remark~\ref{re:two_sorts}) by the usual one-sorted one where the base set is a disjoint union of $\N$ and $\List(\N)$, equipped with the two unary predicates defining the subsets $\N$ and $\List(\N)$, two predicates representing the operations $+$ and $\cdot$ on $\N$, the predicate $t(s,a,i)$, and a predicate representing the length function $l(s)$. 

Now we follow Definition~\ref{de:interpretable} and notation there to describe $\Gamma_\varepsilon$. We interpret $\N \cup \List(\N)$ as subset of $\mathbb N^2$, where $\mathbb N^*=\mathbb N\times \{1\}$ and $\List(\mathbb N)^*=\mathbb N\times\{2\}$ (both are definable subsets of $\N^2$), the equality as the equivalence relation on $\mathbb N^* \cup \List(\N)^*$. The operations $+$ and $\cdot$ are interpreted naturally via the first component of $\N^*$. To interpret the predicate $t(s,a,i)$ on 
$\mathbb N^* \cup \List(\N)^*$ one uses the formula $T(n,a,i)$, and to interpret the function $l(s)$ - the formula $L(n,k)$.

Formulas~\eqref{eq:tuple1}--\eqref{eq:tuple7} guarantee that the interpretation $\langle \mathbb N^*,\List(\mathbb N)^*; t^\mathbb N,\ell^\mathbb N\rangle$ is isomorphic to $\mathbb S(\mathbb N,\mathbb N)$. Indeed, the morphism is delivered by formulas \eqref{eq:tuple1}--\eqref{eq:tuple3}, and the bijectivity---by \eqref{eq:tuple4},\eqref{eq:tuple7}. 

If we were so inclined, it would be easy to define $\mathbb N^*$ and $\List(\mathbb N)^*$ in $\mathbb N$ instead of $\mathbb N^2$, say by placing the first sort in even integers, and the second sort in odd integers.

Note that the reverse absolute interpretation $\mathbb N\cong\Delta(\mathbb S(\mathbb N,\mathbb N))$ is constructed straightforwardly by interpreting $\mathbb N$ as the sort $\mathbb N$ of $\mathbb S(\mathbb N,\mathbb N)$.
\begin{proposition}[\cite{Cooper:2017,Rogers:1967}]\label{le:list_bi-int}
In the above notation the following holds:
\begin{itemize} \item [1)] $\mathbb S(\mathbb N,\mathbb N)\cong \Gamma_\mathcal E(\mathbb N)$ and the reverse interpretation $\mathbb N\cong \Delta(\mathbb S(\mathbb N,\mathbb N))$ give an absolute bi-interpretation of $\mathbb S(\mathbb N,\mathbb N)$ and $\mathbb N$ in each other.
\item [2)] Moreover, for two arithmetical enumerations $\mathcal E,\mathcal D$ the definable function $f$ from Lemma~\ref{le:arithmetical_enum_welldef} gives an isomorphism of interpretations $\Gamma_\mathcal E$ and $\Gamma_\mathcal D$.

\end{itemize}

\end{proposition}

In the sequel we will need tuples of tuples in $\mathbb N$, that is functions $[1,k]\to \List(\mathbb N,\mathbb N)$. An arithmetical enumeration $\mathcal E$ suffices to interpret $\mathbb S(\mathbb S(\mathbb N,\mathbb N),\mathbb N)$ in $\mathbb N$. Indeed, for each $n\in\mathbb N$, we consider the corresponding tuple $\mathcal E(n)=(n_1,\ldots,n_k)\in\mathbb S(\mathbb N,\mathbb N)$. To the same $n$, we then associate the tuple of tuples $(\mathcal E(n_i),\ldots,\mathcal E(n_k))$. We note that the length of $i$th tuple $\ell(\mathcal E(n_1))$ and $j$the component of $i$th tuple $t(\mathcal E(n_i),a,j)$ are definable using the original formulas $T$ and $L$. The following statement follows from this observation.

\begin{proposition}\label{le:tuples_tuples}
Let $\mathcal E$ be an arithmetical enumeration of tuples in $\mathbb{N}$. Then 
\begin{enumerate}[(a)]
    \item For every regular interpretation $\Delta$ of a structure $\mathbb A$ in $\mathbb N$, there is a corresponding regular interpretation $\Delta_\mathcal E$ of $\mathbb S(\mathbb A,\mathbb N)$ in $\mathbb N$.
    \item If, additionally, $\Delta$ gives rise to a regular bi-interpretation, then so does $\Delta_\mathcal E$.
    \item If, additionally, $\mathcal D$ is an arithmetical enumeration of tuples, then $\Delta_\mathcal E$ and $\Delta_\mathcal D$ are definably isomorphic.
\end{enumerate}
\end{proposition}

\subsection{Nonstandard list superstructure}\label{se:nonstd_list}
Let $\mathcal E=(T,L)$ be an arithmetical enumeration of tuples in $\mathbb N$. Let $\nstdM\equiv\mathbb N$ be a nonstandard model of arithmetic. 

\begin{definition} Let $\mathcal E=(T,L)$ be an arithmetical enumeration of tuples. A function $s:[1,k]\to\nstdM$, $k\in\nstdM$, is called a \emph{nonstandard} (or \emph{definable}) tuple (list) if its graph is the truth set of $T(n,a,i)$ for some $n\in\nstdM$. By $\List_\mathcal E(\nstdM)$ we denote the the set of all such nonstandard tuples in $\nstdN$. 
\end{definition}

Note that $\List_\mathcal E(\mathbb N)=\List(\N,\mathbb N)$ for every arithmetical enumeration~$\mathcal E$. However, at least out of cardinality considerations, $\List_\mathcal E(\nstdN) \neq \List(\nstdN,\nstdN)$.

\begin{definition}
By Theorem~\ref{th:complete-scheme} a structure is elementarily equivalent to $\mathbb S(\mathbb N,\mathbb N)$ when and only when it is of the form $\Gamma_\mathcal E(\nstdM)$ for some $\nstdM \equiv \mathbb N$, where $\Gamma_\mathcal E$ is as in Section~\ref{se:list_interpretation}. We denote $\Gamma_\mathcal E(\nstdN)$ by $\mathbb S(\nstdN,\nstdN)$ and call it the \emph{nonstandard list superstructure of arithmetic $\nstdN$}.
\end{definition}
We will exploit the interpretation $\Gamma_\mathcal E$
and elementary equivalence to understand $\mathbb S(\nstdN,\nstdN)$ in more specific terms.

The structure $\mathbb S(\nstdN,\nstdN)$ is two-sorted with two underlying sets: $\nstdN$ and $\List_\mathcal E(\nstdM)$.
According to the interpretation $\Gamma_\mathcal E$, the set $\nstdN$ is interpreted as $\nstdN^* = \nstdN \times \{1\}$ and $\List_\mathcal E(\nstdM)$ as $\List_\mathcal E(\nstdM)^* = \nstdN \times \{2\}$. We can naturally identify $\nstdN^*$ and $\List_\mathcal E(\nstdM)^*$ with two disjoint copies of $\nstdN$. 

After this identification each fixed $n \in \List_\mathcal E(\nstdM)^* = \nstdN$ defines uniquely the nonstandard tuple $s:[1,k] \to \nstdN$ where $s(i)=a$ if and only if $\nstdN\models L_\mathcal E(n,k)\wedge T_\mathcal E(n,a,i)$, which is equivalent to 
\[
\mathbb S(\nstdN,\nstdN)\models [\ell(s)=k]\wedge [\forall 1\le i\le k\ t(s,s(i),i)].
\]
In this event, we denote $s=\mathcal E(n)$ and refer to it as the \emph{$n$th nonstandard tuple in $\nstdN$}. Notice that if $\nstdM$ is countable, then so is $\List_\mathcal E(\nstdM)$, unlike the set of \emph{all} tuples of elements of $\nstdN$ indexed by $1\le i\le n$, $n\in \nstdN$.

As usual, the element $a$ such that $\mathbb S(\nstdN,\nstdN)\models t(s,a,i)$ is called the $i$th term, or $i$th component, of a nonstandard tuple $s$. To indicate a nonstandard tuple with the first element $a_1$ and the last element $a_k$, $k\in\nstdN$, we write $(a_1,\ndots,a_k)$.
Concatenation of nonstandard tuples is described by formula~\eqref{eq:concat}; membership $a\in (a_1,\ldots,a_k)$ is given by formula \eqref{eq:membership}.

\begin{proposition} Let $\mathcal E$ and $\mathcal F$ be arithmetical enumerations of tuples in $\N$ and $\nstdN \equiv \N$. Then the following holds:
\begin{itemize}
\item [1)] $\List_\mathcal E(\nstdN) = \List_\mathcal F(\nstdN)$ and there is a formula $\phi(x,y)$ in the language of rings $\mathcal L_r$ such that for every $m,n \in \nstdN$ 
\[
\mathcal E(m) = \mathcal F(n) \Longleftrightarrow 
\nstdN \models \phi(m,n)
\]

\item [2)] the interpretations $\Gamma_\mathcal E (\nstdN)$ and $\Gamma_\mathcal F (\nstdN)$ are isomorphic. Moreover, there is a formula $\psi(\bar x, \bar y)$ in the language of rings $\mathcal L_r$ that defines in $\nstdN$ an isomorphism $\Gamma_\mathcal E (\nstdN) \to \Gamma_\mathcal F (\nstdN)$.
\end{itemize}
\end{proposition}
\begin{proof}
Since $\mathbb S(\nstdM,\nstdM)\equiv\mathbb S(\mathbb N,\mathbb N)$ the result follows from Lemma~\ref{le:arithmetical_enum_welldef}. 
\end{proof}
\begin{remark}
The set $\List_\mathcal E(\nstdN)$ and the interpretation $\Gamma_\mathcal E(\nstdN)$ do not depend (the latter up to a definable isomorphism) on the enumeration $\mathcal E$. We denote them 
$\List(\nstdN)$ and $\mathbb S(\nstdN,\nstdN)$ respectively. \end{remark}

\begin{remark}
The structure $\mathbb S(\nstdN,\mathbb N)$
 is \emph{not} elementarily equivalent to $\mathbb S(\mathbb N,\mathbb N)$, since such equivalence would give a way to define $\mathbb N$ in $\nstdN$. For example, $\mathbb N\subseteq \nstdN$ would be defined in $\mathbb S(\nstdN,\mathbb N)$ by the formula
\[
\mathop{\mathrm{std}}(n)=\exists \mathrm{tuple}\, s\ ( s_1=1\wedge s_{\ell(s)}=n\wedge [\forall 1\le i\le\ell(s)-1\ s_{i+1}=s_{i}+1]).
\]
The sentence $\forall n\, \mathop{\mathrm{std}}(n)$ holds in $\mathbb S(\mathbb N,\mathbb N)$ but not in $\mathbb S(\nstdN,\mathbb N)$. It follows that none of the nonstandard models $\nstdM$ have the same weak second order theory as $\mathbb N$. See also Section~\ref{se:summation_arb_superstructre} for a related discussion.
\end{remark}

\subsection{Properties of nonstandard tuples}\label{se:nonstd_tuples}

\subsubsection{Nonstandard tuples that are always there}
In this section we describe some tuples that belong to $\List(\nstdN)$ for every $\nstdN \equiv \N$.
\begin{lemma} [Definable functions] The following holds.
\begin{itemize}
\item [1)] Let $f:\N \to \N$ be an arithmetic function, i.e., there is a formula $\phi(x,y)$ of arithmetic such that $\forall m,n \in \N$ 
\[
f(m) = n \Longleftrightarrow \N \models \phi(m,n).
\]
Then for any $\nstdN \equiv \N$ and any $k \in \nstdN$ there is a nonstandard tuple $s = (s_1,\ldots,s_k) \in \List(\nstdN)$ such that $\forall i \leq k$ $\nstdN \models \phi(i,s_i)$.
\item [2)] Conversely, let $\phi(x,y)$ be a formula of arithmetic, such that for some $\nstdN \equiv \N$ and some infinite $k \in \nstdN$ there is a tuple $s \in \List(\nstdN)$ such that $l(s) = k$ and $\nstdN \models \phi(i,s_i)$ for every $i \leq k$. Then $\phi(x,y)$ defines an arithmetic function $f:\N \to \N$. 
\end{itemize}
\end{lemma}
\begin{proof}
Fix an arbitrary arithmetic enumeration $\mathcal E$ of tuples in $\N$.

For any $k \in \N$ there is a tuple $s \in \List(\N) = \List_\mathcal E(\N)$ such that $l(s) = k$ and $\forall i \leq k$ $\N \models \phi(i,s_i)$. It follows that 
\[
S(\nstdN,\nstdN) \models \forall k \in \N \ \exists s \in \List(\nstdN) [\forall i \leq k \ \phi(i,s_i)].
\]
Since 
\[
S(\nstdN,\nstdN) \cong \Gamma_\mathcal E(\nstdN) \equiv \Gamma_\mathcal E(\N) = \List_\mathcal E(\N)
\]
one has that for any $k \in \nstdN$ there is a nonstandard tuple $s \in \List_\mathcal E(\nstdN)$ such that $l(s) = k$ and $\nstdN \models \phi(i,s_i)$ for any $i \leq k$.
\end{proof}

The following sentence holds in $\mathbb S(\mathbb N,\mathbb N)$, hence in $\mathbb S(\nstdN,\nstdN)$:
\[
\forall k\ \exists s\ \forall a\ \forall i\ t(s,a,i)\iff i\le k \wedge a=i.
\]
Therefore, the function $s$ defined by $s(i)=i$ for $1\le i\le k$ is a nonstandard tuple. We denote this nonstandard tuple by $[1,k]$ or $(1,\ndots,k)$.

Similar reasoning shows that there is a ``constant'' nonstandard tuple $(1,\ldots,1)$ of length $n$ for each $n\in\nstdM$.

For each nonstandard tuple $s$ and each $n\le \ell(s)$ we can consider its restriction to $[1,n]$, or its initial segment $s|_{n}$, defined by $\ell(s|_n)=n$ and $\forall i \forall a\ t(s|_{n},a,i)\leftrightarrow i\le n\wedge t(s,a,i)$.

\subsubsection{Nonstandard permutations}\label{se:nonstd_permutations} We say that a nonstandard tuple $\sigma$ of length $n=\ell(\sigma)\in\nstdN$ is a \emph{nonstandard permutation} of $[1,n]$ if each term of $\sigma$ is at least $1$ and at most $n$, and each $1\le k\le n$ appears exactly once in $\sigma$. The following first-order formula $P(\sigma,n)$ tells if a tuple $\sigma$ is a nonstandard permutation of $[1,n]$:
\begin{equation}\label{eq:permutation}
 \begin{split}
 P(\sigma,n)&=(\ell(\sigma)=n)\\
 &\wedge \forall 1\le k\le n\ \exists! 1\le i\le n\ t(\sigma,k,i)\\
 &\wedge \forall 1\le i\le n\ \exists! 1\le k\le n\ t(\sigma,k,i).
 \end{split}
\end{equation}
Note that every nonstandard permutation gives a first-order definable bijection of $[1,n]$ onto itself.

In the sequel, we often omit the word nonstandard and simply term such $\sigma$ \emph{permutations}. In particular, note that if $\sigma$ is a permutation, it is implied that it is a nonstandard tuple.

For illustration purposes, in the next proposition we show that the behavior we expect in the finite case is observed in the nonstandard case. In particular, note that it follows from this proposition that the second or the third line in~\eqref{eq:permutation} can be omitted.

\begin{proposition}
Suppose $\sigma$ is a nonstandard tuple of length $n\in\nstdN$.
\begin{enumerate}[(a)]
    \item If every $1\le i\le n$ appears in $\sigma$, then $\sigma$ is a permutation.
    \item If all components of $\sigma$ are distinct, and every component is at least $1$ and at most $n$, then $\sigma$ is a permutation.
\end{enumerate}
\end{proposition}
\begin{proof}
To show (a), note that in the finite case the formula
 \[
 (\ell(\sigma)=n)\; \wedge\; \forall 1\le k\le n\ \exists 1\le i\le n\ t(\sigma,k,i)
 \]
 implies~\eqref{eq:permutation}. By elementary equivalence, the same takes place for $\nstdN$ and nonstandard tuples. The reasoning for (b) is similar, with formula
 \[
 \begin{split}
 (\ell(\sigma)=n)\; &\wedge\; \forall 1\le i\le n\ \exists 1\le k\le n\ t(\sigma,k,i)\\
 &\wedge\; \forall 1\le i,j\le n\ \forall k\ t(\sigma,k,i)\wedge t(\sigma,k,j)\to i=j. 
 \end{split}
 \]
\end{proof}

\subsubsection{Nonstandard sum, product, and exponentiation}\label{se:nonstandard_summation} Using the above reasoning, we can extend sum, product and other definable functions to nonstandard tuples.
\begin{lemma}\label{le:summation}
Let $f:\mathbb N\times\mathbb N\to \mathbb N$ be a definable in $\mathbb N$ function. Then there is a definable in $\mathbb S(\nstdM,\nstdM)$ function $\widetilde f:\mathbb S(\nstdM,\nstdM)\to \nstdM$ s.t. $\widetilde f((a_1,a_2))=f(a_1,a_2)$ and $\widetilde f(a_1,\ldots,a_k)=f(\widetilde f((a_1,\ldots,a_{k-1})),a_k)$ for each $k\in\nstdM, k> 2$.
\end{lemma}
\begin{proof}
For a tuple $(a_1,\ldots,a_k)\in \List(\nstdN)$, we define a new tuple $s\in \List(\nstdN)$ of length $k$ by
\begin{enumerate}
    \item $s_1=a_1$,
    \item $s_{i+1}=f(s_i,a_{i+1})$, $1\le i\le k-1$.
\end{enumerate}
Then the function $(a_1,\ldots,a_k)\mapsto s_k$ is definable and satisfies the requirement of the lemma.
\end{proof}
We think of $\widetilde f$ as an extension of $f$ to the nonstandard case. For example, this allows summation over nonstandard tuples: for a tuple $s\in \mathbb S(\nstdN,\nstdN),$ consider the tuple $r$ of the same length $\ell(s)=k$ given by
\begin{enumerate}
    \item $r_1=s_1$,
    \item $r_{i+1}=r_i+s_{i+1}$, $1\le i\le k-1$.
\end{enumerate}
Then the component $r_{k}$ gives the sum $\sum s=\sum_{i\le n} s_i$. We can similarly define $\prod s$. If $a,k\in \nstdN$ we can define $a^k$ as $\prod s$, where $\ell(s)=k$ and $s_i=a$ for each $1\le i\le k$.
It is easy to see (for example, by induction) that sum and product inherit associativity and commutativity in the following sense. If $s,s'$ are nonstandard tuples, then $\sum s+\sum s'=\sum s\frown s'$; if additionally $s'$ is a permutation of $s$, then $\sum s=\sum s'$. The same hold for the product.

As an example of commutativity and associativity, suppose $S_i$, $1\le i\le n$, are nonstandard tuples with $\ell(S_i)=m$ for each $1\le i\le n$, where $m,n\in\nstdM$, and $(i,j)\mapsto (S_i)_j$ is a definable function. Then we can observe that $\sum_i\sum_j (S_i)_j=\sum_j\sum_i (S_i)_j$. Indeed, the former sum by associativity is equal to $\sum_{k\le mn} T_{k}$, where $T_{im-m+j}=(S_i)_j$ for $1\le i\le n$, $1\le j\le m$, and the latter sum is equal to $\sum_{k\le mn}$, where $U_{jn-n+i}=(S_i)_j$ for $1\le i\le n$, $1\le j\le m$. $T$ and $U$ clearly differ by a nonstandard permutation, so by commutativity the two sums are equal.

We can also organize counting via the same approach: if $s$ is a nonstandard tuple of length $k\in\nstdN$, and $a\in \nstdN$, define $r$ as follows:
\begin{enumerate}
    \item If $s_1=a$, then $r_1=1$, and $r_1=0$ otherwise.
    \item For $1\le i\le k-1$, if $s_{i+1}=a$, then $r_{i+1}=r_{i}+1$, and $r_{i+1}=r_{i}$ otherwise.
\end{enumerate}
Then $r_k$ is the (nonstandard) number of occurrences of $a$ in $s$.
Equivalently, if we want to refer to Lemma~\ref{le:summation}, we first map $s\mapsto s'$ where $s'_i=1$ if $s_i=a$ and $s'_i=0$ otherwise (which is a definable mapping), and then take $\sum s'$.

The above lemma is a particular case of a quite general observation, which we, in fact, already have encountered in the proof of Theorem~\ref{th:well-def}. We formulate this observation precisely in Remark~\ref{re:terms_generalization} following Lemma~\ref{le:arb_summation}.\label{page:terms_generalization}

\subsection{Nonstandard tuples over a given structure and weak second-order logic.}\label{se:summation_arb_superstructre}
Consider the list superstructure $\mathbb S(\mathbb A,\mathbb N)$ over an arbitrary structure $\mathbb A=\langle A;L\rangle$. We are particularly interested in the case when $\mathbb A$ is a field, since our present investigation is enabled by interpretation of (Laurent) polynomials over a field $F$ in the list superstructure $\mathbb S(F,\mathbb N)$ (Sections~\ref{se:polynomials} and~\ref{se:laurent}).
In light of Theorem~\ref{th:complete-scheme}, we therefore are interested in structures $\mathcal M$ elementarily equivalent to $\mathbb S(F,\mathbb N)$, or, generally, $\mathbb S(\mathbb A,\mathbb N)$.

For an arbitrary structure $\mathbb A$, consider $\mathcal M\equiv \mathbb S(\mathbb A,\mathbb N)$. Such $\mathcal M$ is three-sorted, with sorts $\widetilde{\mathbb A}\equiv \mathbb A$, $\nstdM\equiv \mathbb N$, and a sort $\widetilde S$ corresponding to $\List(\mathbb A,\mathbb N)$.
We refer to these sorts as the \emph{structure sort}, \emph{number sort}, and \emph{list sort} of $\mathcal M$, respectively.
The list sort can be treated similarly to nonstandard tuples described above. Namely, for each $a\in \widetilde S$ its length $\ell(s)$ is defined, and
\[
\mathcal M\models \forall a\in\widetilde S\ \forall 1\le i\le \ell(s)\ \exists! x\in A\ t(a,x,i),
\]
that is, each element $s$ of list sort of $\mathcal M$ has a unique $i$th component for each $1\le i\le \ell(s)$. Further, each $a$ is entirely defined by its components. Indeed, the following formula is satisfied in $\mathbb S(\mathbb A,\mathbb N)$ and therefore in $\mathcal M$:
\[
\forall a,b\ a=b \leftrightarrow [\forall x\in A\ \forall i\ (i\le \ell(a)\vee i\le \ell(b)) \to (t(a,x,i)\leftrightarrow t(b,x,i))].
\]
With that in mind, we often write $a=(a_1,\ldots,a_k)$, meaning that $\ell(a)=k$ and that predicates $t(a,a_1,1)$ and $t(a,a_k,k)$ are satisfied. We also write $a_i=x$ to denote that $t(a,x,i)$ is satisfied.

Like it is the case with standard tuples, we can define the membership predicate. Namely, we write $x\in a$ to denote the formula $\exists i\ t(a,x,i)$.

In the same vein, we can define the operation of concatenation $a^\frown b$ of $a,b\in \widetilde S$ through $t,\ell$ by taking advantage of elementary equivalence. To be specific, the following formula is satisfied in $\mathbb S(\mathbb A,\mathbb N)$ and therefore in $\mathcal M$:
\[
\begin{split}
\forall a,b\ &\exists! c\ \ell(c)=\ell(a)+\ell(b)\ \wedge\\
&\forall 1\le i\le\ell(c)\ [t(c,x,i)\leftrightarrow\\
&(i\le \ell(a)\wedge t(a,x,i))\vee (i>\ell(a) \wedge t(b,x,i-\ell(a)))].
\end{split}
\]
The unique $c$ given by the above formula is denoted $a^\frown b$, like in the case of standard tuples.

We refer to elements of the list sort of $\mathcal M$ as (nonstandard) tuples, and we say that a nonstandard tuple $a$ has length $k$ if $\ell(a)=k$, and we say that $x$ is the $i$th element of a nonstandard tuple $a$ if $\mathcal M\models t(a,x,i)$. Further, we can define the (nonstandard) number of occurrences $n_x(a)$ of $x$ in $a$, by defining first $n_x(a)$ to be the number of occurrences of $x$ in $a$ if $a$ has finite length, and then defining
\[
n_x(a_1,\ldots,a_k)=\begin{cases}
 n_x(a_1,\ldots,a_{k-1}),& x\neq a_k,\\
 n_x(a_1,\ldots,a_{k-1})+1,& x=a_k.
\end{cases}
\]
We say that $b$ is a nonstandard permutation of $a$ if $\forall x\in\widetilde{\mathbb A}\ n_x(a)=n_x(b)$. The same can be defined explicitly through nonstandard permutations introduced in~Subsection~\ref{se:nonstd_permutations}. Indeed, fix an arithmetic enumeration $\mathcal E$ (see Definition~\ref{de:arithmetic_enumeration}) and consider formula analogous to~\eqref{eq:permutation} that says that $n\in\mathbb N$ encodes a permutation under $\mathcal E$:
\[
\begin{split}
P(n)=\exists k\ L(n,k) &\wedge \forall 1\le i\le k\ \exists! 1\le j\le \ T(n,i,j)\\
&\wedge \forall 1\le j'\le k\ \exists! 1\le i'\le k\ T(n,i',j').
\end{split}
\]
Since $\mathcal M\equiv \mathbb S(\mathbb A,\mathbb N)$, it follows that $b$ is a nonstandard permutation of $a$ if and only if the following formula is satisfied in $\mathcal M$:
\[
\ell(a)=\ell(b) \wedge \exists n\ [P(n)\wedge L(n,\ell(a))\wedge \forall i,j\ T(n,i,j)\to b_j=a_i],
\]
where the latter equality $b_j=a_i$ is a shorthand for $\forall x\ t(b,x,j)\leftrightarrow t(a,x,i)$.

We can extend definable binary operations (for example, addition and multiplication in case of a field) on $\mathbb A$ to the tuples in $\mathcal M\equiv \mathbb S(\mathbb A,\mathbb N)$.

\begin{lemma}\label{le:arb_summation} Let $\mathbb A=\langle A;L\rangle$ be an arbitrary structure. Let $\mathcal M\equiv \mathbb S(\mathbb A,\mathbb N)$, with list sort $\widetilde{S}$. Let $f: A\times A\to A$ be a definable in $\mathbb A$ function. Then there is a definable in $\mathcal M$ function $\widetilde f:\widetilde{S}\to A$ s.t. $\widetilde f((a_1,a_2))=f(a_1,a_2)$ and $\widetilde f(a_1,\ldots,a_k)=f(\widetilde f((a_1,\ldots,a_{k-1})),a_k)$ for each $k\in\nstdM, k> 2$.
\end{lemma}
\begin{proof}
The proof essentially repeats that of Lemma~\ref{le:summation}.

Indeed, for a tuple $(a_1,\ldots,a_k)\in \widetilde S$, we define a new tuple $s\in \widetilde S$ of length $k$ by
\begin{enumerate}
    \item $s_1=a_1$,
    \item $s_{i+1}=f(s_i,a_{i+1})$, $1\le i\le k-1$.
\end{enumerate}
Then the function $(a_1,\ldots,a_k)\mapsto s_k$ is definable and satisfies the requirement of the lemma.
\end{proof}

Similarly to and in the same sense as in Subsection~\ref{se:nonstandard_summation}, properties like associativity and commutativity of $f$ carry to $\widetilde{f}$.
Namely, let $(b_1,\ldots,b_k)\in \widetilde{S}$ be a nonstandard permutation of $(a_1,\ldots,a_k)\in \widetilde{S}$. Then, if $f$ is associative and commutative, then $\widetilde f((a_{1},\ldots,a_{k}))=\widetilde f((b_1,\ldots,b_k))$.

\begin{remark}\label{re:terms_generalization}
Lemma~\ref{le:arb_summation} is a particular application of the following general observation. Suppose we have an indexed by $i\in\mathbb N$ family of functions $f_i$ in variables $x_1,\ldots,x_{k_i}$ defined by a formula of $\mathbb S(\mathbb A,\mathbb N)$ in the following sense: there is a formula $\sigma(z,\bar x,i)$ s.t. $\mathbb S(\mathbb A,\mathbb N)\models \sigma(b,\bar a,i)$ if and only if $\ell(\bar a)=k_i$ and $b=f_i(a_1,\ldots,a_{k_i})$.
Then over any $\mathcal M\equiv \mathbb S(\mathbb A,\mathbb N)$ with numbers sort~$\nstdN$, the same formula defines generalization of $f_i$ to $i\in\nstdN$, in particular, to $i>\mathbb N$.
\end{remark}

As we mention above, we are particularly interested in $\mathcal M\equiv \mathbb S(F,\mathbb N)$ for a field~$F$. However, a specific description of $\mathcal M$ is an open question. Note that the choices of $\widetilde F$ and $\nstdM$ are not independent of each other, as the following two observations show.

\begin{remark}\label{re:embed} Note that $\widetilde S$ must contain tuples $(1,\ldots,1)$ (see Section~\ref{se:nonstd_tuples}). Taking summation over such tuple by Lemma~\ref{le:arb_summation}, we see that there must be a definable homomorphism of additive semigroups $\nstdM\to \widetilde{F}$, which is injective whenever the respective homomorphism $\mathbb N\to F$ is. For example, there is no structure $\mathcal M\equiv \mathbb S(\mathbb R,\mathbb N)$ with field sort $\widetilde F\cong \mathbb R$ and number sort $\nstdM\not\cong\mathbb N$ since these $\widetilde F$ and $\nstdN$ would not allow such injective homomorphism.
\end{remark}

\begin{remark} Conversely, if $\widetilde{F}\equiv \mathbb R$, then the image of the above homomorphism is ``archimedean'' in $\widetilde{F}$. Indeed, recall that the usual order relation is definable in $\langle \mathbb R;+,\cdot,0,1\rangle$ and note the following first-order sentence in $\mathbb S(\mathbb R,\mathbb N)$:
\[
\forall x\in\mathbb R\,\exists (1,\ldots,1)\ \sum(1,\ldots,1)>x.
\]
Here, $\exists (1,\ldots,1)$ is a shorthand for $\exists n\in\mathbb N\, \exists s\, [\ell(s)=n\, \wedge \forall i\le n\ t(s,1,i)]$. The above statement is false in the case when,
say, a field of hyperreal numbers $\widetilde{\mathbb R}$ is taken instead of $\mathbb R$. Therefore, there is no $\mathcal M\equiv \mathbb S(\mathbb R,\mathbb N)$ with field sort $\widetilde{\mathbb R}$ and number sort $\mathbb N$.
\end{remark}

\begin{remark} \label{re:any_field} If $\mathcal M\equiv \mathbb S(F,\mathbb N)$, then the field sort $\widetilde{F}$ of $\mathcal M$ is elementarily equivalent to~$F$.
However, not every field $F'\equiv F$ may appear as a field sort of some $\mathcal M\equiv \mathbb S(F,\mathbb N)$.
Indeed, consider the real algebraic closure $\mathbb R_\pi=\overline{\mathbb Q(\pi)}$ of the field $\mathbb Q(\pi)$, where $\pi$ is any transcendental number.
We have $\mathbb R_\pi\equiv \mathbb R$, since both are real closed fields.
We claim that $\mathbb R_\pi$ cannot be a field sort of such $\mathcal M$.
Indeed, consider $\mathcal M\equiv \mathbb S(F,\mathbb N)$ with field sort $\mathbb R_\pi$.
Then the numbers sort $\mathbb N'$ of $\mathcal M$ cannot be a nonstandard model of $\mathbb N$, for the same reasons as in Remark~\ref{re:embed}; nor can it be the standard natural numbers $\mathbb N$, since then existence of a transcendental element in the field sort would be possible to express by a first-order formula $\phi$ of $\mathcal M$, and then we would have $\mathcal M\models \phi$ but $\mathbb S(F,\mathbb N)\not\models \phi$.
\end{remark}

This leads to the following question.

\begin{question}\begin{enumerate}[(a)]
    \item For an infinite field $F$, describe all pairs $\widetilde F\equiv F$ and $\nstdM\equiv \mathbb N$ such that there is a subset $\widetilde S$ of functions $\nstdM\to \widetilde F$ such that the triple $\widetilde F, \widetilde S,\nstdM$ underlies a three-sorted structure $\mathcal M$ in the signature of a list superstructure of $F$ elementary equivalent to $\mathbb S(F,\mathbb N)$.
    \item Given such pair $\widetilde F,\nstdM$, describe the involved subset $\widetilde S$ in specific terms.
\end{enumerate} 
\end{question}

In the general case, even though we do not offer a specific description of $\mathcal M\equiv \mathbb S(\mathbb A,\mathbb N)$, the properties explicit in its first-order theory are enough to allow a description of elements of rings elementarily equivalent to the ring of (Laurent) polynomials. We give such descriptions in Sections~\ref{se:single_var}, \ref{se:multi_var}, \ref{se:laurent_single}, \ref{se:laurent_multi}.

\begin{remark}\label{re:F_interpretable_in_N}
One important case when the organization of such $\mathcal M$ is clear is when $\mathbb A$ is regularly invertibly interpretable in $\mathbb N$ (for example, like the field of rationals $\mathbb Q$), and therefore, so is $\mathbb S(\mathbb A,\mathbb N)$, say, as $\mathbb S(\mathbb A,\mathbb N)\cong \Gamma(\mathbb N)$.
Then by Theorem~\ref{th:complete-scheme}, every $\mathcal M\equiv \mathbb S(\mathbb A,\mathbb N)$ has the form $\mathcal M\cong \Gamma(\nstdN)$ for some $\nstdN\equiv \mathbb N$.
In this case we denote $\mathcal M=\mathbb S(\widetilde{\mathbb A},\widetilde{\mathbb N})$.
Note that the sort corresponding to $\mathbb A$ is indeed uniquely defined by $\nstdN$ by Theorem~\ref{th:well-def-inN}.
\end{remark}

\subsubsection{Weak second-order logic and nonstandard models}
By a weak second-order logic formula over a structure $\mathbb A=\langle A; L\rangle$ we mean a first-order formula over the list superstructure $\mathbb S(\mathbb A,\mathbb N)$. There are other, equivalent, descriptions, for details of which we refer the reader to~\cite{KharlampovichMyasnikovSohrabi:2021}.

For our investigation, it is a crucial observation that quantification over tuples becomes quantification over nonstandard tuples when a nonstandard model is considered. We formulate this as the following remark.

\begin{remark}\label{re:nonstd_tuples}
Suppose $\mathbb A$ is regularly bi-interpreted in a list superstructure $\mathbb S(\mathbb B,\mathbb N)$ as $\mathbb A\cong \Gamma(\mathbb S(\mathbb B,\mathbb N))$. It is then straightforward to see that $\mathbb S(\mathbb A,\mathbb N)$ is bi-interpreted in $\mathbb S(\mathbb B,\mathbb N)$ and therefore, by transitivity, in $\mathbb A$, as $\mathbb S(\mathbb A,\mathbb N)\cong \Delta(\mathbb A)$.

Let $\mathcal M\equiv \mathbb S(\mathbb B,\mathbb N)$. Consider the corresponding nonstandard models $\widetilde{\mathbb A}\cong \Gamma(\mathcal M)$ of $\mathbb A$, and $\widetilde{\mathbb S}\cong \Delta(\widetilde{\mathbb A})$ of $\mathbb S(\mathbb A,\mathbb N)$. We have $\widetilde{\mathbb A}\equiv\mathbb A$ and $\widetilde{\mathbb S}\equiv \mathbb S(\mathbb A,\mathbb N)$. Denote the list sort of $\widetilde{\mathbb S}$ by $\List_{\widetilde{\mathbb A}}$.

Suppose $\varphi$ is a second-order formula over $\mathbb A$, that is, a first-order formula over $\mathbb S(\mathbb A,\mathbb N)$. Suppose further $\varphi$ is of the form
\[
\forall P\in \List(A,\mathbb N)\ \psi(P)\mbox{\quad or\quad}\exists P\in \List(A,\mathbb N)\ \psi(P).
\]
Let $\varphi^\diamond$ be the respectitve first-order formula of $\mathbb A$ produced via the interpretation:
\[
\mathbb S(\mathbb A,\mathbb N)\models \varphi \iff \mathbb A\models \varphi^\diamond.
\]
By the elementary equivalence $\mathbb A\equiv \widetilde{\mathbb A}$, we have $\mathbb A\models \varphi^\diamond\iff \widetilde{\mathbb A}\models \varphi^\diamond$. By the elementary equivalence $\widetilde{\mathbb S}\equiv \mathbb S(\mathbb A,\mathbb N)$, we have $\mathbb S(\mathbb A,\mathbb N)\models \varphi\iff \widetilde{\mathbb S}\models \varphi$. It follows that $\widetilde{\mathbb A}\models\varphi^\diamond\iff \widetilde{\mathbb S}\models\varphi$. Note that over $\widetilde{\mathbb S}$, the formula $\varphi$ takes form
\[
\forall P\in \List_{\widetilde{\mathbb A}}\ \psi(P)\mbox{\quad or\quad}\exists P\in \List_{\widetilde{\mathbb A}}\ \psi(P).
\]
In other words, under these assumptions, a statement of the weak second-order theory of $\mathbb A$ translates to a statement of the second-order theory of $\mathbb A$ by replacing quantification over tuples with quantification over nonstandard tuples.
\end{remark}

\section{Polynomials over a field}\label{se:polynomials}
\subsection{Definition of the ring of nonstandard polynomials} Let $F$ be an infinite field. Then the polynomial ring $F[x]$ is regularly invertibly interpretable in the list superstructure $\mathbb S(F,\mathbb N)$. We refer to~\cite{KharlampovichMyasnikov:2018a} for details and overview the interpretation here, with slight adjustments to fit our purposes.

To interpret $F[x]$ in $\mathbb S(F,\mathbb N)$, a polynomial $p(x)=a_0+a_1x+\ldots+a_nx^n$ is associated to a tuple $\bar a=(a_0,\ldots,a_n)\in \List(F,\mathbb N)$. The respective ``tail of zeroes'' equivalence relation, addition, multiplication, and constants $0,1$ are then defined in the straightforward way. Note that this interpretation is absolute.

For completeness, we explicate the details. In the context of this interpretation, and indeed throughout the paper whenever tuples of coefficients are concerned, as convention, we index the entries from $0$ rather than from $1$. So, for example, the predicate $t(\bar a,a_0,0)$ holds for $\bar a=(a_0,\ldots,a_n)$ rather than the predicate $t(\bar a,a_0,1)$.

We define the ``tail of zeros'' equivalence relation $\sim$ on $\List(F,\mathbb N)$ by the formula $\varphi_\sim((a_0,\ldots,a_n),(b_0,\ldots,b_m))$:
\[
(n\le m \wedge \forall n<i\le m\ b_i=0) \vee (m\le n \wedge \forall m< i\le n\ a_i=0).
\]
This formula delivers the definability of the relation $\sim$ in Definition~\ref{de:interpretable}. We also note that each equivalence class has easily definable representative with $a_n\neq 0$ (or empty list in case of $p(x)=0$). In the subsequent formulas, we use the relation $\sim$ in formulas under the convention that a formula incorporates $\varphi_\sim(\bar a,\bar b)$ in place of $\bar a\sim \bar b$.

The constants $1$ and $0$ are interpreted by the tuple $(1)$ and the empty tuple, respectively. Addition $a(x)+b(x)=c(x)$ is then defined by the formula
\[
\begin{split}
&\varphi_+((a_0,\ldots,a_n),(b_0,\ldots,b_m),(c_0,\ldots,c_r))=\\
&\exists a'\sim a\ \exists b'\sim b\ \exists c'\sim c\ \ell(a')=\ell(b')=\ell(c') \wedge \forall 0\le i\le \ell(a')\ a'_i+b'_i=c'_i
\end{split}
\]
and multiplication by
\[
\begin{split}
&\varphi_\times((a_0,\ldots,a_n),(b_0,\ldots,b_m),(c_0,\ldots,c_r))=\\
&\exists a'\sim a\ \exists b'\sim b\ \exists c'\sim c\ 
\ell(a)+\ell(b)=\ell(c')=\ell(a')=\ell(b')\ \wedge\\
&\forall 0\le k\le \ell(c')\ \exists d\ (\ell(d)=k+1 \wedge \forall 0\le i\le k\ d_i=a_ib_{k-i} \wedge c_k=\sum d).
\end{split}
\]
The absolute interpretation $\Gamma$ of $F[x]$ in $\mathbb S(F,\mathbb N)$ is given by formulas $\varphi_\sim, \varphi_+,\varphi_\times$.

Conversely, to interpret $\mathbb S(F,\mathbb N)$ in $F[x]$, a tuple $(a_0,\ldots,a_n)$ is associated with the pair $(\sum_{i=0}^n a_it^i, t^n)\in F[x]^2$, where $t$ is any linear polynomial in $F[x]$. Observe that the set of linear polynomials in $F[x]$ is definable. Indeed, $s$ is linear if and only if $F[x]\models \varphi_{\mathrm{lin}}(s)$, where
\[
\varphi_{\mathrm{lin}}(s)=[\forall \alpha\ (\alpha=0\vee \exists\beta\ \alpha\beta=1)\to s+\alpha\ \mathrm{irreducible}].
\]
In~\cite{KharlampovichMyasnikov:2018a} it is checked that, if $F$ is infinite, this leads to a regular interpretation of $\mathbb S(F,\mathbb N)$ in $F[x]$. To keep exposition self-contained, we repeat the reasoning in Lemma~\ref{le:Fx_mutual} below.
(Note that the cited paper uses the term $0$-interpretation rather than regular interpretation in several statements, but in the terminology of the present paper, it is indeed a regular interpretation.)
Although not stated explicitly in~\cite{KharlampovichMyasnikov:2018a}, the two described interpretations give rise to a bi-interpretation between $F[x]$ and $\mathbb S(F,\mathbb N)$.
Largely repeating the exposition in~\cite{KharlampovichMyasnikov:2018a}, we verify both regular interpretation and bi-interpretation here.

We start by observing the set of pairs $(s,s^n)$ for linear $s$ is definable in $F[x]$. Recall that $\Irr$ denotes the definable set of irreducible polynomials.

\begin{lemma}[\cite{KharlampovichMyasnikov:2018a}]\label{le:pair_t_tn} Let $F$ be a field. There is a formula $\Phi_{\mathrm{LP}}(s,u)$ of $F[x]$ that defines the set $\{(s,s^n): \deg s=1, n\in\mathbb N\}$ in $F[x]$.
\end{lemma}
\begin{proof}
Observe that the formula $\varphi_1(p,q)=\forall r\in \Irr\ r\mid p \leftrightarrow r\mid q$
defines the set of pairs of polynomials $(p,q)$ with the same irreducible divisors. Therefore, $F[x]\models \varphi_1(s,p)\wedge \varphi_{\mathrm{lin}}(s)$ if and only if $p=\alpha s^n$, $\alpha \in F, n\in\mathbb N$, $n\ge 1$. Finally, note that $s-1\mid \alpha s^n-1$ if and only if $\alpha = 1$. Indeed, $\alpha s^n-1=\alpha (s^n-1)+\alpha-1=(s-1)p' + \alpha-1$ for some polynomial $p'$, from which the claim follows. Now we see that the set in question is the truth set of the formula $\varphi_{\mathrm{lin}}(s)\wedge [p=1\vee (\varphi_1(s,p)\wedge s-1\mid p-1)]$.
\end{proof}

\begin{lemma}[\cite{KharlampovichMyasnikov:2018a}]\label{le:deg_and_value} Let $F$ be an infinite field.
\begin{enumerate}[(a)]
    \item There is a formula $\varphi_{\mathrm{ev}}(s,\alpha,p,\beta)$ s.t. $F[x]\models \varphi_{\mathrm{ev}}(s,\alpha,p,\beta)$ if and only if $s$ is linear and $s=\alpha$ implies $p=\beta$.
    \item \label{le:deg_and_value_b}There is a formula $\varphi_{\deg}(p,u)$ s.t. $F[x]\models \varphi_{\deg}(p,u)$ if and only if $\deg p\le \deg u$ and $u$ is a power of a linear polynomial.
    \item There is a formula $\varphi_{\mathrm{lower}}(p,p')$ s.t. $F[x]\models \varphi_{\mathrm{lower}}(p,p')$ if and only if $\deg p<\deg p'$.
\end{enumerate}
\end{lemma}
\begin{proof}
\begin{enumerate}[(a)]
    \item The required relation is defined by the formula $\varphi_{\mathrm{lin}}(s)\wedge s-\alpha\mid p-\beta$.
    \item Define $\deg p\le \deg u$, where $u=s^n$ is a power of a linear polynomial $s$. Observe that $p$ has degree at most $n$ when and only when $s^n p(\tfrac{1}{s})$ is a polynomial and that $p(\tfrac1s)$ takes value $\beta$ at $s=\alpha$ if and only if $p(s)$ takes value $\beta$ at $s=\tfrac1\alpha$. Lastly, note that over an infinite field, two rational functions coincide when and only when all their values at equal points are equal. Therefore, the required relation is defined by the formula
    \[
    \begin{split}
    \exists s, q\ \Phi_{\mathrm{LP}}(s,u) \wedge \forall \alpha, \beta,\gamma\in F\ (s-\tfrac1\alpha\mid p-\beta \wedge s-\alpha\mid u-\gamma)\to s-\alpha\mid q-\beta\gamma.
    \end{split}
    \]
    The above is the required formula $\varphi_{\deg}(p,u)$.
    \item Now, $\deg p<\deg p'$ if and only if $F[x]\models \exists u\ \varphi_{\deg}(p,u)\wedge \neg \varphi_{\deg}(p',u)$.
\end{enumerate}
\end{proof}

\begin{lemma}[\cite{KharlampovichMyasnikov:2018a}]\label{le:Fx_mutual} Let $F$ be an infinite field. Then $F[x]$ and $\mathbb S(F,\mathbb N)$ are mutually regularly interpretable in each other.
\end{lemma}
\begin{proof}
We skip the straightforward verification that $\Gamma$ given by formulas $\varphi_\sim, \varphi_+,\varphi_\times$ introduced above is indeed an absolute interpretation of $F[x]$ in $\mathbb S(F,\mathbb N)$.

We can check that $\mathbb S(F,\mathbb N)$ is regularly interpreted in $F[x]$ whenever $F$ is an infinite field. As we mention above, the interpretation uses a parameter $s\in F[x]$ s.t. $F[x]\models \varphi_{\mathrm{lin}}(s)$. To interpret the sort $\mathbb N$, it suffices to interpret $\mathbb N$ with language $\langle 0,1,+,\mathrel{|}\rangle$~\cite[Section 4b]{Robinson:1951}. Each $n\in\mathbb N$ is interpreted as $s^n$, $0$ and $1$ are interpreted as $1$ and $t$, respectively. Addition is interpreted via $n_1+n_2=n_3 \iff s^{n_1}s^{n_2}=s^{n_3}$. Divisibility is interpreted via $n_1\mid n_2\iff s^{n_1}-1\mid s^{n_2}-1$.

The field sort $F$ is interpreted straightforwardly as a definable subset of $F[x]$. For the list sort and the list predicates, we interpret $\bar p=(p_0,\ldots,p_n)\in \List(F,\mathbb N)$ as $(\sum_{i=0}^n p_i s^i, s^n)$. The predicate $t(\bar p,a,k)$ is interpreted by the following formula $\varphi_t(p,u,\beta,v)$ with parameter $s$, which essentially extracts the coefficient $\beta=p_k$ in front of $s^k=v$ in $p=\sum_{i=1}^n p_is^i$ by saying that $p=p'+\beta s^k +s^{k+1}p''$, where $\deg p'<k$:
\[
\begin{split}
&\Phi_{\mathrm{LP}}(s,u)\wedge \Phi_{\mathrm{LP}}(s,v)\wedge \varphi_{\mathrm{deg}}(p,u)\wedge \varphi_{\mathrm{deg}}(v,u)\\
&\wedge \exists p',p''\ [\varphi_{\mathrm{lower}}(p',v) \wedge p=p'+\beta v+svp''],
\end{split}
\]
where the participating formula are as in Lemmas~\ref{le:pair_t_tn} and ~\ref{le:deg_and_value}.
Interpretation of length predicate $\ell$ is similarly straightforward using the same lemmas.
\end{proof}

\begin{theorem}[\cite{KharlampovichMyasnikov:2018a}]\label{th:Fx_SFN_biint}
Let $F$ be an infinite field.
Then $F[x]$ is regularly bi-interpretable with $\mathbb S(F,\mathbb N)$.
\end{theorem}
\begin{proof}
Denote the above interpretation of $F[x]$ in $\mathbb S(F,\mathbb N)$ by $F[x]\cong \Gamma(\mathbb S(F,\mathbb N))$, and the reverse interpretation by $\mathbb S(F,\mathbb N)\cong \Delta(F[x],\varphi_{\mathrm{lin}})$. By Lemma~\ref{le:Fx_mutual}, we only need to check the definability of isomorphisms.

To show that an isomorphism $F[x]\cong \Gamma\circ \Delta(F[x],\varphi_{\mathrm{lin}})$ is definable, we note that if $p(x)=a_0+\ldots+a_nx^n$, then the polynomial $p$ is interpreted as $(a_0,\ldots,a_n)$. The latter tuple, in turn, is interpreted as the pair $(p(s),s^n)$, where $s$ is linear. To deliver the conditions of Definition~\ref{def:reg}, we note that the mapping $(p(s),s^n)\mapsto p(s)$ serves as the required isomorphism. Therefore, it suffices to define by a formula the subset $\{(p(s),s^n): n\ge \deg p\}\subseteq F[x]^2$. By Lemma~\ref{le:deg_and_value}(\ref{le:deg_and_value_b}), this subset is defined by $\Phi_{\mathrm{LP}}(s,u)\wedge \varphi_{\deg}(p,u)$.

Consider the reverse direction $\mathbb S(F,\mathbb N)\cong \Delta\circ \Gamma(\mathbb S(F,\mathbb N)),\varphi_{\mathrm{lin}}^*)$ (see Definition~\ref{def:reg}.
A tuple $(a_0,\ldots,a_n)$ is interpreted as $(a_0+\ldots+a_ns^n,s^n)$ in $F[x]$. The latter pair is then interpreted in $\mathbb S(F,\mathbb N)$ as $((b_0,\ldots,b_n),(c_0,\ldots,c_n))$, where $a_0+\ldots+a_ns^n=b_0+\ldots+b_nx^n$ and $s^n=c_0+\ldots+c_nx^n$. Note that $\varphi^*$ defines the set of tuples $\{(l_0,l_1)\mid l_1\neq 0\}$.
Then the formula that applies ``reverse change of variable'' $x=\tfrac{1}{l_1}(s-l_0)$ to $b$ delivers the required isomorphism.
\end{proof}

This allows to take advantage of Theorem~\ref{th:complete-scheme}.

\begin{corollary}\label{co:Fx_elem_equiv_list}
Let $F$ be an infinite field. Let $\Gamma_0$ be the code of a regular invertible interpretation of $F[x]$ in $\mathbb S(F,\mathbb N)$, $F[x]\cong \Gamma_0(\mathbb S(F,\mathbb N))$. Then every ring elementarily equivalent to $F[x]$ has the form $\Gamma_0(\mathcal M)$, where $\mathcal M\equiv \mathbb S(F,\mathbb N)$.
\end{corollary}

Suppose further that $F$ is regularly bi-interpretable with $\mathbb N$ (or regularly invertibly interpretable in $\mathbb N$), thus necessarily infinite. Then $\mathbb S(F,\mathbb N)$ is regularly bi-interpretable with $\mathbb S(\mathbb N,\mathbb N)$ (resp., regularly invertibly interpretable in $\mathbb S(\mathbb N,\mathbb N)$), which leads to a regular bi-interpretation of $F[x]$ with $\mathbb S(\mathbb N,\mathbb N)$ (resp., regular invertible interpretation in $\mathbb S(\mathbb N,\mathbb N)$) and therefore with $\mathbb N$. We record this as the following statement.

\begin{corollary}
\label{co:bi-int_Fx_N}
Let $F$ be a field regularly bi-interpretable with $\mathbb N$ (regularly invertibly interpretable in $\mathbb N$). Then the ring of polynomials in a single variable $\langle F[x];+,\cdot,0,1\rangle$ is regularly bi-interpretable with $\mathbb N$ (resp., regularly invertibly interpretable in $\mathbb N$), $F[x]\cong \Gamma(\mathbb N)$.
\end{corollary}

\begin{corollary}\label{co:Fx_elem_equiv}
Let $F$ be a field regularly invertibly interpretable in $\mathbb N$. Let $\Gamma$ be the code of a regular invertible interpretation of $F[x]$ in $\mathbb N$, $F[x]\cong \Gamma(\mathbb N)$. Then every ring elementarily equivalent to $F[x]$ has the form $\Gamma(\nstdM)$, where $\nstdM\equiv \mathbb N$.
\end{corollary}

Consider a nonstandard model of $\mathbb N$, denoted $\nstdN$. With the above statement in mind, we can shed more light on the structure $\Gamma(\nstdM)$, which we naturally call the \emph{ring of nonstandard polynomials} in a single variable, and denote by 
$\nonstandardmodel{F[x]}{\nstdN}$. Notice that $\nonstandardmodel{F[x]}{\nstdN}$ is independent of the choice of a particular interpretation $\Gamma$ by Theorems~\ref{th:well-def-inN}. 
\begin{definition}\label{de:nonstandard_model}
Let $\mathbb A$ be a structure regularly interpretable in a structure $\mathbb B$ as $\mathbb A\cong \Gamma(\mathbb B)$. Let $\widetilde{\mathbb B}\equiv\mathbb B$. We say that $\Gamma(\widetilde{\mathbb B})$ is a \emph{nonstandard model} of $\mathbb A$ and denote it by $\nonstandardmodel{\mathbb A}{\widetilde{\mathbb B}}$.
\end{definition}
Note that for plain interpretations in the list superstructure, nonstandard models are independent of choice of interpretation, in the sense of Theorem~\ref{th:well-def}.

In this notation, $\Gamma_0(\mathcal M)$ in Corollary~\ref{co:Fx_elem_equiv_list} can be denoted 
$\nonstandardmodel{F[x]}{\widetilde{\mathbb S(F,\mathbb N)}}$ or $\nonstandardmodel{F[x]}{\mathcal M}, \mathcal M\equiv \mathbb S(F,\mathbb N)$,
and $\Gamma(\nstdN)$ in Corollary~\ref{co:Fx_elem_equiv} can be denoted $\nonstandardmodel{F[x]}{\nstdN}$.
Taking advantage of the latter notation, we restate Corollaries~\ref{co:Fx_elem_equiv_list} and~\ref{co:Fx_elem_equiv} as follows.
\begin{corollary}\label{th:Fx_elem_equiv_restate}
Let $F$ be a field regularly invertibly interpretable in $\mathbb N$. Then every ring elementarily equivalent to $F[x]$ has the form $\nonstandardmodel{F[x]}{\nstdN}$, where $\nstdN\equiv \mathbb N$.
\end{corollary}
\begin{corollary}\label{co:Fx_elem_equiv_list_restate}
Let $F$ be an infinite field. Then every ring elementarily equivalent to $F[x]$ has the form $\nonstandardmodel{F[x]}{\mathcal M}$, where $\mathcal M\equiv\mathbb S(F,\mathbb N)$.
\end{corollary}

In Section~\ref{se:single_var}, we exploit the independence on the choice of an interpretation to study algebraic properties of $\nonstandardmodel{F[x]}{\nstdN}$ and $\nonstandardmodel{F[x]}{\mathcal M}$. In the case of a field regularly invertibly interpretable in $\mathbb N$, we will typically focus on the straightforward interpretation of $F[x]$ in~$\mathbb S(F,\mathbb N)$ (with the implication that the latter is interpreted in~$\mathbb S(\mathbb N,\mathbb N)$ and then in~$\mathbb N$).

Finally, we note here that in the sequel, we exhibit statements similar to the four statements Theorems~\ref{th:Fx_SFN_biint}--Corollary~\ref{co:Fx_elem_equiv} for other objects we study: polynomials in several variables in Section~\ref{se:multi_var} and Laurent polynomials in Sections~\ref{se:laurent_single} and~\ref{se:laurent_multi}).

\subsection{Nonstandard polynomials in single variable}\label{se:single_var}
In this section we describe and study properties of the elements of the ring of nonstandard polynomials $\nonstandardmodel{F[x]}{\nstdM}$ and, more generally, $\nonstandardmodel{F[x]}{\mathcal M}$, $\mathcal M\equiv \mathbb S(F,\mathbb N)$. These elements are called \emph{nonstandard polynomials}.

Let $F[x]$ be interpreted as $F[x]\cong \Gamma(\mathbb S(F,\mathbb N))$. Let $\mathcal M\equiv \mathbb S(F,\mathbb N)$; in particular, $\mathcal M=\mathbb S(\widetilde F,\widetilde{\mathbb N})$ if $F$ is regularly invertibly interpretable in $\mathbb N$.
From Theorems~\ref{th:complete-scheme}, \ref{th:well-def-inN}, \ref{th:well-def}
we get that $\nonstandardmodel{F[x]}{\mathcal M}\cong \Gamma(\mathcal M)$, or, in the case of a regularly invertibly interpretable in $\mathbb N$ field $F$, $\nonstandardmodel{F[x]}{\nstdN}\cong \Gamma(\mathbb S(\widetilde F,\nstdM))$. For convenience of notation, in the context of this interpretation we index entries of a tuple from $0$ to $n$, as opposed to indexing from $1$ to $n$.

\noindent\textsc{Convention.} In what follows, we will denote (nonstandard) polynomials by small Latin letters like $p$; coefficients of $p(x)$ by $p_i$; tuples of polynomials by capital Latin letters like $P$; $i$th polynomial in a tuple $P$ by $P_i$;
constant elements of $\nonstandardmodel{F[x]}{\mathcal M}$ by Greek letters like $\alpha$ or $\beta$.

\noindent\textsc{Convention.} In our first-order formulas over $\nonstandardmodel{F[x]}{\nstdN}$ we will use quantification over constants, since the set of constants is definable: $p\in \nonstandardmodel{F[x]}{\mathcal M}$ is constant if and only if $\nonstandardmodel{F[x]}{\mathcal M}\models p=0\vee \exists q\ pq=1$. To indicate such quantification, we will write $\forall \alpha$, $\exists\beta$, etc.

\noindent\textsc{Convention.} In our first-order formulas over $\nonstandardmodel{F[x]}{\mathcal M}$ we will use quantification over (nonstandard) tuples, since the respective list superstructure is interpretable in $\nonstandardmodel{F[x]}{\mathcal M}$. To indicate that, we will write $\forall P$, $\exists Q$, etc. Note that when we pass from $F[x]$ to $\nonstandardmodel{F[x]}{\mathcal M}$, quantification over tuples becomes quantification over nonstandard tuples, as explained in Remark~\ref{re:nonstd_tuples}.

With the above interpretation in mind, for $n\in\nstdN$, a polynomial $p(x)\in \nonstandardmodel{F[x]}{\mathcal M}$ of degree $n$ can be viewed as an element $p$ of the list sort of $\mathcal M$ (of $\List(\widetilde F,\nstdM)$ in the case of a regularly invertibly interpretable in $\mathbb N$ field $F$), with $\ell(p)=n+1$ and the relation $t(p,p_{i},i)$, $0\le i\le n$, indicating the coefficient $p_{i}$ in front of $x^{i}$. Moreover, any tuple with $p'$ with $\ell(p')\ge n+1$ with initial segment $p$ and $t(p',0,i)$ for all $i> n$, also corresponds to $p(x)$.

Notice that if $p,q\in F[x]$ with $\deg p=n$ and $\deg q=m$, their sum and product in $F[x]$ correspond to the tuples $p+q$, $pq$ given by the following formulas:
\begin{itemize}
    \item $\ell(p+q)=1+\max\{\deg p, \deg q\}$,
    $(p+q)_i=
    \begin{cases}
    p_i+q_i, &0\le i\le \min\{\deg p, \deg q\},\\
    q_i, &\deg p<i\le \deg q,\\
    p_i, &\deg q<i\le \deg p.
    \end{cases}
    $
    Notice that it may be the case that $\deg(p+q)+1<\ell(p+q)$.
    \item $\ell(pq)=\deg p+\deg q$. For simplicity of the subsequent formula, we assume that $p$ and $q$ are represented by tuples $p,q$ with $\ell(p)=\ell(q)=m+n+1$. The $i$th coefficient of $pq$ is then defined by
    \[
    (pq)_i=\sum_{\nu=0}^{i}p_{\nu}q_{i-\nu}.
    \]
\end{itemize}
Since both of the above formulas can be expressed as the first-order formulas of $\mathbb S(F,\mathbb N)$, they also define addition and multiplication in $\nonstandardmodel{F[x]}{\mathcal M}$. Inspecting the respective first-order sentences, we see that $\nonstandardmodel{F[x]}{\mathcal M}$ is an associative commutative ring with unity.

Further, whenever $P=(P_1(x),\ldots,P_n(x))$ is a nonstandard tuple of nonstandard polynomials (in the sense of Subsection~\ref{se:summation_arb_superstructre}, ultimately meaning that $(P_i)_j$ can be defined by first order formulas),
we can define the sum and the product of the tuple, $\sum_{i\le n} P_i(x)$ and $\prod_{i\le n} P_i(x)$, via the procedure shown Subsections~\ref{se:nonstandard_summation} and~\ref{se:summation_arb_superstructre}. Indeed, to define, say, the sum, we consider the tuple $(S_1(x),\ldots,S_n(x))$ given by
\begin{enumerate}
    \item $S_1=P_1$
    \item $S_{i+1}=S_i+P_{i+1}$ for $1<i\le n$.
\end{enumerate}
By induction we see that $\prod_{i\le n}x^1=x^n$, where $x$ stands for the monomial $x^1$.

We can similarly define cooordinate-wise sum of a nonstandard tuple of nonstandard tuples $(S_1,\ldots,S_n)$, assuming all tuples are extended by zeros to length $\max\ell(s_i)$ (notice that this maximum is well-defined). In that case, inspecting each coordinate, we get that
\[
\left(\sum\nolimits_{i\le n} S_i\right)_j=\sum_{i\le n} (S_i)_j.
\]
In particular,
\[
\sum_{i\le n}(0,\ldots,0,a_i,0,\ldots,0)=(a_1,\ldots,a_n).
\]
It follows that the polynomial corresponding to the nonstandard tuple $(a_0,\ldots,a_n)$ of coefficients can be expressed as
\begin{equation}\label{eq:nstd_poly_desc}
\sum_{i\le n} a_i\prod_{j\le i} x=\sum_{i\le n} a_ix^i.
\end{equation}
Further, notice that for a given polynomial the uniqueness of the ``coefficient tuple'' $(a_0,\ldots,a_n)$ (with nonzero $a_n$) is a first-order sentence. It follows that the coefficient tuple is unique for a given polynomial, assuming nonzero leading coefficient.

For two polynomials written as $p=\sum_{i\le n} a_ix^i$ and $q=\sum_{i\le m} b_ix^i$, their sum and product defined earlier take the familiar form
\begin{equation}\label{eq:nstd_poly_sum_product}
p+q=\sum_{k\le\max\{n,m\}} (a_k+b_k)x^k,\qquad pq=\sum_{k\le n+m} \left(\sum_{i+j=k}a_ib_j\right)x^k,
\end{equation}
where for convenience of notation (and consistently with the equivalence relation in the considered interpretation) we set $a_i=0$ for $i>n$ and $b_i=0$ for $i>m$.

\subsubsection{Universal property}\label{se:single_var_universal} To justify using the term \emph{polynomials} for elements of~$\nonstandardmodel{F[x]}{\mathcal M}$, we show that the latter has a certain universal property, like the algebra of polynomials does.
To this end, in the following lemma we observe that a homomorphism of nonstandard algebras is nonstandardly additive and multiplicative. We call such homomorphisms (namely, those that satisfy formula~\eqref{eq:nstd_homomorphism} below) \emph{nonstandard}.

\begin{lemma}\label{le:nonstd_homomorphism_general}
Suppose $\mathbb A_1,\mathbb A_2$ are two associative algebras over a field $F$ with regular interpretations in the list superstructure $\mathbb S(F, \mathbb N)$, $\mathbb A_i\cong \Delta_i(\mathbb S(F, \mathbb N))$, $i=1,2$. Suppose that $\mathbb A_1$ is finitely generated as an algebra over $F$. Let $\mathcal M\equiv\mathbb S(F, \mathbb N)$ and let $\widetilde{\mathbb A}_i\cong \nonstandardmodel{\mathbb A_i}{\mathcal M}$. If $\varphi:\widetilde{\mathbb A}_1\to\widetilde{\mathbb A}_2$ is an algebra homomorphism, then \begin{equation}\label{eq:nstd_homomorphism}
\varphi\left(\sum_{i=1}^n a_i\right)=\sum_{i=1}^n \varphi(a_i),\qquad \varphi\left(\prod_{i=1}^n a_i\right)=\prod_{i=1}^n \varphi(a_i)
\end{equation}
for every nonstandard tuple $(a_1,\ldots,a_n)$ of elements of $\widetilde{\mathbb A}_1$.
\end{lemma}
\begin{proof}
By enumerating words in $x_1,\ldots,x_k$, we can write a first order formula $\Phi_{\mathrm{hom}}(x_1,\ldots,x_k,y_1,\ldots,y_k)$ of $\mathbb S(F, \mathbb N)$ s.t. $\mathbb S(F, \mathbb N)\models \Phi_{\mathrm{hom}}(x_1,\ldots,x_k,y_1,\ldots,y_k)$ if and only if the mapping $\mu_1(x_i)\mapsto \mu_2(y_i)$ extends to an algebra homomorphism
\[
\langle \mu_1(x_1),\ldots,\mu_1(x_k)\rangle\to \Delta_2(\mathbb S(F, \mathbb N))\cong\mathbb A_2,
\]
where $\langle \ldots \rangle$ denotes the subalgebra generated by the listed elements.

We similarly can write a first order formula $\Phi_{\mathrm{tuple}}(x_1,\ldots,x_k,y_1,\ldots,y_k)$ of $\mathbb S(F, \mathbb N)$ s.t. $\mathbb S(F, \mathbb N)\models \Phi_{\mathrm{tuple}}(x_1,\ldots,x_k,y_1,\ldots,y_k)$ if and only if the mapping $\mu_1(x_i)\mapsto \mu_2(y_i)$ extends to a mapping for which the equations~\eqref{eq:nstd_homomorphism} are satisfied.

By induction, the following formula $\Phi$ is satisfied in $\mathbb S(F, \mathbb N)$ (with an abuse of notation, we use use $a_i\in \mathbb A_1,b_i\in \mathbb B_2$ instead $x_i,y_i\in \List(F,\mathbb N)$):
\[
\begin{split}
\forall a_1,&\ldots,a_k\in \mathbb A_1\ \forall b_1,\ldots,b_k\in\mathbb A_2\\
&\Phi_{\mathrm{hom}}(a_1,\ldots,a_k,b_1,\ldots,b_k)\to \Phi_{\mathrm{tuple}}(a_1,\ldots,a_k,b_1,\ldots,b_k).
\end{split}
\]
Now, since $\mathcal M\equiv \mathbb S(F, \mathbb N)$, we have $\mathcal M\models \Phi$, from which the required statmement follows.
\end{proof}

Now we consider the regular interpretation of $F[x]$ in the list superstructure $\mathbb S(F,\mathbb N)$, $F[x]\cong \Gamma(\mathbb S(F,\mathbb N))$.
Let $\mathcal M\equiv \mathbb S(F,\mathbb N)$, with field sort $\widetilde F$, list sort $\List(\widetilde F,\nstdN)$, and number sort $\nstdN$.
Consider the ring $\nonstandardmodel{F[x]}{{\mathcal M}}$. Note that if $F$ is regularly invertibly interpretable in $\mathbb N$, then $\nonstandardmodel{F[x]}{{\mathcal M}}$ is $\nonstandardmodel{F[x]}{\nstdN}$ for an appropriate $\nstdN$.

Further, consider an associative algebra with unity $\mathbb A$ over a field $F$ with a regular interpretation in the list superstructure $\mathbb S(F,\mathbb N)$, $\mathbb A\cong \Delta(\mathbb S(F,\mathbb N))$.
For the same $\mathcal M\equiv \mathbb S(F,\mathbb N)$, consider the algebra arising from the interpretation $\Delta$, that is, $\widetilde{\mathbb A}\cong \Delta(\mathcal M)=\nonstandardmodel{\mathbb A}{\mathcal M}$.

\begin{theorem}\label{th:universal_nstd_single}
Let $F$ be an infinite field, $F[x]$ the ring of polynomials, and $\mathbb A$ an arbitrary $F$-algebra with a regular interpretation in $\mathbb S(F, \mathbb N)$. Let $\mathcal M\equiv\mathbb S(F, \mathbb N)$ and $\widetilde{\mathbb A}=\nonstandardmodel{\mathbb A}{\mathcal M}$. Then every mapping $x\mapsto a\in \widetilde{\mathbb A}$ extends uniquely to an algebra homomorphism $\nonstandardmodel{F[x]}{\mathcal M}\to\widetilde{\mathbb A}$.
\end{theorem}
\begin{proof}
By Lemma~\ref{le:nonstd_homomorphism_general}, such homomorphism $\varphi$ must be nonstandardly additive and multiplicative. It follows that uniquely
\[
\varphi\left(\sum_{i=0}^np_ix^i\right)=\sum_{i=0}^np_ia^i.
\]
Then
\[
\begin{split}
\varphi(pq)&=\varphi\left(\left(\sum_{i=0}^np_ix^i\right)\left(\sum_{j=0}^mq_jx^j\right)\right)
=\varphi\left(\sum_{k=1}^{n+m}x^k\sum_{i+j=k}p_iq_j\right)\\
&=\sum_{k=1}^{n+m}a^k\sum_{i+j=k}p_iq_j=
\left(\sum_{i=0}^np_ia^i\right)\left(\sum_{j=0}^mq_ja^j\right)
=\varphi(p)\varphi(q).
\end{split}
\]
By a similar, simpler check we see $\varphi(p+q)=\varphi(p)+\varphi(q)$. This shows that the above defined $\varphi$ is indeed a homomorphism.
\end{proof}
By Lemma~\ref{le:nonstd_homomorphism_general} it follows that under the same assumptions, every mapping $x\mapsto a\in \widetilde{\mathbb A}$ extends uniquely to a nonstandard algebra homomorphism $\nonstandardmodel{F[x]}{\mathcal M}\to\widetilde{\mathbb A}$.

Appropriate universal properties similarly hold for polynomials in several variables (Section~\ref{se:multi_var}), and for Laurent polynomials (Section~\ref{se:laurent_single}, \ref{se:laurent_multi}).

\subsubsection{Algebraic properties} Now we go over some features of nonstandard polynomials. The general sentiment is that we can treat nonstandard polynomials like we would the usual kind, as long the respective features can be expressed in the first-order language. However, the properties whose expression requires finite tuples of polynomials (that is, requires weak second-order language) only carry over to the nonstandard case if ``finite'' is replaced with ``nonstandard,'' for example, like in Proposition~\ref{le:ufd} below.

In the first-order formulas below we use expressions that, while do not belong in the language, are straightforwardly definable. For example, ``$d| p$'' below can be expressed by the formula $\exists q\, p=qd$. Of particular importance is the observation
made in Remark~\ref{re:nonstd_tuples}, applied to the case of second-order formula over $F[x]$; that is, that quantification over tuples becomes quantification over nonstandard tuples when we consider $\nonstandardmodel{F[x]}{\mathcal M}$.

For the remainder of this section, we consider the ring of nonstandard polynomials $\nonstandardmodel{F[x]}{\mathcal M}$, $\mathcal M\equiv \mathbb S(F,\mathbb N)$. By $\nstdN$ we denote the numbers sort of $\mathcal M$. Recall that by Remark~\ref{re:F_interpretable_in_N} in the case when $F$ is invertibly interpretable in $\mathbb N$, such $\mathcal M$ is uniquely defined by $\nstdN$ and can be explicitly described as $\mathcal M\cong \mathbb S(\widetilde F,\nstdN)$.
\begin{lemma}\label{le:remainder}
For every $p,q\in \nonstandardmodel{F[x]}{\mathcal M}$, $q\neq 0$, there is $s,r\in \nonstandardmodel{F[x]}{\nstdN}$ s.t. $p=qs+r$ with $\deg r<\deg q$.
\end{lemma}
\begin{proof}
The relation $\deg r<\deg q$ is definable by Lemma~\ref{le:deg_and_value}. Then the statement follows since $\nonstandardmodel{F[x]}{\mathcal M}\equiv F[x]$.
\end{proof}

\begin{lemma}\label{le:bezout}
If for $p,q\in \nonstandardmodel{F[x]}{\mathcal M}$ we have $\gcd(p,q)=1$, then there are polynomials $a,b\in \nonstandardmodel{F[x]}{\mathcal M}$ s.t. $a p+b q = 1$.
\end{lemma}
\begin{proof}
This can be expressed by a first order formula:
\[
\forall p,q\ \left(\forall d\ d|p \vee d|q \rightarrow d\ \mathrm{invertible}\right)\rightarrow \exists a, b\ a p+b q = 1.
\]
\end{proof}

Nonstandard analogue of Lemma~\ref{le:bezout} takes place.
\begin{lemma}\label{le:nstd_bezout}
If for a nonstandard tuple of polynomials $P$ we have $\gcd(P)=1$, then there is a nonstandard tuple of polynomials $A$ s.t. $\sum A P= 1$.
\end{lemma}
\begin{proof}
This can be expressed by a first order formula:
\[
\forall P\ \left(\forall d\ (\forall i\le\ell(p)\ d|P_i) \rightarrow d\ \mathrm{invertible}\right)\rightarrow \exists A\ \ell(\alpha)=\ell(p)\wedge \sum AP= 1.
\]
\end{proof}

\begin{lemma}
The following takes place for all polynomials $p,q,r\in \nonstandardmodel{F[x]}{\mathcal M}$:
\begin{enumerate}[(a)]
    \item If $p|qr$, $\gcd(p,q)=1$, then $p|r$.
    \item If $p|r$, $q|r$, and $\gcd(p,q)=1$, then $pq|r$.
\end{enumerate}
\end{lemma}
\begin{proof}
As usual, follows by Lemma~\ref{le:bezout}.
\end{proof}

\begin{lemma}\label{le:principal}
Let $I\lhd \nonstandardmodel{F[x]}{\mathcal M}$ be finitely generated. Then $I$ is principal.
\end{lemma}
\begin{proof}
This can be expressed by a family of first-order formulas $(n\in\mathbb N)$:
\[
\forall P_1,\ldots,P_n\ \exists q\ \forall T_1,\ldots,T_n\ \exists s\ T_1 P_1+\ldots+T_n P_n = s q.
\]
\end{proof}
Looking at the above formula, it is easy to observe that a stronger statement can be made by quantifying over a tuple instead of individual polynomials, see Proposition~\ref{pr:nstd_principal} below.

Note that $\nonstandardmodel{F[x]}{\mathcal M}$ is not a unique factorization domain unless $\nstdN\cong \mathbb N$. For example, for each $n>\mathbb N$, the polynomial $x^n$ is not a product of finitely many irreducible polynomials. However, it is not hard to see that $\nonstandardmodel{F[x]}{\mathcal M}$ is a unique ``nonstandard'' factorization domain in the following sense.
\begin{proposition}\label{le:ufd}
For each $q\in \nonstandardmodel{F[x]}{\mathcal M}$, $\mathcal M\equiv \mathbb S(F,\mathbb N)$, there is $n\in\nstdN$ and a nonstandard tuple $P=(P_1,\ndots,P_n)$, where each $P_i$ is irreducible ($1\le i\le n$) s.t. $q=\prod_{i=1}^n P_i$. Moreover, such factorization is unique up to a nonstandard permutation of factors and invertible factors.
\end{proposition}
\begin{proof}
Existence can be expressed by a first-order formula:
\[
\forall q\ \exists P\ \left(\prod P=q \wedge (t(P,f,i)\rightarrow f\mbox{ irreducible})\right).
\]
For uniqueness, if $\prod P=\prod S$ and $\ell(P)<\ell(S)$, we first consider the definable tuple $P'$ given by $\ell(P')=\ell(S)$, $i\le \ell(P)\rightarrow (t(P',f,i)\leftrightarrow t(P,f,i))$, $\ell(P)<i\le \ell(S)\rightarrow t(P',1,i).$ Now, without loss of generality, we assume $\ell(P)=\ell(S)$, and uniqueness of an irreducible decomposition can be expressed by a first-order formula (see Subsection~\ref{se:nonstd_permutations} concerning nonstandard permutations):
\[
 \begin{split}
 \forall P,S\ &\left(
 \ell(P)=\ell(S)\wedge \prod P = \prod S \wedge \forall i\forall f\ t(P\vee S,f,i) \rightarrow f\ \mathrm{irreducible}
 \right)
 \rightarrow\\
 &\exists A\ \prod A=1 \wedge \ell(A)=\ell(P)\\
 &\exists \mathrm{permutation\ }\sigma\ \forall 1\le i\le n\ \forall f\ t(P,A_if,i)\leftrightarrow t(S,f,\sigma_i).
 \end{split}
\]
\end{proof}

Notice that $\nonstandardmodel{F[x]}{\mathcal M}$ is not a Noetherian ring (unless $\nstdN\cong \mathbb N$). Indeed, for any $n\in\nstdN$, $n>\mathbb N$ we consider the principal ideals
 \[
 \langle x^n\rangle<\langle x^{n-1}\rangle<\langle x^{n-2}\rangle<\ldots.
 \]
The same observation also follows by~\cite{Bauval:1985} where it is shown that a Noetherian ring elementarily equivalent to $F[x]$ must necessarily be a polynomial ring.
$\nonstandardmodel{F[x]}{\mathcal M}$ is, however, nonstandardly Noetherian in the following sense.
\begin{proposition}
Suppose a formula $\Phi(p,n)$ defines, for each $n\in\nstdN$, a set $I_n$ that is an ideal of $\nonstandardmodel{F[x]}{\mathcal M}$, so that $I_{n_1}\subseteq I_{n_2}$ whenever $n_1\le n_2$. Then there is $m\in\nstdN$ s.t. $I_n=I_m$ for all $n\ge m$.
\end{proposition}
\begin{proof}
Consider the polynomial $d\in \cup I_n$ of lowest degree, which is expressed by the following formula:
\[
\exists m\, \Phi(d,m)\wedge \left[\forall p\ \forall n\ \Phi(p,n) \rightarrow \deg d\le \deg p\right].
\]
Note that the above can be expressed in the first-order language by Lemma~\ref{le:deg_and_value}.
By dividing with a remainder (Lemma~\ref{le:remainder}), it is easy to see that $d| p$ for all $p\in\cup I_n$, from which the statement follows. Notice that the lowest $m$ satisfying the conclusion of the lemma is defined by the formula:
\[
(\forall n\ge m\ \Phi(d,n))\wedge (\forall n<m\ \neg\Phi(d,n)).
\]
\end{proof}
Note that we could prove the above lemma by directly expressing nonstandard chain stabilization by a first-order formula, see Proposition~\ref{le:noetherity_multi} and its proof for details.

We can also consider nonstandardly generated ideals: for a tuple $P={(P_1,\ndots,P_n)}$, put
\[\langle P\rangle^\sim = \left\{q\in \nonstandardmodel{F[x]}{\nstdN}\left| \exists T\ \ell(T)=\ell(P) \wedge q=\sum T P\right.\right\}.
\]
It is worth noting that if a definable subset is an ideal, i.e., closed under linear combinations, then by induction it is also closed under nonstandard linear combinations.
\begin{proposition}\label{pr:nstd_principal}
Let $I\lhd \nonstandardmodel{F[x]}{\mathcal M}$ be nonstandardly generated by a nonstandard tuple $P$ of polynomials. Then $I$ is principal.
\end{proposition}
\begin{proof}
This can be expressed by a first-order formula:
\[
\forall P\ \exists q\ \forall T\ \exists s \ \ell(T)=\ell(P)\rightarrow \sum TP = s q.
\]
\end{proof}

\subsection{Nonstandard polynomials in several variables}\label{se:multi_var}

Let $\{x_1,\ldots,x_k\}$ be a finite set of variables, and let $F$ be an infinite field.
Similarly to Section~\ref{se:single_var}, we introduce nonstandard models of $F[x_1,\ldots,x_k]$ by establishing a regular invertible interpretation in the list superstructure over $F$.
If, additionally, $F$ is regularly invertibly interpretable in $\mathbb N$, then this leads to a regular invertible interpretation with $\mathbb N$, and thus the nonstandard models are the rings $\nonstandardmodel{F[x_1,\ldots,x_k]}{\nstdN}$ (or $\nonstandardmodel{F[x_1,\ldots,x_k]}{\mathcal M}$, $\mathcal M\equiv \mathbb S(F,\nstdN)$ in the general case).

We fix a recursive enumeration of commutative monomials in variables $x_1,\ldots,x_k$ (for example, the shortlex enumeration of the respective exponent tuples).
Let $M_n$ stand for the $n$-th monomial.
As above, we encode this enumeration by first-order formulas, and express addition and multiplication of tuples $(a_0,\ldots,a_n)$ correspondingly to addition and multiplication of polynomials $\sum_{i\le n} a_iM_i$
In particular, sum and product of two nonstandard polynomials $p=\sum_{i\le n} a_iM_i$ and $q=\sum_{i\le m} b_iM_i$ take the familiar form
\begin{equation}\label{eq:multivar_sum}
(a_0,\ldots,a_n)+(b_0,\ldots,b_m)=(a_0+b_0,a_1+b_1,\ldots,a_{\max\{n,m\}}+b_{\max\{n,m\}}),
\end{equation}
where for convenience we set $a_i=0$ for $i>n$ and $b_i=0$ for $i>m$, and
\begin{equation}\label{eq:multivar_product}
(a_0,\ldots,a_n)\cdot(b_0,\ldots,b_m)=(c_0,\ldots,c_\ell), \quad c_k=\sum_{\substack{i\le n,j\le m\\ M_iM_j=M_k}}a_ib_j\mbox{ for } 0\le k\le \ell,
\end{equation}
where $\ell= \max\{k\mid M_k=M_iM_j, i\le n, j\le m\}$.

This allows us to interpret $F[x_1,\ldots,x_k]\cong \Gamma_0(\mathbb S(F,\mathbb N))$, and therefore to consider the nonstandard version of $F[x_1,\ldots,x_k]$, $\Gamma_0(\mathcal M)\cong \nonstandardmodel{F[x_1,\ldots,x_k]}{\mathcal M}$, where $\mathcal M\equiv \mathbb S(F,\mathbb N)$. Namely, a nonstandard tuple $(a_1,\ldots,a_n)$ is viewed as a tuple of coefficients for a nonstandard polynomial $p(x_1,\ldots,x_k)=\sum_{i\le n} a_iM_i$, with operations given by formulas~\eqref{eq:multivar_sum} and~\eqref{eq:multivar_product}. Note that the case $k=1$ is consistent with the definition in Section~\ref{se:single_var}. If, additionally, $F$ is regularly invertibly interpretable in $\mathbb N$, this gives a regular invertible interpretation $F[x_1,\ldots,x_k]\cong \Gamma(\mathbb N)$, leading to a nonstandard model $\Gamma(\nstdN)\cong \nonstandardmodel{F[x_1,\ldots,x_k]}{\nstdN}$. 

To make the nonstandard ring of polynomials more immediately approachable, we consider a more explicit interpretation. Express a polynomial $p(x_1,\ldots,x_k)$ as
\begin{equation}\label{eq:poly_multivar}
p=\sum_{i=1}^m a_i M_i=\sum_{i=1}^n a_i x_1^{n_{i1}}\cdots x_k^{n_{ik}},\quad m\in\mathbb N, a_i\in F, n_{ij}\in\mathbb N.
\end{equation}
To this polynomial, we associate a tuple
\begin{equation}\label{eq:multivar_interpretation}
(a_1,n_{11},\ldots,n_{1k},\ a_2,n_{21},\ldots,n_{2k},\ \ldots,\ a_m,n_{m1},\ldots,n_{mk}).
\end{equation}
This tuple is thought of as a pair of tuples $A=(a_1,\ldots,a_m), N=(n_1,\ldots,n_m)$, where $n_i$ is the index of $(n_{i1},\ldots,n_{ik})$ under enumeration of $k$-tuples in $\mathbb N$.
Using the approach of Sections~\ref{se:nonstandard_summation} and~\ref{se:summation_arb_superstructre}, it is straightforward to define a formula $\varphi_{\mathrm{red}}(A,N,A_{\mathrm{red}}, N_{\mathrm{red}})$ over $\mathbb S(F,\mathbb N)$ that says that the polynomial $p_{\mathrm{red}}$ encoded by $(A_{\mathrm{red}}, N_{\mathrm{red}})$ is a reduced form of the polynomial $p$ encoded by $(A,N)$, that $p_{\mathrm{red}}=p$ and the expression ~\eqref{eq:poly_multivar} for $p_{\mathrm{red}}$ has $M_i\neq M_j$ for $i\neq j$ and $0\neq a_i\in F$. Using the same approach, we can define addition~\eqref{eq:multivar_sum} and multiplication~\eqref{eq:multivar_product} by formulas $\varphi_+$ and $\varphi_\times$, respectively, over $\mathbb S(F,\mathbb N)$.

It is shown in~\cite{KharlampovichMyasnikov:2018a} that this interpretation gives rise to a regular bi-interpretation of $F[x_1,\ldots,x_k]$ with $\mathbb S(F,\mathbb N)$. If $F$ is invertibly regularly interpretable in $\mathbb N$, this ultimately gives a regular invertible interpretation of $F[x_1,\ldots,x_k]$ with $\mathbb N$, $F[x_1,\ldots,x_k]\cong \Gamma(\mathbb N)$. As a corollary of Theorem~\ref{th:complete-scheme}, like in the case of a single variable, either of the two invertible interpretations gives a description of rings elementarily equivalent to $F[x_1,\ldots,x_k]$.

We now record the above conclusions in the form of multivariable versions of Theorem~\ref{th:Fx_SFN_biint}--Corollary~\ref{co:Fx_elem_equiv} from Section~\ref{se:single_var}.

\begin{theorem}[\cite{KharlampovichMyasnikov:2018a}]\label{th:Fx_multi_SFN_biint}
Let $F$ be an infinite field, and let $\{x_1,\ldots,x_k\}$ be a finite set of variables. Then the polynomial ring $F[x_1,\ldots,x_k]$ is regularly bi-interpretable with $\mathbb S(F,\mathbb N)$.
\end{theorem}

\begin{corollary}\label{co:Fx_multi_elem_equiv_list}
Let $F$ be an infinite field.
Let $\Gamma_0$ be the code of a regular invertible interpretation of $F[x_1,\ldots,x_k]$ in $\mathbb S(F,\mathbb N)$, $F[x_1,\ldots,x_k]\cong \Gamma_0(\mathbb S(F,\mathbb N))$. Then every ring elementarily equivalent to $F[x_1,\ldots,x_k]$ has the form $\nonstandardmodel{F[x_1,\ldots,x_k]}{\mathcal M}\cong \Gamma_0(\mathcal M)$, where $\mathcal M\equiv \mathbb S(F,\mathbb N)$.
\end{corollary}

\begin{corollary}
\label{co:bi-int_Fx_multi_N}
Let $F$ be a field regularly bi-interpretable with $\mathbb N$ (regularly invertibly interpretable in $\mathbb N$). Then the ring of polynomials $\langle F[x_1,\ldots,x_k];+,\cdot,0,1\rangle$ is regularly bi-interpretable with $\mathbb N$ (resp., regularly invertibly interpretable in $\mathbb N$), $F[x_1,\ldots,x_k]\cong \Gamma(\mathbb N)$.
\end{corollary}

\begin{corollary}\label{co:Fx_multi_elem_equiv}
Let $F$ be a field regularly invertibly interpretable in $\mathbb N$. Let $\Gamma$ be the code of an invertible interpretation of $F[x_1,\ldots,x_k]$ in $\mathbb N$, $F[x_1,\ldots,x_k]\cong \Gamma(\mathbb N)$. Then every ring elementarily equivalent to $F[x_1,\ldots,x_k]$ has the form $\nonstandardmodel{F[x_1,\ldots,x_k]}{\nstdN}\cong \Gamma(\nstdM)$, where $\nstdM\equiv \mathbb N$.
\end{corollary}

For the purpose of understanding the ring of nonstandard polynomials $\nonstandardmodel{F[x_1,\ldots,x_k]}{\mathcal M}$, even in the case when $F$ is regularly invertibly interpretable in $\mathbb N$, it is convenient to think of the tuples $A$ and $N$, rather than their ultimate number in $\mathbb N$.
Indeed, if $\nstdM\equiv \mathbb N$ and $\widetilde{F}\equiv F$ are the numbers sort and the field sort of $\mathcal M\cong \mathbb S(F,\mathbb N)$, respectively, then $A=(a_1,\ldots,a_m)$, where $m\in\nstdM$, $a_i\in \widetilde{F}$, is a nonstandard tuple from the list sort of $\mathcal M$,
and $N=(n_1,\ldots,n_m)$, where $m\in\nstdM$, $n_i\in\nstdM$.
This leads to the following form of elements of $\nonstandardmodel{F[x_1,\ldots,x_k]}{\mathcal M}\cong \Gamma_0(\mathcal M)$:
\begin{equation}\label{eq:nonstd_multivar}
\begin{split}
p=&\sum_{i=1}^m a_i M_i=\sum_{i=1}^m a_i x_1^{n_{i1}}\cdots x_k^{n_{ik}},\\
&(a_1,\ldots,a_m)\in \mathcal M,\
(n_{11},\ldots,n_{mk})\in \mathbb S(\nstdM,\nstdM),
\end{split}
\end{equation}
which is an expression immediately analogous to~\eqref{eq:poly_multivar}, with the difference that nonstandard tuples are used (and therefore variables $m,a_i,n_{ij}$ are allowed nonstandard values).
The formula $\varphi_{\mathrm{red}}$ above defines the reduced form of a nonstandard polynomial, that is, an expression~\eqref{eq:nonstd_multivar} with $M_i\neq M_j$ for $i\neq j$ and $0\neq a_i\in\widetilde{F}$. The formulas $\varphi_+$ and $\varphi_\times$ define addition and multiplication of nonstandard polynomials.

We can define summation and product over a nonstandard tuple of nonstandard polynomials, ultimately by Lemma~\ref{le:summation}. Observe that the polynomial $p=x_i$ corresponds to a pair of tuples $(a_1=1), (n_{11}=0,\ldots,n_{1i}=1,\ldots,n_{1k}=0)$. It is then a straightforward check that
\[
\prod_{j=1}^{n_1}x_1\cdots \prod_{j=1}^{n_k}x_k=x_1^{n_1}\cdots x_k^{n_k},
\]
where the products are taken in the sense of Lemma~\ref{le:summation}, and then that
\[
\sum_{i=1}^m a_i \prod_{j=1}^{n_{i1}}x_1\cdots \prod_{j=1}^{n_{ik}}x_k=\sum_{i=1}^n a_i x_1^{n_{i1}}\cdots x_k^{n_{ik}},
\]
where the sum and the products in the left hand side are taken in the sense of Lemma~\ref{le:summation}. Theorem~\ref{th:well-def} shows that the resulting ring is isomorphic to $\nonstandardmodel{F[x_1,\ldots,x_k]}{\mathcal M}$ defined previously.

Unique nonstandard factorization takes place for $\nonstandardmodel{F[x_1,\ldots,x_k]}{\mathcal M}$, similarly to Proposition~\ref{le:ufd}. We omit the nearly identical proof.
\begin{proposition}\label{le:ufd_multi}
For each $q\in \nonstandardmodel{F[x_1,\ldots,x_k]}{\mathcal M}$, there is $n\in\nstdN$ and a nonstandard tuple of irreducible polynomials $P$ s.t. $q=\prod P$. Moreover, such factorization is unique up to a nonstandard permutation of factors and invertible factors.
\end{proposition}

Noetherity behaves similarly to the single variable case. Namely, the ring is not Noetherian (via the same example), but is nonstandardly Noetherian for definable ideals.

\begin{proposition}\label{le:noetherity_multi}
Suppose a formula $\Phi(p,n)$ defines, for each $n\in\nstdN$, a set $I_n$ that is an ideal of $\nonstandardmodel{F[x_1,\ldots,x_k]}{\mathcal M}$. Then there is $m\in\nstdN$ s.t. $I_n=I_m$ for all $n\ge m$.
\end{proposition}
\begin{proof}
Stabilization of the ascending ideal chain given by the formula $\Phi(p,n)$ can be recorded by the following first-order formula:
\[
\begin{split}
 [&\forall n\ \{p\mid \Phi(p,n)\} \mbox{ is an ideal}\\
 \wedge &\forall n_1,n_2\ n_1\le n_2\rightarrow\forall p\ \Phi(p,n_1)\rightarrow \Phi(p,n_2)]\\
 \rightarrow &\exists m\ \forall n\ge m\ \forall p\ \Phi(p,n)\leftrightarrow \Phi(p,m).
\end{split}
\]
\end{proof}
Finally, we note that a universal property similar to the one described in Section~\ref{se:single_var_universal} holds for~$\nonstandardmodel{F[x_1,\ldots,x_k]}{\mathcal M}$. We omit the details, since the reasoning is almost identical.

\begin{theorem}\label{th:universal_nstd_multi}
Let $F$ be an infinite field, $F[x_1,\ldots,x_k]$ the ring of polynomials, and $\mathbb A$ an arbitrary commutative $F$-algebra with a regular interpretaion in $\mathbb S(F, \mathbb N)$. Let $\mathcal M\equiv\mathbb S(F, \mathbb N)$ and $\widetilde{\mathbb A}=\nonstandardmodel{\mathbb A}{\mathcal M}$. Then every mapping $x_i\mapsto a_i\in \widetilde{\mathbb A}$, $i=1,\ldots,k$, extends uniquely to an algebra homomorphism $\nonstandardmodel{F[x_1,\ldots,x_k]}{\mathcal M}\to\widetilde{\mathbb A}$.
\end{theorem}
We omit the proof since it essentially repeats that of Theorem~\ref{th:universal_nstd_single}.

\section{Laurent polynomials over a field}\label{se:laurent}
\subsection{Laurent polynomials in single variable}\label{se:laurent_single}
Let $F$ be an infinite field. It is straightforward to see that the ring of Laurent polynomials $\langle F[x,x^{-1}]; +,\cdot,0,1\rangle$ in a single variable can be regularly interpreted in $\mathbb S(F,\mathbb N)$ by associating a Laurent polynomial $a_{-m}x^{-m}+\ldots+a_nx^n$, $m,n\ge 0$, with a pair of tuples
\begin{equation}\label{eq:laurent_in_list}
(a_0,\ldots,a_n), (a_{-1},\ldots,a_{-m})
\end{equation}
(with ``tail of zeros'' equivalence relation). In the case when $F$ is regularly interpretable in $\mathbb N$, this further leads to a regular (in fact, absolute) interpretation $F[x,x^{-1}]\cong \Gamma(\mathbb N)$.

\begin{proposition}\label{pr:laurent_interpretation}
Let $F$ be an infinite field. Let $\Gamma_0$ be the code of regular interpretation of $F[x,x^{-1}]$ in $\mathbb S(F,\mathbb N)$, $F[x,x^{-1}]\cong \Gamma_0(\mathbb S(F,\mathbb N))$.
Then for every $\mathcal M\equiv \mathbb S(F,\mathbb N)$, the ring $\Gamma_0(\mathcal M)$ is elementarily equivalent to $F[x,x^{-1}]$.

If, additionally, $F$ is regularly interpretable in $\mathbb N$ and $\Gamma$ is the code of regular interpretation of $F[x,x^{-1}]$ in $\mathbb N$, $F[x,x^{-1}]\cong \Gamma(\mathbb N)$, then for every $\nstdM\equiv \mathbb N$, $\Gamma(\nstdM)$ is elementarily equivalent to $F[x,x^{-1}]$.
\end{proposition}

However, we cannot take full advantage of our techniques unless this interpretation gives rise to a regular bi-interpretation, or at least a regular invertible interpretation.

For reverse interpretation, we follow the ideas of~\cite[Lemma 9]{KharlampovichMyasnikov:2018a}. We recall that the definable set of invertible Laurent polynomials is $\Inv=\{\alpha x^n\mid 0\neq \alpha \in F, n\in\mathbb Z\}$.
Then $F\subseteq F[x,x^{-1}]$ is definable as
\[
\{p\mid p\in \Inv\wedge [p=1\vee (p+1)\in \Inv \vee (p-1)\in \Inv]\}.
\]
(The disjunction in the square brackets is set up to accommodate the case of a finite characteristic.)
We can also define the set of irreducible Laurent polynomials as
\[
\Irr=\{p\mid \neg (p\in \Inv)\wedge \forall q\forall r\ p=qr\to [q\in \Inv\vee r\in \Inv]\}.
\]
Further, the formula
\[
p\in \Inv \wedge \forall 0\neq \alpha\in F\quad p-\alpha\in\Irr
\]
defines the set $X=\{\alpha x^{\pm 1}\mid 0\neq \alpha\in F\}$.

\begin{lemma}\label{le:Xpowers}
Let $a=\alpha x^\varepsilon \in X$, with $X$ as above. Then $a-1$ divides $\beta x^m-1$ in $F[x,x^{-1}]$, where $\beta\in F$, $m\in\mathbb Z$, if and only if $\beta=\alpha^{\varepsilon m}$.
\end{lemma}
\begin{proof}
Suppose $a-1$ divides $\beta x^m -1$. Write
\[
\beta x^m -1=\beta x^m -\alpha^{\varepsilon m} x^{m}+\underbrace{\alpha^{\varepsilon m} x^{m}-1}_{a^{\varepsilon m}-1},
\]
therefore $a-1\mid x^m(\beta-\alpha^{\varepsilon m})$, from which the statement follows.
\end{proof}

As an immediate corollary, we get the following.
\begin{lemma}\label{le:X}
For every $a\in X$, the formula
\[
\varphi(p,a)=(p\in \Inv)\wedge a-1\mid p-1
\]
defines the set $\{a^m\mid m\in\mathbb Z\}$.
\end{lemma}
This allows a regular interpretation $\mathbb Z\cong \Delta(F[x,x^{-1}])$. Indeed, with $a$ as above, we note that for $m,n,k\in\mathbb Z$ we have
\[
\begin{split}
&m+n=k\iff a^m\cdot a^n=a^k,\\
&m\mid n\iff (a^m-1)\mid (a^n-1).
\end{split}
\]
This allows to regularly interpret $\langle\mathbb Z;+,|,0,1\rangle$ in $F[x,x^{-1}]$. Since $\langle\mathbb N;+,\cdot,0,1\rangle$ is absolutely interpretable (in fact, definable) in $\langle \mathbb Z;+,|,0,1\rangle$ by Lemma~\ref{le:Z_div} below, we obtain a regular interpretation of the former in $F[x,x^{-1}]$.
To interpret $\List(F,\mathbb N)$ in $F[x,x^{-1}]$, we, similarly to the case of polynomials, to a tuple $(b_0,b_1,\ldots,b_n)$ we associate a quadruple
\begin{equation}\label{eq:list_in_laurent}
(b_0+b_1a+\ldots+b_na^n, a^n, b_0+b_1a^{-1}+\ldots+b_na^{-n}, a^{-n}).
\end{equation}
The field sort $F$ of $\mathbb S(F,\mathbb N)$ is interpretable by definability of $F\subseteq F[x,x^{-1}]$, as mentioned above. The length function $\ell$ is interpreted straightforwardly; the predicate $t(s,a,i)$ by Lemma~\ref{le:extract_coeffs} below. We denote the resulting regular interpretation by $\mathbb S(F,\mathbb N)\cong \Delta(F[x,x^{-1}]$.

\begin{lemma}\label{le:Z_div}
$\langle\mathbb N;+,\cdot,0,1\rangle$ is definable in $\langle \mathbb Z;+,|,0,1\rangle$ without parameters.
\end{lemma}
\begin{proof}
We employ the idea similar to the one in~\cite[Section 4b]{Robinson:1951}, where it is shown that $\langle\mathbb N;+,\cdot,0,1\rangle$ is definable in $\langle \mathbb N;+,|,0,1\rangle$.

The only difficulty is to define the multiplication. We note that in $\langle \mathbb Z;+,|,0,1\rangle$, the relation $n=\pm\mathop{\mathrm{lcm}}(k,l)$ is defined by
\[
n=\pm\mathop{\mathrm{lcm}}(k,l)\leftrightarrow [\forall m\ n\mid m\leftrightarrow k\mid m\wedge l\mid m].
\]
Then $n=\pm(k-1)(k+1)$ is given by
\[
\begin{split}
n=\pm(k-1)(k+1)\leftrightarrow [&2\mid k\ \wedge\ n=\pm\mathop{\mathrm{lcm}}(k-1,k+1)]\vee\\
\vee [&2\nmid k\ \wedge\ n/2=\pm\mathop{\mathrm{lcm}}(k-1,k+1)].
\end{split}
\]
(Here $n/2$ is a shorthand for $n_1$ such that $n=n_1+n_1$.)
Finally, we observe that $n-1=\pm (k-1)(k+1)$ if and only if $n=k^2$ or $n=2-k^2$. With that in mind, the relation $n=k^2$ is defined by
\[
\begin{split}
n=k^2\leftrightarrow (&k=\pm 2\wedge n=4)\vee(k=\pm 1\wedge n=1)\vee (k=0\wedge n=0)\vee \\
\vee [&(k\neq \pm2\wedge k\neq \pm1\wedge k\neq 0)\wedge (n-1=\pm(k-1)(k+1)\wedge k\mid n)].
\end{split}
\]
After that, multiplication is defined in a standard way:
\[
n=kl\leftrightarrow n+n=(k+l)^2-k^2-l^2.
\]
Once multiplication is defined, $\mathbb N\subseteq \mathbb Z$ is defined in the usual way by the four-square theorem. (Alternatively, once we defined $n=k^2$, we could use the four-square theorem to define $\mathbb N\subseteq\mathbb Z$ and then refer to the definability of $\langle\mathbb N;+,\cdot,0,1\rangle$ in $\langle \mathbb N;+,|,0,1\rangle$ by \cite[Section 4b]{Robinson:1951}.)

\end{proof}

To show bi-interpretation, we need to show that $\mathbb S(F,\mathbb N) \cong \Delta\circ \Gamma_0(\mathbb S(F,\mathbb N))$ is definable in $\mathbb S(F,\mathbb N)$ and $F[x,x^{-1}]\cong \Gamma_0\circ \Delta(F[x,x^{-1}])$ is definable in $F[x,x^{-1}]$.

The definability of $\mathbb S(F,\mathbb N) \cong \Delta\circ \Gamma_0(\mathbb S(F,\mathbb N))$ in $\mathbb S(F,\mathbb N)$ is straightforward, since, for the list sort, composition of~\eqref{eq:laurent_in_list}, \eqref{eq:list_in_laurent} results in the correspondence
\[
\begin{split}
(b_0,\ldots,b_n)\to (&(b_0,\ldots,b_n),(0,\ldots,0,0),\\
&(0,\ldots,0,1),(0,\ldots,0,0),\\
&(1,\ldots,0,0),(b_1,\ldots,b_n),\\
&(0,\ldots,0,0),(0,\ldots,0,1))
\end{split}
\]
if $a=\alpha x$, or a similar tuple if $a=\alpha x^{-1}$.

The definability of $F[x,x^{-1}]\cong \Gamma_0\circ \Delta(F[x,x^{-1}])$ in $F[x,x^{-1}]$ follows by inspection, which we perform below.
We start by considering the following lemma.

\begin{lemma}\label{le:free_term}
Let $a\in X=\{\alpha x^{\pm 1}\mid 0\neq \alpha\in F\}$. The relation 
\[
\{(p,p^-,p_0,p^+)\mid p=p^-+p_0+p^+, p_0\in F, p^{-}=\sum_{i<0}\beta_ia^i, p^+=\sum_{i>0}\beta_ia^i, \beta_i\in F\}
\]
on $F[x,x^{-1}]^4$ is definable with parameter $a$.
\end{lemma}
\begin{proof}
By~\cite[Lemma 2]{KharlampovichMyasnikov:2018b}, the polynomial ring $F[a+a^2+a^3]$ is definable in $F[x,x^{-1}]$. Therefore, so is the polynomial ring \[
F[a]=F[a+a^2+a^3]+aF[a+a^2+a^3]+a^2F[a+a^2+a^3].
\]
It follows that the set $aF[a]$ is definable in $F[x,x^{-1}]$, and similarly definable is the set $a^{-1}F[a^{-1}]$. Now, the desired formula in $F[x,x^{-1}]$ is
\[
p_0\in F\wedge p^-\in a^{-1}F[a^{-1}]\wedge p^+\in aF[a]\quad p=p^{-}+p_0+p^+.
\]
\end{proof}

Recall that the set $\{a^{i}\mid i\in\mathbb Z\}$ is definable by Lemma~\ref{le:X}. It follows that for $p\in F[x,x^{-1}]$, each coefficient $\beta_i$ in $p=\sum_i\beta_i a^i$ is can be defined by a formula in $F[x,x^{-1}]$ as $((a^i)^{-1}p)_0$. Namely, the following lemma takes place.
\begin{lemma}\label{le:extract_coeffs} There is a formula $\mathop{\mathrm{coeff}}(a,a^m,p,p_m)$ that is satisfied in $F[x,x^{-1}]$ if and only if $a\in X=\{\alpha x^{\pm 1}\mid 0\neq \alpha\in F\}$ and $p=\sum_i \beta_ia^i$, $\beta_i\in F$, with $\beta_m=p_m$.
\end{lemma}
\begin{proof}
Let $\varphi_0$ be the formula that defines $X$, $\varphi(p,a)$ be as in Lemma~\ref{le:X}, and $\varphi_1(p,p^-,p_0,p^+)$ be the formula from Lemma~\ref{le:free_term}. Then
\[
\mathop{\mathrm{coeff}}(a,a^m,p,p_m)=\varphi_0(a)\wedge \varphi(a^m,a)\wedge \exists p^-,p^+\ \varphi_1((a^m)^{-1}p,p^-,p_m,p^+).
\]
\end{proof}
The desired definability of the isomorphism between $F[x,x^{-1}]$ and $\Gamma_0\circ\Delta(F[x,x^{-1}])$ follows by Lemma~\ref{le:extract_coeffs}. Thus we obtain the following theorem.

\begin{theorem}\label{th:laurent_bi}
Let $F$ be an infinite field. Then the interpretation $F[x,x^{-1}]\cong \Gamma_0(\mathbb S(F,\mathbb N))$ defined above gives rise to a regular bi-interpretation of $F[x,x^{-1}]$ and $\mathbb S(F,\mathbb N)$ in each other.
\end{theorem}
Like we did in Section~\ref{se:single_var}, we see that this leads to the following representation of nonstandard Laurent polynomials: each nonstandard Laurent polynomial is a nonstandard sum
\[
p=\sum_{i=-m}^n p_i x^i,
\]
where $m,n\in\mathbb N$, $(p_{-m},\ldots,p_n)$ is a nonstandard tuple in $\mathbb S(\widetilde{F},\nstdM)$, and $x^i$, $i\in\nstdM$, is the Laurent polynomial corresponding under $\Gamma$ to a pair of tuples where all components are zeros except one component $a_i=1$. It is straightforward to check that in this case the multiplication and addition are defined by formulas similar to~\eqref{eq:nstd_poly_sum_product}. 
Namely, for $p=\sum_{i=-m}^n p_i x^i$ and $q=\sum_{i=-m'}^{n'} q_i x^i$ we have
\begin{equation}\label{eq:laurent}
p+q=\sum_{k=-\max{m,m'}}^{\max\{n,n'\}} (p_k+q_k)x^k,\qquad pq=\sum_{k=-m-m'}^{n+n'} \left(\sum_{i+j=k}a_ib_j\right)x^k.
\end{equation}

As a corollary of Theorem~\ref{th:laurent_bi}, we can strengthen Proposition~\ref{pr:laurent_interpretation} using Theorem~\ref{th:complete-scheme}.

\begin{corollary}
\label{co:laurent_equiv}
Let $F$ be an infinite field. Then every ring elementarily equivalent to $F[x,x^{-1}]$ has the form $\nonstandardmodel{F[x,x^{-1}]}{\mathcal M}$, $\mathcal M\equiv \mathbb S(F,\mathbb N)$.
\end{corollary}

Like in Section~\ref{se:polynomials}, in the case when $F$ is regularly bi- or invertibly interpretable in $\mathbb N$, we ultimately get a respective regular bi- or invertible interpretation of Laurent polynomials in $\mathbb N$, and a description of elementarily equivalent rings through nonstandard models of $\mathbb N$.

\begin{corollary}
\label{co:laurent_bi_N}
Let $F$ be a field regularly bi-interpretable with $\mathbb N$ (regularly invertibly interpretable in $\mathbb N$). Then the ring of Laurent polynomials $\langle F[x,x^{-1}];+,\cdot,0,1\rangle$ is regularly bi-interpretable with $\mathbb N$ (resp., regularly invertibly interpretable in $\mathbb N$), $F[x,x^{-1}]\cong \Gamma(\mathbb N)$.
\end{corollary}

\begin{corollary}\label{co:laurent_elem_equiv_N}
Let $F$ be a field regularly invertibly interpretable in $\mathbb N$. Let $\Gamma$ be the code of a regular invertible interpretation of $F[x,x^{-1}]$ in $\mathbb N$, $F[x,x^{-1}]\cong \Gamma(\mathbb N)$. Then every ring elementarily equivalent to $F[x,x^{-1}]$ has the form $\nonstandardmodel{F[x,x^{-1}]}{\nstdN}\cong \Gamma(\nstdM)$, where $\nstdM\equiv \mathbb N$.
\end{corollary}

\begin{theorem}\label{th:universal_nstd_laurent_single}
Let $F$ be an infinite field, $F[x,x^{-1}]$ the ring of Laurent polynomials, and $\mathbb A$ an arbitrary commutative $F$-algebra with a regular interpretaion in $\mathbb S(F, \mathbb N)$. Let $\mathcal M\equiv\mathbb S(F, \mathbb N)$ and $\widetilde{\mathbb A}=\nonstandardmodel{\mathbb A}{\mathcal M}$. Then every mapping $x\mapsto a\in \widetilde{\mathbb A}$, where $a$ is invertible in $\mathbb A$, extends uniquely to an algebra homomorphism $\nonstandardmodel{F[x,x^{-1}]}{\mathcal M}\to\widetilde{\mathbb A}$.
\end{theorem}
We omit the proof since it essentially repeats that of Theorem~\ref{th:universal_nstd_single}.

\subsection{Laurent polynomials in several variables}\label{se:laurent_multi}
To interpret the ring of Laurent polynomials in $k$ variables $F[x_1,x_1^{-1},\ldots, x_k,x_k^{-1}]$ in $\mathbb S(F,\mathbb N)$, we fix a recursive enumeration of Laurent monomials $\{x_1^{i_1}\cdots x_k^{i_k}\mid i_1,\ldots,i_k\in\mathbb N\}\to \mathbb N$. Similarly to Section~\ref{se:multi_var}, we denote $n$-th Laurent monomial by $M_n$ and interpret a Laurent polynomial $\sum_{i\le n} a_iM_i$ as a tuple $(a_0,\ldots,a_n)$ (up to a tail of zeros).
We can record formulas for addition and multiplication of $p=\sum_{i\le n}a_iM_i$ and $q=\sum_{i\le m}b_iM_i$ by formulas~\eqref{eq:multivar_sum} and~\eqref{eq:multivar_product}.
This produces a regular intepretation of $F[x_1,x_1^{-1},\ldots, x_k,x_k^{-1}]\cong\Gamma_0\mathbb S(F,\mathbb N)$.

By~\cite[Lemma 3]{KharlampovichMyasnikov:2018b}, $F[x,x^{-1}]$ is regularly interpretable in $F[x_1,x_1^{-1},\ldots, x_k,x_k^{-1}]$, which allows regular interpretation of $\mathbb S(F,\mathbb N)$ in $F[x_1,x_1^{-1},\ldots, x_k,x_k^{-1}]$.
Omitting details of a consideration similar to the one in Section~\ref{se:laurent_single}, it follows by~\cite[Theorem 3, Lemma 11]{KharlampovichMyasnikov:2018b} that this gives rise to a regular bi-interpretation, which we record as the following theorem.

\begin{theorem}\label{th:laurent_mult_bi}
Let $F$ be an infinite field. Then, in the above notation, the interpretation $F[x_1,x_1^{-1},\ldots, x_k,x_k^{-1}]\cong\Gamma_0(\mathbb S(F,\mathbb N))$ gives rise to a regular bi-interpretation of $F[x_1,x_1^{-1},\ldots, x_k,x_k^{-1}]$ and $\mathbb S(F,\mathbb N)$ in each other.
\end{theorem}

The description of rings elementarily equivalent to $F[x_1,x_1^{-1},\ldots, x_k,x_k^{-1}]$ follows by Theorem~\ref{th:complete-scheme}.

\begin{corollary}
\label{co:laurent_multi_equiv}
Let $F$ be an infinite field. Then every ring elementarily equivalent to $F[x_1,x_1^{-1},\ldots, x_k,x_k^{-1}]$ has the form $\nonstandardmodel{F[x_1,x_1^{-1},\ldots, x_k,x_k^{-1}]}{\mathcal M}$, $\mathcal M\equiv \mathbb S(F,\mathbb N)$.
\end{corollary}

Like in Sections~\ref{se:polynomials} and~\ref{se:laurent_single}, in the case when $F$ is regularly bi- or invertibly interpretable in $\mathbb N$, we ultimately get a respective regular bi- or interpretable interpretation of Laurent polynomials in $\mathbb N$, and a description of elementarily equivalent rings through nonstandard models of $\mathbb N$.

\begin{corollary}
\label{co:laurent_multi_bi_N}
Let $F$ be a field regularly bi-interpretable with $\mathbb N$ (regularly invertibly interpretable in $\mathbb N$). Then the ring of Laurent polynomials $\langle F[x_1,x_1^{-1},\ldots, x_k,x_k^{-1}];+,\cdot,0,1\rangle$ is regularly bi-interpretable with $\mathbb N$ (resp., regularly invertibly interpretable in $\mathbb N$), $F[x_1,x_1^{-1},\ldots, x_k,x_k^{-1}]\cong \Gamma(\mathbb N)$.
\end{corollary}

\begin{corollary}\label{co:laurent_multi_elem_equiv_N}
Let $F$ be a field regularly invertibly interpretable in $\mathbb N$. Let $\Gamma$ be the code of a regular invertible interpretation of $F[x_1,x_1^{-1},\ldots, x_k,x_k^{-1}]$ in $\mathbb N$, $F[x_1,x_1^{-1},\ldots, x_k,x_k^{-1}]\cong \Gamma(\mathbb N)$. Then every ring elementarily equivalent to $F[x,x^{-1}]$ has the form $\nonstandardmodel{F[x_1,x_1^{-1},\ldots, x_k,x_k^{-1}]}{\nstdN}\cong \Gamma(\nstdM)$, where $\nstdM\equiv \mathbb N$.
\end{corollary}

As usual, 
$\nonstandardmodel{F[x_1,x_1^{-1},\ldots,x_k,x_k^{-1}]}{\mathcal M}$ denotes $\Gamma_0(\mathcal M)$, where $\Gamma_0$ is the code of a regular invertible interpretation of $F[x_1,x_1^{-1},\ldots, x_k,x_k^{-1}]$ in $\mathbb S(F, \mathbb N)$, and $\mathcal M\equiv\mathbb S(F,\mathbb N)$.
Respectively, $\nonstandardmodel{F[x_1,x_1^{-1},\ldots,x_k,x_k^{-1}]}{\nstdN}$ denotes $\Gamma(\nstdM)$, where $\Gamma$ is the code of a regular invertible interpretation of $F[x_1,x_1^{-1},\ldots, x_k,x_k^{-1}]$ in $\mathbb N$.
In either case, we call the resulting ring the ring of nonstandard Laurent polynomials in variables $x_1,\ldots,x_k$.
Inspecting the interpretation of $F[x_1,x_1^{-1},\ldots,x_k,x_k^{-1}]$ in $\mathbb S(F,\mathbb N)$, similarly to Section~\ref{se:multi_var}, we obtain that the elements of $\nonstandardmodel{F[x_1,x_1^{-1},\ldots,x_k,x_k^{-1}]}{\mathcal M}$, which we naturally call \emph{nonstandard Laurent polynomials}, have the form directly analogous to~\eqref{eq:nonstd_multivar}:
\begin{equation}\label{eq:laurent_nonstd_multivar}
\begin{split}
p=&\sum_{i=1}^m a_i M_i=\sum_{i=1}^n a_i x_1^{n_{i1}}\cdots x_k^{n_{ik}},\\
&(a_1,\ldots,a_m)\in \mathcal M,\ (n_{11},\ldots,n_{mk})\in \mathbb S(\widetilde{\mathbb Z},\nstdM),
\end{split}
\end{equation}
where $\widetilde{\mathbb Z}=\nstdM-\nstdM$, the summation is understood in the sense of 
Lemma~\ref{le:arb_summation}, and, in the case of a field interpretable in $\mathbb N$, $\mathcal M\cong \mathbb S(\widetilde{F},\nstdM)$.

\begin{theorem}\label{th:universal_nstd_laurent_multi}
Let $F$ be an infinite field, $F[x_1,x_1^{-1},\ldots, x_k,x_k^{-1}]$ the ring of Laurent polynomials, and $\mathbb A$ an arbitrary commutative $F$-algebra with a regular interpretaion in $\mathbb S(F, \mathbb N)$. Let $\mathcal M\equiv\mathbb S(F, \mathbb N)$ and $\widetilde{\mathbb A}=\nonstandardmodel{\mathbb A}{\mathcal M}$. Then every mapping $x_i\mapsto a_i\in \widetilde{\mathbb A}$, where $a_i$ is invertible in $\mathbb A$, $i=1,\ldots,k$, extends uniquely to an algebra homomorphism $\nonstandardmodel{F[x_1,x_1^{-1},\ldots, x_k,x_k^{-1}]}{\mathcal M}\to\widetilde{\mathbb A}$.
\end{theorem}
We omit the proof since it essentially repeats that of Theorem~\ref{th:universal_nstd_single}.

\section{Nonstandard polynomials and Laurent polynomials over integers}\label{se:over_Z}
In this section we describe bi-interpretations of polynomials and Laurent polynomials with integers coefficients with $\mathbb Z$, which allows to treat their nonstandard models the same way we did polynomials and Laurent polynomials over a field. Note that these bi-interpretability results are particular cases of a more general statement recently shown in~\cite{Aschenbrenner-etal:2020}.

\subsection{Polynomials over integers}\label{se:Z_poly}

\begin{theorem}\label{pr:zx_biint}
 The ring $\Z[x]$ is regularly bi-interpretable with $\Z$.
\end{theorem}
\begin{proof}
We show first that $\Z$ is regularly definable in $\Z[x]$. Let $s \in \Z[x]$ and denote by $\langle s\rangle$ the ideal generated by $s$ in $\Z[x]$.
Then $\Z[x]/\langle s\rangle $ is isomorphic to $\Z$ if and only if $s = x +c$ or $p = -x +c$ for some $c \in \Z$. The ring $\Z$ is QFA~\cite{Nies:2007}, hence there is a formula $\psi(s)$ such that 
\[\Z[x]/\langle s\rangle \models \psi(s) \Longleftrightarrow \Z[x]/\langle s\rangle \cong \mathbb Z.
\]
Obviously, the ring $\Z[x]/\langle s\rangle $ is interpretable in $\Z[x]$ with the parameter $s$, so there is a formula $\psi^*(s)$ such that 
\[
\Z[x] \models \psi^*(s) \Longleftrightarrow \Z[x]/\langle s\rangle \models \psi(s).
\]
Hence the set of polynomials $S = \{\pm x + c \mid c \in \Z\}$
is definable in $\Z[x]$ as $S = \{s\in \Z[x] \mid \Z[x] \models \psi^*(s) \}$. Observe that polynomials $s,t \in S$ have the same leading coefficient if either $\frac12(s+t) \in S$ or $\frac12(s+t+1) \in S$. 
Denote by $S_0$ the set of polynomial in $S$ with the same leading coefficient as $s$. It follows that the set $S_0$ as well as the set $\Z = \{s-t\mid s,t \in S\}$ are definable in $\Z[x]$ with parameter $s$ provided $s \in S$. This shows the regular definability of $\Z$ in $\Z[x]$, which we denote by $\Z \cong \Gamma(\Z[x], \psi^*)$.

Interpretation of $\Z[x]$ in $\Z$ is organized via the numbering of tuples, i.e., the polynomial $p=a_nx^n + \ldots a_1x + a_0 \in \Z[x]$ corresponds to a tuple $(a_n, \ldots,a_0)$, which corresponds to its code $n_p \in \Z$. This gives an absolute interpretation of $\Z[x]$ in $\Z$, say $\Z[x] \cong \Delta(\Z)$.

To verify bi-interpretaion, now we construct a formula $\theta(p,n_p)$ that defines in $\Z[x]$ an isomorphism $p\circ s^{-1} \to n_p$ of $\Z[x]$ and $\Delta(\Z)$, where the latter copy of $\mathbb Z$ is interpreted via $\Gamma$ with parameter $s$ as definable subset of $\mathbb Z[x]$.

Let $n \in \Z \subseteq \Z[x]$ then $n = n_q$ for some polynomial $q \in \Z[x]$. Observe that there is a formula $\phi(n,\alpha, \beta)$, where $n,\alpha, \beta \in \Z$ such that for $n= n_q$
\[
\Z \models \phi(n,\alpha, \beta) \Longleftrightarrow q(\alpha)=\beta.
\] 
Consider the relation on pairs $(q,n)\in\mathbb Z[x]\times\mathbb Z$ given by the formula
\begin{equation}\label{eq:isomorphism}
\forall \alpha,\beta\ \phi(n,\alpha,\beta)\leftrightarrow s-\alpha\mid q-\beta,
\end{equation}
where $s$ is the parameter of interpretation $\Gamma$.
This relation is satisfied when and only when $n=n_p$, where $p=a_nx^n+\ldots+a_1x+a_0$ and $q=a_ns^n+\ldots+a_1s+a_0$. This defines an isomorphism between $\mathbb Z[x]$ and $\Delta(\Gamma(\mathbb Z[x],\psi^*))$.

The other isomorphism $\Z \to \Gamma(\Delta(\Z))$ is clearly definable in $\Z$ as the graph of codes of tuples of length $1$.
\end{proof}

Similarly, to prove the same for the ring $\Z[x_1, \ldots,x_k]$, one only needs to take $k$ polynomials $s_1, \ldots,s_k$ such that $\Z[x_1, \ldots,x_n]/\langle s_1, \ldots,s_k \rangle \cong \Z$. We call such $k$-tuples $s_1,\ldots,s_k$ \emph{$\mathbb Z$-maximal} in $\mathbb Z[x_1,\ldots,x_k]$.

Since $\mathbb Z$ is QFA, there is a formula $\psi_k(s_1,\ldots,s_k)$ such that
\[
\mathbb Z[x_1,\ldots,x_k]\models \psi_k(s_1,\ldots,s_k)\iff \mathbb Z[x_1,\ldots,x_k]/\langle s_1,\ldots,s_k\rangle\cong \mathbb Z.
\]
Denote $\chi_i(s_1,\ldots,s_k)$, where $i=1,\ldots,k$, to be the formula
\[\begin{split}
\forall y\ &\psi_k(s_1,\ldots,s_i+y,\ldots,s_k)\to\\
&\psi_k(s_1,\ldots,s_i+y+1,\ldots,s_k)\wedge \psi_k(s_1,\ldots,s_i+y-1,\ldots,s_k).
\end{split}
\]
Now let $\Psi_k(s_1,\ldots,s_k)$ be the formula
\[
\psi_k(s_1,\ldots,s_k)\wedge\bigwedge_{i=1}^k \chi_i(s_1,\ldots,s_k).
\]
\begin{lemma}\label{le:glkZ}
Every tuple $(\sum_{j=1}^k a_{ij}x_j)_{i=1}^k$, where $(a_{ij})\in GL_k(\mathbb Z)$, satisfies $\Psi_k$.
\end{lemma}
\begin{proof}
Since every such tuple is an automorphic image of $x_1,\ldots,x_k$, it suffices to show the statement for the latter tuple. For that, it is enough to show that $\chi_1$ is satisfied.

Note that $s_1,\ldots,s_k$ is $\mathbb Z$-maximal if and only if $s_1-\sum_{i=2}^k t_is_i,s_2,\ldots,s_k$ is $\mathbb Z$-maximal, where $t_i\in \mathbb Z[x_1,\ldots,x_k]$. Using this observation to remove all monomials that involve variables other than $x_1$ from $y$, we see that $x_1+y,\ldots,x_n$ is $\mathbb Z$-maximal if and only if $x_1+y=\pm x_1+c+t_2x_2+\ldots+t_kx_k$, where $c\in\mathbb Z$ and $t_i\in \mathbb Z[x_1,\ldots,x_k]$. The latter condition does not change by adding $\pm 1$ to $y$.
\end{proof}

\begin{theorem}\label{th:z_polynomials_biint}
The ring $\mathbb Z[x_1,\ldots,x_k]$, $k\ge 1$, is regularly bi-interpretable with $\mathbb Z$.
\end{theorem}
\begin{proof} We begin by defining $\mathbb Z$ as a subset of the polynomial ring.

\noindent\textsc{Claim.} The subset $\mathbb Z\subseteq \mathbb Z[x_1,\ldots,x_k]$ is given by the formula
\[
\varphi(q)=\forall s_1,\ldots,s_k\ \Psi_k(s_1,\ldots,s_k)\to \psi_k(s_1+q,\ldots,s_k),
\]
where $\psi_k$ and $\Psi_k$ are as above.

Indeed, if $q\in \mathbb Z$, then $\mathbb Z[x_1,\ldots,x_k]\models \varphi(q)$ by repeated application of $\chi_1$, starting with $y=0$.
Now, let $q\notin \mathbb Z$. Shifting variables by constants, we may assume that $q$ has a nonzero linear part $b_1x_1+\ldots+b_kx_k$, $b_i\in\mathbb Z$. By the choice of $\Psi_k$,
we may assume $q$ has a zero free term. Next, we restrict our consideration to $s_1,\ldots, s_k$ with zero free terms. Denote the linear part of each $s_i$ by $a_{i1}x_1+\ldots+a_{ik}x_k$. Then a necessary condition for $s_1,\ldots,s_k$ to be $\mathbb Z$-maximal is that the matrix $(a_{ij})$ belongs to $GL_k(\mathbb Z)$. However, it easy to see that for every nonzero integer vector $(b_1,\ldots,b_k)$ there is a matrix $A\in GL_k(\mathbb Z)$ s.t. adding $(b_1,\ldots,b_k)$ to the first row of $A$ takes it out of $GL_n(\mathbb Z)$. By Lemma~\ref{le:glkZ}, the corresponding to $A$ choice of linear polynomials $s_1,\ldots,s_k$ breaks $\varphi(q)$. This completes the proof of the claim.

From this point, the proof proceeds similarly to the proof of Theorem~\ref{pr:zx_biint}. Namely, a polynomial $p$ is interpreted by a tuple like in formula~\eqref{eq:multivar_interpretation}, which, in turn, is encoded by an integer $n_p$; the divisibility condition in formula~\eqref{eq:isomorphism} is replaced with membership in the respective ideal.
\end{proof}

\subsection{Laurent polynomials over integers}\label{se:Z_Laurent}
Now we organize a regular invertible interpretation of $\mathbb Z[x_1,x_1^{-1},\ldots,x_k,x_k^{-1}]$ in $\mathbb Z$ (in fact, a regular bi-interpretation). Our first steps are very similar to those in Section~\ref{se:laurent_single}. We notice that the definable set of invertible elements of $\mathbb Z[x_1,x_1^{-1},\ldots,x_k,x_k^{-1}]$ is $\Inv=\{\pm x_1^{\ell_1}\cdots x_k^{\ell_k}\mid \ell_i\in\mathbb Z\}$. We define the set of irreducible Laurent polynomials over $\mathbb Z$ as
\[
\Irr=\{p\mid p\notin \Inv\wedge \forall q\forall r\ p=qr\to [q\in\Inv\vee r\in\Inv]\}.
\]

\begin{lemma}\label{le:p_to_m}
There is a formula $\varphi_{\mathrm{power}}(p,b)$ that defines the set $\{b^m:m\in\mathbb N\}$ in $\mathbb Z[x_1,x_1^{-1},\ldots,x_k,x_k^{-1}]$, $k\ge 1$, whenever $b\in \mathbb Z[x_1,x_1^{-1},\ldots,x_k,x_k^{-1}]$ is irreducible such that $b-1$ does not divide any Laurent polynomial of the form $a-1$ for invertible $a\neq 1$.
\end{lemma}
\begin{proof} Let $\Inv_1=\Inv\setminus\{1\}$. 
We note that for an irreducible $b$, the set $\{ab^m:a\in\Inv, m\in\mathbb N\}$ is given by the formula
\[
\varphi_0(p,b)=b\in \Irr\wedge (\forall a\in\Inv_1\ b-1\nmid a-1) \wedge \forall q\ q\mid p\to (q\in\mathrm{Inv}\vee b\mid q).
\]
Then the set $\{b^m:m\in\mathbb N\}$ is defined by
\[
\varphi_{\mathrm{power}}(p,b)=\varphi_0(p,b)\wedge b-1\mid p-1.
\]
Indeed, if $b-1\mid ab^m-1$ for $a\in\Inv$, then
\[
b-1\mid a b^m-1= ab^m-a + a-1=a(b^m-1)+a-1,
\]
so $b-1\mid a-1$, which by the assumption on $b$ implies $a=1$.
\end{proof}
Since invertible polynomials have the form $a=\pm x_1^{\ell_1}\cdots x_k^{\ell_k}$, where $\ell_i\in\mathbb Z$, such $b$ exist, for example $b=x_1+3$.

Now the observation $(f+1)^m=1+mf+f^2f'$ allows to prove the following proposition.
\begin{proposition}\label{pr:Z_def_in_Laurent}
The subset $\mathbb Z$ is definable in $\mathbb Z[x_1,x_1^{-1},\ldots,x_k,x_k^{-1}]$, $k\ge 1$.
\end{proposition}
\begin{proof}

For any $b$ that satisfies the conditions of Lemma~\ref{le:p_to_m}, we consider the formula
\[
\varphi_1(p,b,q)=\varphi_{\mathrm{power}}(p,b)\wedge (b-1)^2\mid p-1-q(b-1).
\]
By Lemma~\ref{le:p_to_m}, we have $p=((b-1)+1)^m=1+m(b-1)+(b-1)^2p'$, so this formula defines the set of Laurent polynomials $q$ of the form $q=m+(b-1)q'$. Then the formula
\[
\begin{split}
\varphi_{\mathbb N}(q)&=\forall b\ [b\in \mathrm{Irr}\wedge (\forall a\in\Inv_1\ b-1\nmid a-1)]\to \exists p\ \varphi_1(p,b,q)
\end{split}
\]
says that $q$ is of the form $q=m_b+(b-1)q'_b$ for each $b$ that satisfies conditions of Lemma~\ref{le:p_to_m}. Show that if $q$ satisfies $\varphi_\mathbb N$, then it is constant.

Suppose that $q$ is not constant and satisfies $\varphi_\mathbb N$. Without loss of generality, assume it non-trivially involves the variable $x_1$.
Consider the irreducible (by Eisenstein's criterion) integer polynomials $b=x_1^{n}+3$, where $n\in\mathbb N$.
Since $3-1>1$, it follows that such $b$ satisfy the conditions of Lemma~\ref{le:p_to_m}.
For a Laurent polynomial~$r$, let $\overline{\deg}\,r$ (respectively, $\underline{\deg}\, r$) denote the highest (respectively, the lowest) degree in $x_1$ of the monomials nontrivially involved in $r$. Then we see that
\[
\overline{\deg}\,(b-1)q'- \underline{\deg}\, (b-1)q'\ge n
\]
whenever $q'$ is nonzero.
Therefore, for every nonzero $q'$ and every integer $m$, at least one of the following takes place for the Laurent polynomial $m+(b-1)q'$:
\[
\begin{split}
\overline{\deg}\,(m+(b-1)q')- \underline{\deg}\, (m+(b-1)q')\ge n,\\
\overline{\deg}\,(m+(b-1)q')\ge n,\\
-\underline{\deg}\, (m+(b-1)q')\ge n.
\end{split}
\]
Taking $b$ as above with $n>\max\{\overline{\deg}\,q- \underline{\deg}\, q,\overline{\deg}\,q, -\underline{\deg}\,q\}$, we obtain that $q'_b$ must be $0$ and therefore $q=m_b\in\mathbb N$.

Since each $m\in\mathbb N$ clearly satisfies $\varphi_\mathbb N$ with $p=b^m$, the formula $\varphi_\mathbb N(q)$ defines $\mathbb N\subseteq \mathbb Z[x_1,x_1^{-1},\ldots,x_k,x_k^{-1}]$, which suffices.

\end{proof}

At this point, regular bi-interpretability with $\mathbb Z$ is established similarly to Theorems~\ref{pr:zx_biint},~\ref{th:z_polynomials_biint}. We omit the details of the proof.

\begin{theorem}\label{th:laurent_z_biint}
The ring $\mathbb Z[x_1,x_1^{-1},\ldots,x_k,x_k^{-1}]$, $k\ge 1$, is regularly bi-interpretable with $\mathbb Z$.
\end{theorem}

Theorems~\ref{th:z_polynomials_biint} and~\ref{th:laurent_z_biint} allows to describe nonstandard polynomials and Laurent polynomials in a fashion directly analogous to our description in Sections~\ref{se:polynomials} and~\ref{se:laurent}. If $R$ is $\mathbb Z[x_1,\ldots,x_k]$ or $\mathbb Z[x_1,x_1^{-1},\ldots,x_k,x_k^{-1}]$ and $R\cong \Gamma(\mathbb Z)$ is the respective interpretation in $\mathbb Z$, then the ring $\nonstandardmodel{R}{\nstdM}\cong\Gamma(\widetilde{\mathbb Z})$ is the the nonstandard (Laurent) polynomial ring. By Theorem~\ref{th:complete-scheme} we get the following result.

\begin{theorem}\label{th:Zpoly_desc}
Let $R$ be $\mathbb Z[x_1,\ldots,x_k]$ or $\mathbb Z[x_1,x_1^{-1},\ldots,x_k,x_k^{-1}]$. Then every ring elementarily equivalent to $R$ has the form $\nonstandardmodel{R}{\nstdM}$.
\end{theorem}

Elements of $\nonstandardmodel{R}{\nstdM}$, that is, nonstandard (Laurent) polynomials $p(x_1,\ldots,x_k)$, are associated to tuples
\begin{equation}\label{eq:z_multivar_interpretation}
(a_1,n_{11},\ldots,n_{1k},\ a_2,n_{21},\ldots,n_{2k},\ \ldots,\ a_m,n_{m1},\ldots,n_{mk}),
\end{equation}
with $m\in\nstdM, a_i\in \widetilde{\mathbb Z}, n_{ij}\in\nstdM$, and, in the case of polynomials, $n_{ij}\ge 0$. Summation and multiplication of nonstandard (Laurent) polynomials is given by the formulas exhibited in the Sections~\ref{se:polynomials},~\ref{se:laurent}. By the same straightforward check that we performed in those sections, $p(x_1,\ldots,x_k)$ can be written as
\[
p(x_1,\ldots,x_k)=\sum_{i=1}^m a_ix_1^{n_{i1}}\cdots x_k^{n_{ik}},
\]
where $m\in\nstdM$, $a_i\in\widetilde{\mathbb Z}$, $n_{ij}\in\widetilde{\mathbb Z}$ and, in the case of polynomials, $n_{ij}\ge 0$, and the summation is understood in the sense of Lemmas~\ref{le:summation} or~\ref{le:arb_summation}.

\bibliographystyle{plain}
\bibliography{biblio}

\end{document}